\documentclass[11pt, a4paper]{amsart}
\usepackage[utf8]{inputenc}
\usepackage{scalerel}[2016/12/29]
\usepackage[mathscr]{euscript}
\usepackage{amssymb,amsmath,amsthm}
\usepackage{geometry}
\usepackage{tikz}
\usetikzlibrary{cd, intersections, calc}

\usepackage{varwidth}

\usepackage{enumitem}
\usepackage{ulem}

\usepackage{amscd}
\tikzset{ext/.style={circle, draw,inner sep=1pt},int/.style={circle,draw,fill,inner sep=1pt},nil/.style={inner sep=1pt}}
\usetikzlibrary{braids}
\usepackage[all]{xy}
\definecolor{BleuTresFonce}{rgb}{0.215, 0.215, 0.36}
\definecolor{EgyptianBlue}{rgb}{0.06, 0.2, 0.65}
\definecolor{darkspringgreen}{rgb}{0.09, 0.45, 0.27}
\definecolor{dartmouthgreen}{rgb}{0.05, 0.5, 0.06}

\usepackage[colorlinks,final,hyperindex]{hyperref}

\hypersetup{citecolor=dartmouthgreen, urlcolor=blue, linkcolor=darkspringgreen}
\usepackage{graphicx,float}
\usepackage[font=small, labelfont=bf, margin=10pt]{caption}
\usepackage{csquotes}
\usepackage{bbold}
\usepackage{import}
\usepackage{xifthen}
\usepackage{pdfpages}
\usepackage{transparent}

\numberwithin{equation}{section}

\newtheorem*{conjecture*}{Conjecture}
\newtheorem{conjecture}[equation]{Conjecture}
\newtheorem{theorem}[equation]{Theorem}
\newtheorem*{theorem*}{Theorem}
\newtheorem{definition}[equation]{Definition}
\newtheorem{lemma}[equation]{Lemma}
\newtheorem{proposition}[equation]{Proposition}
\newtheorem{prop-def}[equation]{Proposition-Definition}
\newtheorem{corollary}[equation]{Corollary}
\newtheorem{remark}[equation]{Remark}
\theoremstyle{definition}
\newtheorem{example}[equation]{Example}
\newtheorem*{notation*}{Notations}
\newtheorem*{example*}{Example}
\newtheorem*{question*}{Question}
\newtheorem*{claim*}{Claim}

\usepackage{color}
\newcommand\cred[1]{{\color{red} #1}}
\newcommand\cbl[1]{{\color{blue} #1}}

\newcommand{\Gr}{\operatorname{\Gamma}}
\newcommand{\parti}{\operatorname{\Pi_{\mathsf{gr}}}}
\newcommand{\Path}{\operatorname{\mathsf{P}}}
\newcommand{\K}{\operatorname{\mathsf{K}}}
\newcommand{\Cyc}{\operatorname{\mathsf{C}}}
\newcommand{\St}{\operatorname{\mathsf{St}}}
\DeclareMathOperator{\Ver}{Vert}
\DeclareMathOperator{\Edge}{Edge}
\DeclareMathOperator{\Leav}{Leav}
\DeclareMathOperator{\In}{In}
\DeclareMathOperator{\Id}{Id}

\newcommand{\Aut}{\operatorname{\mathsf{Aut}}}

\newcommand{\Top}{\operatorname{\mathsf{Top}}}

\newcommand{\Real}{\operatorname{\mathbb{R}}}
\newcommand{\Cmplx}{\operatorname{\mathbb{C}}}
\DeclareMathOperator{\Aff}{Aff}
\DeclareMathOperator{\Dil}{Dil}
\DeclareMathOperator{\Bl}{Bl}
\DeclareMathOperator{\PGL}{PGL}
\DeclareMathOperator{\SO}{SO}
\DeclareMathOperator{\Spec}{Spec}
\newcommand{\Fr}{\operatorname{\mathrm{Fr}}}

\newcommand{\M}{\operatorname{\mathcal{M}}}
\newcommand{\Ta}{\operatorname{\mathcal{T}}}
\newcommand{\bC}
{\operatorname{\overline{\mathcal{C}}}}
\newcommand{\bM}{\operatorname{\overline{\mathcal{M}}}}
\newcommand{\beM}{\overline{\EuScript{M}}}

\newcommand{\eM}{\EuScript{M}}
\newcommand{\frM}{\operatorname{\widetilde{\mathcal{M}}^{\mathrm{fr}}}}
\newcommand{\fr}{\operatorname{\mathrm{fr}}}
\newcommand{\frakFM}{\operatorname{\mathfrak{FM}}}
\newcommand{\ufrM}{\operatorname{\widetilde{\mathcal{M}}}}
\newcommand{\Pro}{\operatorname{\mathbb{P}}}
\newcommand{\Sph}{\operatorname{\mathbb{S}}}
\newcommand{\X}{\operatorname{\mathsf{X}}}
\newcommand{\KN}{\operatorname{\mathsf{KN}}}
\DeclareMathOperator{\Conf}{Conf}
\DeclareMathOperator{\NConf}{NConf}

\newcommand{\Disc}{\operatorname{\mathbb{D}}}

\newcommand{\fD}{\operatorname{\mathrm{f}\mathcal{D}}}
\newcommand{\D}{\operatorname{\mathcal{D}}}
\newcommand{\FM}{\operatorname{\mathsf{FM}}}
\newcommand{\fFM}{\operatorname{\mathrm{f}\mathsf{FM}}}
\newcommand{\PB}{\operatorname{\mathsf{PB}}}

\newcommand{\La}{\operatorname{\mathcal{L}}}

\newcommand{\G}{\operatorname{\mathcal{G}}}
\newcommand{\N}{\operatorname{\mathcal{N}}}
\newcommand{\B}{\operatorname{\mathcal{B}}}
\newcommand{\A}{\operatorname{\mathcal{A}}}
\newcommand{\Orb}{\operatorname{\mathcal{O}}}
\newcommand{\F}{\operatorname{\mathcal{F}}}
\newcommand{\Q}{\operatorname{\mathcal{Q}}}

\newcommand{\E}{\operatorname{\mathcal{E}}}
\newcommand{\R}{\operatorname{\mathcal{R}}}
\newcommand{\Ho}{\operatorname{\mathcal{H}}}
\newcommand{\T}{\operatorname{\mathbb{T}}}
\newcommand{\Pop}{\operatorname{\mathcal{P}}}
\newcommand{\C}{\operatorname{\mathcal{C}}}

\newcommand{\cokoszul}{\operatorname{\text{!`}}}

\usepackage[OT2,OT1]{fontenc}
\newcommand{\Com}{\operatorname{\mathsf{gcCom}}}

\newcommand{\Susp}{\operatorname{\mathcal{S}}}
\newcommand{\Lie}{\operatorname{\mathsf{gcLie}}}
\newcommand{\Ass}{\operatorname{\mathsf{gcAss}}}
\newcommand{\OS}{\operatorname{\mathsf{OS}}}
\newcommand{\Gerst}{\operatorname{\mathsf{gcGerst}}}
\newcommand{\BV}{\operatorname{\mathsf{gcBV}}}

\newcommand{\Tree}{\operatorname{\mathsf{Tree}}}

\newcommand{\AD}{\operatorname{\mathsf{AD}}}

\newcommand{\Graphs}{\operatorname{\mathsf{Graphs}}}

\title{Graphical configuration spaces, Contractads and Formality}
\author[A.\,Khoroshkin]{Anton Khoroshkin}
\address{Anton Khoroshkin: \newline
Department of Mathematics, University of Haifa, Mount Carmel, 3103301, Haifa, Israel
}
\email{khoroshkin@gmail.com}

\author[D.\,Lyskov]{Denis Lyskov}
\address{Denis Lyskov: \newline
Université Paris Cité, IMJ-PRG, 75013, Paris,
France
\newline
Higher School of Economics, Faculty of Mathematics, 119048, Moscow, Russia}
\email{ddlyskov@gmail.com}

\date{}

\begin{document}
\begin{abstract}
Given a finite simple connected graph $\Gamma$, the \emph{graphical configuration space} $\mathrm{Conf}_{\Gamma}(X)$ is the space of collections of points in $X$ indexed by the vertices of $\Gamma$, where points corresponding to adjacent vertices must be distinct. 
When $X=\mathbb{R}^d$
and the points are replaced by small disks, the resulting spaces for all possible graphs fit together into an algebraic structure that extends the little disks operad, called \emph{the little disks contractad} $\mathcal{D}_d$.

In this paper, we investigate the homotopical and algebraic properties of the little disks contractad \(\mathcal{D}_d\). We construct and study Fulton--MacPherson compactifications of graphical configuration spaces, which provide a convenient model for \(\mathcal{D}_d\) within the class of compact manifolds with boundary. Using these and wonderful compactifications, we prove that \(\mathcal{D}_d\) is formal in the category of (Hopf) contractads for \(d=1\), \(d=2\), and for chordal graphs for any \(d\). We also identify the first obstructions to coformality in the case of cyclic graphs. In addition, we give a combinatorial description of the cell structure of \(\mathcal{D}_2\) and present applications to the study of graphical configuration spaces \(\mathrm{Conf}_{\Gamma}(X)\) using the language of twisted algebras.
\end{abstract}
\maketitle
\setcounter{section}{-1}
\tableofcontents

\section{Introduction}

\renewcommand{\theequation}{\Alph{equation}}

\subsection{Motivation: from operads to contractads}

The little disks operad $\mathcal{E}_d$ is one of the classical bridges between topology and algebra. 
Its space of $n$-ary operations consists of ``rectilinear'' embeddings of $n$ disjoint unit disks into a single unit disk.
Originally introduced to describe iterated loop spaces (\cite{boardman1968homotopy}, \cite{may2006geometry}, see also~\cite{segal1973configuration}), it later became central in many areas of mathematics. For instance, algebraic models of the little disks operad $\mathcal{E}_2$ and its formality play a key role in Tamarkin's proof of Kontsevich’s formality theorem for Poisson manifolds (\cite{tamarkin1998another}). 
Since then, numerous proofs of the formality of the little disks operad $\mathcal{E}_d$ have been established (\cite{kontsevich1999operads}, \cite{lambrechts2014formality}, \cite{mcclure2003multivariable}, \cite{vsevera2011equivalence}, \cite{fresse2017homotopy}, \cite{vaintrob2021formality}), along with various compatible cell structures. 
The little disks operad has also been instrumental in the geometry of configuration spaces and embedding calculus (see, e.g.,~\cite{campos2019configuration}, \cite{idrissi2019lambrechts}, \cite{turchin2018relative}). 

The notion of a \emph{contractad} was introduced in~\cite{lyskov2023contractads}. 
The term \emph{contractad} combines \emph{``contraction''} and \emph{``operad''}, reflecting the fact that it generalizes operads by replacing composition with graph contractions. 
Here, the indexing set of \emph{operations} is the collection of finite simple connected graphs, while \emph{compositions} are determined by contracting connected subgraphs (see \S\ref{sec:sub:contractads}).
This framework naturally generalizes the little disks operad by relaxing the intersection condition for subdisks.
Specifically, $\D_d(\Gamma)$ denotes the space of embeddings of small disks indexed by the vertices of a connected graph $\Gamma$, with the requirement that disks corresponding to adjacent vertices do not intersect.  
The corresponding generalized configuration space $\Conf_{\Gamma}(X)$ consists of points in $X$ indexed by $V_{\Gamma}$, with the condition that points corresponding to adjacent vertices are distinct.
We call the generalized configuration space $\Conf_{\Gamma}(X)$ \emph{graphical}, as it assembles into a graphical collection.
These natural objects have been extensively studied and appear in the literature under various names:
graph or graphic~\cite{greene1983interpretation, baranovsky2012graph, orlik1980combinatorics, eastwood2007euler},
partial~\cite{berceanu2017geometry}, and chromatic configuration spaces~\cite{zakharov2022rational, kallel2024configuration}.

In this paper, we extend several classical results from the little disks operad $\mathcal{E}_d$ to the little disks contractad $\D_d$.
In particular, we construct Fulton–MacPherson compactifications for all $d$, and in dimensions $d=1$ and $d=2$ we prove the formality of $\D_d$ and describe its compatible cell structure. 
We also analyze the case of long cycles, showing the existence of nontrivial Massey products, which illustrate the intrinsic complexity of the little disks contractad.
We emphasize that the formality problem for higher dimensions ($d>2$), as well as the description of the rational homotopy type of $\mathsf{Conf}_{\Gamma}(\mathbb{R}^d)$ for general $\Gamma$ and $d \geq 2$, remain open questions.

\subsection{Outline of the main results}

Our first main result builds on ideas from~\cite{petersen2016poincare}, showing that the language of twisted Lie algebras is well suited for studying the cohomology of configuration spaces on reasonably "good" topological spaces.
We extend the notion of a twisted algebra over an operad to the setting of contractads and establish the following

\begin{theorem}{\rm (Theorem~\ref{thm::cohomology_twisted_lie})} 
Let $X$ be a paracompact and locally compact Hausdorff space and let $\A^c_X$ be a rational cdga model for the compactly supported cochains on $X$. Then there is an isomorphism of twisted $\Com$-coalgebras
\[
H_{CE}^{\bullet}(\underline{\A^c_X}\underset{\mathrm{H}}{\otimes} \mathsf{S}\Lie)\cong H_c^{\bullet}(\Conf(X);\mathbb{Q}),
\] 
where $\Lie$ is the Lie contractad considered as a free twisted $\Lie$-algebra and $\mathsf{S}\Lie$ is its twisted suspension $\mathsf{S}\Lie(\Gr)\cong \mathrm{Res}^{\Sigma_{V_{\Gr}}}_{\Aut(\Gr)} \mathrm{sgn} \otimes \Lie(\Gr)[-|V_{\Gr}|]$. 
\end{theorem} As an immediate consequence of this Theorem, we have that, for a compact Hausdorff space $X$, its rational compactly supported cohomology of graphical configuration spaces $H_c^{\bullet}(\Conf_{\Gr}(X);\mathbb{Q})$ are rational homotopy invariants.

Next, we generalize the Fulton–MacPherson compactification to graphical configuration spaces and provide a model of the (framed) little disks contractad $\D_d$ (respectively, $\fD_d$) in the category of compact topological spaces with corners (Theorems~\ref{thm::FMn_is_En} and~\ref{thm::FM:M}, Proposition~\ref{prop::parallisable}) .

Lower dimensions $d=1$ and $d=2$ exhibit special features.  
For $d=1$, we identify the spaces of operations with poset associahedra introduced by Galashin~\cite{galashin2021p}:
\begin{theorem}{\rm (Theorem~\ref{thm::D1=As1})}
For a graph $\Gr$, there is a bijection between connected components of $\FM_1(\Gr)$ and the set of acyclic orientations of the graph $\Gr$. Moreover, the connected component indexed by acyclic orientation $\alpha$ is isomorphic to the corresponding Galashin's poset polytope  $\mathcal{K}_{\alpha}(\Gr)$ and the face poset decompositions of these polytopes are compatible with the contractad structure.
\end{theorem}

In dimension $d=2$, we describe in detail the little disks contractad $\D_2$ for a special family of graphs. In particular, whenever the graph $\Gr$ is chordal, the space $\Conf_{\Gr}(\mathbb{C})$ is a $K(\pi,1)$-space whose fundamental group is a generalization of the pure braid group (see~\cite{cohen2021graphic} and \S\ref{sec:sub:chordal}). 
On the other hand, the rational homotopy type of $\Conf_{\Gr}(\mathbb{C})$ might be extremely complicated even for simplest nonchordal graphs such as cycles, as we show in the following
\begin{theorem}
{\rm (Theorem~\ref{thm::cycl::homotopy})}
For a cycle graph $\Cyc_n$ with $n \geq 5$, the Lie algebra of rational homotopy groups of the configuration space $\Conf_{\Cyc_n}(\mathbb{C})$ defined via Whitehead products admits the following presentation:
\begin{equation*}
\pi_{\bullet+1}(\Conf_{\Cyc_n}(\mathbb{C})) \otimes \mathbb{Q} = 
\mathsf{Lie} \left\langle 
\begin{array}{c}
y^0_1, \ldots, y^0_n, \\
u^{n-3}
\end{array}
\left| 
\begin{array}{c}
[y_i, y_j] = 0, \\
\sum_{i=1}^{n} [y_i, u] = 0
\end{array}
\right.
\right\rangle
\end{equation*}

Moreover, the $L_{\infty}$-algebra contains a nontrivial higher Lie bracket
\begin{equation*}
l_{n-1}(y_2, \ldots, y_n) = 
- l_{n-1}(y_1, y_3, \ldots, y_n) = 
\cdots = 
(-1)^n l_{n-1}(y_1, \ldots, y_{n-1}) = u.
\end{equation*}
\end{theorem}

For higher dimensions ($d>2$), we do not yet have a satisfactory 
algebraic model of the little disks contractad $\D_d$.  
Nevertheless, we are able to treat the case of chordal graphs and prove the following:

\begin{theorem}
{\rm (Theorem~\ref{thm:chordal:formal})}
For $d\geq 1$, the little $d$-disks contractad $\D_d$ is formal and coformal over $\mathbb{R}$ in the class of chordal graphs.
\end{theorem}

Moreover, in dimension $d=2$, one can compare two distinct compactifications of Fulton–MacPherson  -- either as a complex or as a real algebraic variety.  
The compactification via complex geometry, also known as the \emph{wonderful compactification} after~\cite{de1995wonderful}, changes the topology but provides additional structure from algebraic geometry.
In particular, we gain a detailed description of the line bundles and their Chern classes (called $\psi$-classes) indexed by vertex points. We explain the interaction of $\psi$-classes with the contractad structure. 
As a corollary, we show that the contractads $\D_2$ are well defined in the category of log-spaces of~\cite{vaintrob2021formality,dupont2024logarithmic}, leading to the following:

\begin{theorem}
\label{thm::D2:formal:intro}
{\rm (Theorem~\ref{thm:D2:log:formal})}
The little disks contractads $\D_2$ and its framed analogue $\fD_2$ are formal over $\mathbb{C}$.
\end{theorem}

Finally, we consider the acyclic direction contractad $\AD_d$ (see~\S\ref{sec:sub:Acyclic:contractad}).  
This is a contractad in the category of posets that generalizes the combinatorial model of the little disks operad in posets introduced in~\cite{berger1997real}. Its components are the posets of acyclic edge-weighted directions of the underlying graphs with a partial order arising from comparing weights of directed edges. We prove that $\AD_1$ coincides with associative contractad $\Ass$, and that $\AD_2$ provides a model for $\D_2$ in the following sense: 

\begin{theorem}{\rm (Corollary~\ref{cor::ADisEn})}
A geometric realisation $|\AD_2|$ of the poset contractad $\AD_2$ is a $E_2$-contractad. In other words, we have a zig-zag of weak homotopy equivalence of topological contractads relating $|\AD_2|$ and $\D_2$
\[
|\AD_2|\overset{\simeq}{\leftarrow}*\overset{\simeq}{\rightarrow}\D_2.
\]
\end{theorem} For $d>2$, a similar statement holds if we restrict ourselves to chordal graphs. For arbitrary graphs, this is an open problem.

\subsection{Structure of the paper}
The paper is organised as follows:

In \S\ref{sec::Ed} we remind the definitions of the main objects: contractads in~\S\ref{sec:sub:contractads}, Lie and commutative contractads in~\S\ref{sec:sub::com::Lie}, the Little Disks contractad (and its framed version) in~\S\ref{sec:sub:E_d} and~\S\ref{sec:sub::fE_d}, respectively.

In~\S\ref{sec::Graphic::Lie} we
define twisted algebras over a contractad and apply them to the cohomology of graphical configuration spaces.  

In Section~\S\ref{sec::compactifications} we describe compactifications of graphical configuration spaces, extending Fulton--MacPherson compactifications to contractads. These constructions clarify the link with the wonderful contractad and yield a topological action of the Little Disks contractad on the union of graphical configuration spaces. In~\S\ref{sec::polytopes} we show that in dimension $1$ the connected components of the sets $\FM_1(\Gamma)$ are isomorphic to the poset associahedra of~\cite{galashin2021p}.

Section~\S\ref{sec::chordal_cycles} studies the rational homotopy type of graphical configuration spaces for two graph families. \S\ref{sec:sub:chordal}-\S\ref{sec:sub:chordal:formality} cover chordal graphs, where the description is parallel to the case of operads, while~\S\ref{sec:sub:cycles} shows the existence of nonzero Massey products for cycles.

In~\S\ref{sec::logformality} we prove the formality of the Little Disks contractad $\D_2$ via logarithmic geometry, generalizing~\cite{vaintrob2021formality, dupont2024logarithmic}. The proof relies on the description of $\psi$-classes of the wonderful contractad given in~\S\ref{sec:sub:psi:classes}.

Section~\S\ref{sec::combinatorial} extends the known cell structure of the Little Disks operad $\mathcal{E}_2$ to the contractad $\D_2$, making it compatible with compositions.

It is worth noting that Sections~\S\ref{sec::Ed} and~\S\ref{sec::compactifications} introduce definitions and constructions that are used throughout the paper. In contrast, Sections~\S\ref{sec::Graphic::Lie}, \S\ref{sec::chordal_cycles}, \S\ref{sec::logformality}, and \S\ref{sec::combinatorial} are largely self-contained, each focusing on different aspects of configuration spaces and the Little Disks contractad. Their interrelations are summarized in the following diagram:
$$
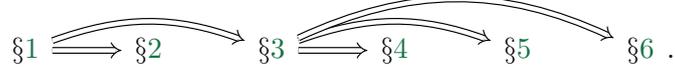
\begin{tikzcd}
   \S\ref{sec::Ed} \arrow[r, Rightarrow] \arrow[rr, bend left=20, Rightarrow] 
   & \S\ref{sec::Graphic::Lie}
   & \S\ref{sec::compactifications} \arrow[r, Rightarrow] \arrow[rr, Rightarrow, bend left=20] \arrow[rrr, Rightarrow, bend left=25]
   & \S\ref{sec::chordal_cycles}
   & \S\ref{sec::logformality} 
   & \S\ref{sec::combinatorial}
\end{tikzcd}.
$$
\textbf{Acknowledgments and Funding.}
The authors are grateful to the anonymous referee for insightful comments regarding the chordal case that helped improve the first draft. \\
The work of the second author was funded within the framework of the HSE University Basic Research Program.

\renewcommand{\theequation}{\thesection.\arabic{equation}}

\section{Little disks contractads}
\label{sec::Ed}
In this section, we recall the basics of contractads and describe the (framed) little disks contractads. 
\subsection{Contractads}
\label{sec:sub:contractads}
In this subsection, we briefly recall the notion of contractads, see~\cite[Sec 1]{lyskov2023contractads}.

Consider the \emph{groupoid of connected graphs} $\mathsf{CGr}$ whose objects are non-empty undirected connected simple\footnote{ by simple we mean that a graph does not contain loops and double edges} graphs $\Gr=(V_{\Gr}, E_{\Gr})$ and whose morphisms are isomorphisms of graphs. We use the following notations for concrete type of graphs:
\begin{itemize}
\label{typesofgraphs}
     \item the path graph $\Path_n$ on the vertex set $\{1,\cdots, n\}$ with edges $\{(i,i+1)| 1\leq i \leq n-1 \}$,  
     \item the cycle graph $\Cyc_n$ on the vertex set $\mathbb{Z}/n\mathbb{Z}$ with edges $\{(i,i+1)| i\in \mathbb{Z}/n\mathbb{Z} \}$, 
     \item the complete graph $\K_n$ on the vertex set $\{1,\cdots, n\}$ and the edges $\{(i,j)|i\neq j\}$,
    \item the stellar graph $\St_n$ on the vertex set $\{0,1,\cdots, n\}$ with edges $\{(0,i)| 1\leq i \leq n \}$. The vertex $"0"$ adjacent to all vertices is called the "core". 
        \item For a partition $\lambda=(\lambda_1\geq \lambda_2\geq\cdots\geq\lambda_k)$, we consider the complete multipartite graph $\K_{\lambda}$.  This graph consists of blocks of vertices of sizes $\lambda_1,\lambda_2,\cdots,\lambda_k$, such that two vertices are adjacent if and only if they belong to different blocks.
\end{itemize}
A \emph{tube} of a graph $\Gr$ is a non-empty subset $G$ of vertices such that the induced subgraph $\Gr|_G$ is connected.  If the tube consists of one vertex, we call it trivial. A \emph{partition of a graph} $\Gr$ is a partition of the vertex set whose blocks are tubes. We denote by $\parti(\Gr)$ the set of partitions of the graph $\Gr$. We use notation $I\vdash \Gr$ for partitions. For a partition $I$ of graph $\Gr$, the contracted graph, denoted $\Gr/I$, is the graph obtained from $\Gr$  by contracting each block of $I$ to a single vertex.
In a contracted graph $\Gr/I$, multiple edges between two vertices are treated as a single edge. Therefore, a contracted graph can also be regarded as a simple graph.

\begin{figure}[ht]
  \centering
  \begin{gather*}
  \vcenter{\hbox{\begin{tikzpicture}[scale=0.7]
    \fill (0,0) circle (2pt);
    \node at (0,0.6) {1};
    \fill (1,0) circle (2pt);
    \node at (1,0.6) {2};
    \fill (2,0) circle (2pt);
    \node at (2,0.6) {3};
    \fill (3,0) circle (2pt);
    \node at (3,0.6) {4};
    \fill (4,0) circle (2pt);
    \node at (4,0.6) {5};
    \draw (0,0)--(1,0)--(2,0)--(3,0)--(4,0);
    \draw[dashed, rounded corners=5pt] (-0.25,-0.25) rectangle ++(1.5,0.5);
    \draw[dashed, rounded corners=5pt] (2.75,-0.25) rectangle ++(1.5,0.5);
    \draw[dashed] (2,0) circle (7pt);
    \end{tikzpicture}}}
    \quad
    \longrightarrow
    \quad
  \vcenter{\hbox{\begin{tikzpicture}[scale=0.7]
    \fill (0,0) circle (2pt);
    \node at (0,0.5) {\{1,2\}};
    \fill (1.5,0) circle (2pt);
    \node at (1.5,0.5) {\{3\}};
    \fill (3,0) circle (2pt);
    \node at (3,0.5) {\{4,5\}};
    \draw (0,0)--(1.5,0)--(3,0);    
    \end{tikzpicture}}}
    \\
    \\
  \vcenter{\hbox{\begin{tikzpicture}[scale=0.7]
    \fill (0,0) circle (2pt);
    \node at (-0.4,-0.3) {1};
    \fill (2,0) circle (2pt);
    \node at (2.4,-0.3) {4};
    \fill (0,2) circle (2pt);
    \node at (-0.4,2.3) {2};
    \fill (1,1) circle (2pt);
    \node at (1,0.5) {5};
    \fill (2,2) circle (2pt);
    \node at (2.4,2.3) {3};
    \draw (0,0)--(2,0)--(2,2)--(0,2)--cycle;
    \draw (0,0)--(1,1)--(2,2);
    \draw (2,0)--(1,1)--(0,2);
    \draw[dashed] (0,0) circle (7pt);
    \draw[dashed] (1,1) circle (7pt);
    \draw[dashed] (2,0) circle (7pt);
    \draw[dashed, rounded corners=5pt] (-0.25,1.75) rectangle ++(2.5,0.5);
    \end{tikzpicture}}}
    \quad
    \longrightarrow
    \vcenter{\hbox{\begin{tikzpicture}[scale=0.75]
    \fill (0,0) circle (2pt);
    \node at (-0.4,-0.3) {\{1\}};
    \fill (1,1) circle (2pt);
    \node at (1,0.5) {\{5\}};
    \fill (1,2) circle (2pt);
    \node at (1,2.4) {\{2,3\}};
    \fill (2,0) circle (2pt);
    \node at (2.4,-0.3) {\{4\}};
    \draw (0,0)--(2,0)--(1,2)-- cycle;
    \draw (0,0)--(1,1)--(1,2);
    \draw (2,0)--(1,1);
    \end{tikzpicture}}}
    \quad
    \quad
    \quad
    \vcenter{\hbox{\begin{tikzpicture}[scale=0.7]
    \fill (-0.63,1.075)  circle (2pt);
    \node at (-0.9,1.4) {1};
    \fill (0.63,1.075)  circle (2pt);
    \draw[dashed] (0.63,1.075) circle  (7pt);
    \node at (0.9,1.4) {2};
    \fill (-1.22,0) circle (2pt);
    \node at (-1.6,0) {6};
    \fill (1.22,0) circle (2pt);
    \node at (1.6,0) {3};
    \draw[dashed] (1.22,0) circle  (7pt);
    \fill (-0.63,-1.075)  circle (2pt);
    \node at (-0.9,-1.4) {5};
    \fill (0.63,-1.075)  circle (2pt);
    \node at (0.9,-1.4) {4};
    \draw[dashed] (0.63,-1.075) circle  (7pt);
    \draw (-1.22,0)--(-0.63,1.075)--(0.63,1.075)--(1.22,0)--(0.63,-1.075)--(-0.63,-1.075)--cycle;
    \draw[dashed] (-1.6,0)[rounded corners=15pt]--(-0.4,1.7)[rounded corners=15pt]--(-0.4,-1.7)[rounded corners=12pt]--cycle;
    \end{tikzpicture}}}
    \quad\longrightarrow\quad
    \vcenter{\hbox{\begin{tikzpicture}[scale=0.7]
     \fill (0.63,1.075)  circle (2pt);
     \node at (0.9,1.5) {\{2\}};
    \fill (-1.22,0) circle (2pt);
    \node at (-2.2,0) {\{1,5,6\}};
    \fill (1.22,0) circle (2pt);
    \node at (1.75,0) {\{3\}};
    \fill (0.63,-1.075)  circle (2pt);
    \node at (0.9,-1.5) {\{4\}};
    \draw (-1.22,0)--(0.63,1.075)--(1.22,0)--(0.63,-1.075)--cycle;
    \end{tikzpicture}}}
  \end{gather*}
  \caption{Examples of partitions of graphs and associated contractions.}
  \label{contrpic}
\end{figure}
A \emph{graphical collection} with values in a symmetric monoidal category $\C$ is a contravariant functor 
\[
\Orb\colon \mathsf{CGr}^{\mathrm{op}}\rightarrow \C.
\]The category of graphical collections forms a monoidal category with respect to \emph{contracted} product
\begin{equation}
\label{eq::contract::product}
    (\Pop \circ \Q)(\Gr) := \bigoplus_{I \vdash \Gr} \Pop(\Gr/I) \otimes \bigotimes_{G \in I} \Q(\Gr|_G),
\end{equation}  
where the sum ranges over all partitions of $\Gr$. This operation is associative, and the graphical collection $\mathbb{1}$ concentrated in the one-vertex graph $\Path_1$, with $\mathbb{1}(\Path_1)=1_{\C}$, is the unit for this operation.

A \emph{contractad} is a monoid in this monoidal category, i.e., it is a graphical collection $\Pop$ with a collection of maps
\[
    \gamma_I^{\Gr}\colon \Pop(\Gr/I)\otimes \bigotimes_{G \in I} \Pop(\Gr|_G)\to \Pop(\Gr),
\] ranging over all graphs and their partitions, and a unit $\eta\colon 1_{\C}\to \Pop(\Path_1)$, satisfying certain coherence relations.

An equivalent way to present contractads is via \emph{infinitesimal compositions} that are analogous to \emph{partial compositions} $\circ_i\colon \Orb(n)\otimes\Orb(m)\to\Orb(n+m-1)$ for operads.  For a tube $G$, we denote by $\Gr/G$ the contraction of $\Gr$ by the partition with the only non-trivial block $G$. When $\Pop$ is a contractad, we define $\circ^{\Gr}_G\colon \Pop(\Gr/G)\otimes \Pop(\Gr|_G)\to \Pop(\Gr)$ by
\[
    \Pop(\Gr/G)\otimes\Pop(\Gr|_G) \cong \Pop(\Gr/G)\otimes\Pop(\Gr|_G)\otimes \bigotimes_{v \not\in G} 1_{\C} \overset{\Id \otimes \eta^{\otimes}}{\hookrightarrow} \Pop(\Gr/G)\otimes\Pop(\Gr|_G)\otimes \bigotimes_{v \not\in G}\Pop(\Gr|_{\{v\}}) \overset{\gamma}{\rightarrow} \Pop(\Gr).
\]
The infinitesimal compositions $\circ^{\Gr}_G$ completely recover the contractad structure.

\subsection{Commutative and Lie contractads, Suspensions}
\label{sec:sub::com::Lie}
The simplest example of a contractad is the commutative contractad $\Com$ in the category of sets. The contractad has the components $\Com(\Gr)=\{*\}$ with the obvious infinitesimal compositions $\circ^{\Gr}_G\colon \Com(\Gr/G)\times \Com(\Gr|_G)\to\Com(\Gr)$ of the form $(*,*)\mapsto*$. The corresponding contractad in the category of vector spaces is generated by a symmetric generator $m$ in component $\Path_2$, satisfying the relations
\begin{gather}
\text{In }\Path_3:\quad m\circ^{\Path_3}_{\{1,2\}}m=m\circ^{\Path_3}_{\{2,3\}}m,
 \\
\text{In }\K_3:\quad m\circ^{\K_3}_{\{1,2\}}m=m\circ^{\K_3}_{\{2,3\}}m.
\end{gather} From the presentation above, we see the similarity of this contractad with the commutative operad $\mathsf{Com}$. 

The Koszul dual contractad, called the Lie contractad and denoted $\Lie$, is the contractad generated by an anti-symmetric generator $b$ in component $\Path_2$, satisfying the relations
\begin{align}
    \text{In }\Path_3:\quad & b\circ^{\Path_3}_{\{1,2\}}b=b\circ^{\Path_3}_{\{2,3\}}b,
    \\
    \text{In }\K_3:\quad &
b\circ^{\K_3}_{\{1,2\}}b+(b\circ^{\K_3}_{\{1,2\}}b)^{(123)}+(b\circ^{\K_3}_{\{1,2\}}b)^{(321)}=0.
\end{align}

For a dg contractad $\Pop$, we define its suspension $\Susp\Pop$ by
\[
\Susp\Pop(\Gr):=\Pop(\Gr)[n-1]\otimes \mathrm{sgn}_{\Gr},
\] where $\mathrm{sgn}_{\Gr}$ is the restriction of the alternating representation: $\mathrm{sgn}_{\Gr}=\mathrm{Res}_{\Aut(\Gr)}^{\Sigma_{V_{\Gr}}} \mathrm{sgn}_{V_{\Gr}}$. In a similar way, we define the desuspension $\Susp^{-1}$ and, more generally, iterative suspension $\Susp^n$, for $n\in \mathbb{Z}$.

\subsection{Little disks contractad $\D_d$}
\label{sec:sub:E_d}
The little $d$-disks contractad $\D_d$ is a topological contractad, generalising the little $n$-disks operad. Each component of this contractad $\D_d(\Gr)$ consists of configurations of $n$-dimensional disks in the unit disk labeled by the vertex set of a graph, such that interiors of disks corresponding to adjacent vertices do not intersect. In other words, an element of $\D_d(\Gr)$ is a family of ``rectilinear''\footnote{By rectilinear embedding we mean the map given by dilations and translations of $\Real^{d}$} embeddings $\{f_v\colon \mathbb{D}^d\to \mathbb{D}^d, v\in V_{\Gr}\}$ of open discs $\mathbb{D}^d$ satisfying the edge-non-intersecting condition:
$$
(v\ w)\in E_{\Gr} \ \Rightarrow \ \mathrm{Im}(f_v) \cap \mathrm{Im}(f_w) = \emptyset. 
$$
\begin{figure}[ht]
    \centering
    \[
    \vcenter{\hbox{\begin{tikzpicture}
     \draw[thick] (0,0) circle [radius=35pt];
     \draw[thick] (-0.5,0.4) circle [radius=12pt];
     \node at (-0.5,0.4) {1};
     \draw[thick] (-0.25,-0.75) circle [radius=9pt];
     \node at (-0.25,-0.75) {2};
     \draw[thick] (0.5,-0.3) circle [radius=10pt];
     \node at (0.5,-0.3) {3};
     \draw[thick] (0.6,0.5) circle [radius=10pt];
     \node at (0.6,0.5) {4};
    \end{tikzpicture}}}\ ,
    \quad
    \vcenter{\hbox{\begin{tikzpicture}
     \draw[thick] (0,0) circle [radius=35pt];
     \draw[thick] (-0.4,0.4) circle [radius=13pt];
     \node at (-0.4,0.4) {1};
     \draw[thick] (-0.3,-0.6) circle [radius=10pt];
     \node at (-0.3,-0.6) {2};
     \draw[thick] (0.3,0.5) circle [radius=10pt];
     \node at (0.3,0.5) {3};
     \draw[thick] (0.7,-0.3) circle [radius=8pt];
     \node at (0.7,-0.3) {4};    
    \end{tikzpicture}}} \ ,
    \quad
    \vcenter{\hbox{\begin{tikzpicture}
     \draw[thick] (0,0) circle [radius=35pt];
     \draw[thick] (0.3,0.5) circle [radius=16pt];
     \node at (-0.1,0.5) {3};
     \draw[thick] (-0.1,-0.8) circle [radius=9pt];
     \node at (-0.1,-0.8) {2};
     \draw[thick] (-0.5,-0.4) circle [radius=10pt];
     \node at (-0.5,-0.4) {4};
     \draw[thick] (0.3,0.5) circle [radius=8pt];
     \node at (0.3,0.5) {1};
    \end{tikzpicture}}}
    \
    \in 
    \D_2\left( 
\vcenter{\hbox{    \begin{tikzpicture}[scale=0.6, shift={(0,-2)}]
    \fill (0,0) circle (2pt);
    \fill (0,1.5) circle (2pt);
    \fill (1.5,0) circle (2pt);
    \fill (1.5,1.5) circle (2pt);
    \draw (0,0)--(1.5,0)--(1.5,1.5)--(0,1.5)-- cycle;
    \node at (-0.25,1.75) {$1$};
    \node at (1.75,1.75) {$2$};
    \node at (1.75,-0.25) {$3$};
    \node at (-0.25,-0.25) {$4$};
    \end{tikzpicture}}}
    \right)
    \]
    \caption{Example of configurations in $\D_2(\Cyc_4)$.}
\end{figure}

The contractad structure is given by the insertion of an outer disk of one configuration in the interior of a subdisk of another configuration, as in the example in Figure~\ref{fig:disccomp}.

\begin{figure}[ht]
    \centering
    \[
    \quad
     \vcenter{\hbox{\begin{tikzpicture}
     \draw[thick] (0,0) circle [radius=35pt];
     \draw[thick] (-0.5,0.3) circle [radius=11pt];
     \node at (-0.5,0.3) {1};
     \draw[thick] (0.3,-0.5) circle [radius=14pt];
     \node at (0.3,-0.5) {\{2,4\}};
     \draw[thick] (0.15,0.6) circle [radius=13pt];
     \node at (0.15,0.6) {3};
    \end{tikzpicture}}}
    \quad
    \circ^{\Gr}_{\{2,4\}}
     \vcenter{\hbox{\begin{tikzpicture}
     \draw[thick] (0,0) circle [radius=30pt];
     \draw[thick] (-0.355,-0.55) circle [radius=11pt];
     \node at (-0.355,-0.55) {2};
     \draw[thick] (0.1,0.645) circle [radius=11pt];
     \node at (0.1,0.645) {4};    
    \end{tikzpicture}}}
    \quad
    =
    \quad
    \vcenter{\hbox{\begin{tikzpicture}
     \draw[thick] (0,0) circle [radius=50pt];
     \draw[thick] (-0.7,0.5) circle [radius=15pt];
     \node at (-0.7,0.5) {1};
     \draw[dashed, thick] (0.6,-0.7) circle [radius=20pt];
     \draw[thick] (0.7,-0.3) circle [radius=8pt];
     \node at (0.7,-0.3) {4};
     \draw[thick] (0.3,-1) circle [radius=8pt];
     \node at (0.3,-1) {2};
     \draw[thick] (0.15,1) circle [radius=17pt];
     \node at (0.15,1) {3};
    \end{tikzpicture}}}
    \ \in \
    \D_2\left(
      \vcenter{\hbox{\begin{tikzpicture}[scale=0.6]
    \fill (0,0) circle (2pt);
    \fill (0,1.5) circle (2pt);
    \fill (1.5,0) circle (2pt);
    \fill (1.5,1.5) circle (2pt);
    \draw (0,0)--(1.5,0)--(1.5,1.5)--(0,1.5)-- cycle;
    \draw (0,0)--(1.5,1.5);
    \node at (-0.25,1.75) {$1$};
    \node at (1.75,1.75) {$2$};
    \node at (1.75,-0.25) {$3$};
    \node at (-0.25,-0.25) {$4$};
    \end{tikzpicture}}}
    \right)
\]
    \caption{Disks composition in $\D_2$}
    \label{fig:disccomp}
\end{figure}

There is a connection between little disks contractad and, so-called, \emph{graphical configuration spaces}. For a topological space $X$ and graph $\Gr$, we define the graphical configuration space by
\[
\Conf_{\Gr}(X)=\{(x_v)\in X^{V_{\Gr}}|x_v\neq x_w \text{ for }(v,w)\in E_{\Gr}\}.
\] There is the centering map from the component $\D_d(\Gr)$ of little disks contractad to the graphical configuration space, 
\begin{equation}\label{eq::centering_map}
r\colon \D_d(\Gr)\to \Conf_{\Gr}(\mathbb{R}^d), \quad (r_v\mathbb{D}^d+x_v)_{v\in V_{\Gr}}\mapsto (x_v)_{v\in V_{\Gr}},
\end{equation}that maps configurations of disks to configurations of their centers. These maps are component-wise homotopy equivalences.

If we replace each component of the contractad $\D_d$ with its rational homology groups, we obtain a
contractad $H_{\bullet}(\D_d)$ in the category of graded vector spaces. 
For $d=1$, the space $\D_1(\Gr)$ is a disjoint union of contractible spaces and the homology contractad $H_{\bullet}(\D_d)$ is called an associative contractad (see~\cite{lyskov2023contractads} and Section~\ref{sec::polytopes} for details).
For $d\geq  2$, the homology contractad $H_{\bullet}(\D_d)$ is the quadratic Koszul contractad obtained from the contractads $\Com$ and $\Susp^{1-d}\Lie$ by rewriting rules
\begin{gather*}
    c_{d} \circ_{\{1,2\}}^{\Path_3} m =  m \circ_{\{2,3\}}^{\Path_3} c_{d},
    \\
    c_{d}\circ_{\{1,2\}}^{\K_3} m = m\circ_{\{2,3\}}^{\K_3} c_{d} + (m\circ_{\{1,2\}}^{\K_3}c_{d})^{(23)},
\end{gather*} where $m, c_d$ are the generators of $\Com$ and $\Susp^{1-d}\Lie$ respectively, and, moreover, we have the isomorphism 
\begin{equation}
\label{eq::Ed::graphical}
H_{\bullet}(\D_d)\cong \Com\circ\Susp^{1-d}\Lie
\end{equation}
of graphical collections~\cite{lyskov2023contractads}. Note that $m$ corresponds to the class of point, while $c_d$ corresponds to the fundamental class of sphere $S^{d-1}\simeq \Conf_{\Path_2}(\mathbb{R}^d)\cong \D_d(\Path_2)$. The Hopf structure is given on generators by the rule
\[
\triangle(m)=m\otimes m,\quad \triangle(c_d)=c_d\otimes m+ m\otimes c_d.
\] In the case $d=2$, the resulting contractad is called the Gerstenhaber contractad and is denoted by $\Gerst$ due to the operad notation~\cite{gerstenhaber1963cohomology}.

On the other hand, passing to rational cohomology, we obtain a Hopf cocontractad $\OS_d:=H^{\bullet}(\D_d)$.
I.e., the cocontractad in the category of commutative algebras over $\mathbb{Q}$.
Each component is the generalised Orlik-Solomon algebra $\OS_d(\Gr)$~\cite[Lemma~5.3.1]{lyskov2023contractads}. This algebra is generated by elements $\omega_e$ of degree $(d-1)$ for every edge $e$ in $\Gr$, satisfying the relations
\begin{equation}
\label{eq::OS::Gamma}
\begin{array}{c}
    \omega^2_{e}=0,
    \\
    \sum_{i=1}^n (-1)^{(d-1)(i-1)}\omega_{e_1}\omega_{e_2}...\hat{\omega}_{e_i}...\omega_{e_n}=0\text{, whenever edges } \{e_1,e_2,\cdots,e_n\} \subset E_{\Gr} \text{ assemble a cycle.}
\end{array}    
\end{equation} 
The cocontractad structure is given by homomorphisms of algebras:
\begin{equation}
\label{eq::OS::composition}
\begin{array}{c}
    \triangle_G^{\Gr}\colon\OS_d(\Gr) \rightarrow \OS_d(\Gr/G)\otimes\OS_d(\Gr|_G)
    \\
    \omega_{e}\mapsto \begin{cases}
    \omega_{e'} \otimes 1, \text{ if } e \not\subset G
    \\
    1\otimes \omega_{e}, \text{ if } e \subset G
    \\
    \end{cases},
\end{array}    
\end{equation} 
where $e'$ is the image of $e$ under contraction $\Gr \rightarrow \Gr/G$.

\subsection{Framed disks and Batalin-Vilkovski contractad}
\label{sec:sub::fE_d}
Similarly to the operad case~\cite{salvatore2003framed} (see also \cite{khoroshkin2013hypercommutative, bellier2014koszul}), for a topological group $G$ and a contractad $\Pop$ in $G$-spaces, we define the semidirect product contractad $\Pop\rtimes G$ by the formula:
\[
    \Pop \rtimes G(\Gr)=\Pop(\Gr)\times G^{\times V_{\Gr}}
\]
with the composition maps:
\[
    \gamma_{\Pop \rtimes G}\colon (\Pop \rtimes G)(\Gr/I) \times (\Pop \rtimes G)(\Gr|_{G_1})\times...\times (\Pop \rtimes G)(\Gr|_{G_k}) \rightarrow (\Pop \rtimes G)(\Gr)
\]
given by
\[
    \gamma_{\Pop \rtimes G}((\alpha,h);(\beta_1,g_1),...,(\beta_k,g_k))=(\gamma_{\Pop}(\alpha,h_{I_1}\beta_1,...,h_{I_k}\beta_k), h_{I_1}g_1,...,h_{I_k}g_k)
\]
where $h=(h_{1},...,h_{k}), g_j=(g_j^v)_{v \in G_j}$, and where $h_j$  acts on $g_{j}$ via diagonal action of $H$ on $H^{G_j}$. In the same way, we define the semiderict product contractad $\Orb\rtimes \Ho$ for a cocommutative Hopf algebra $\Ho$ and a contractad $\Orb$ in $\Ho$-modules. Note that, for a topological $G$-contractad, passing to homology, we obtain the isomorphism of contractads
\begin{equation}\label{eq::semidir_homology}
 H_{\bullet}(\Pop\rtimes G)\cong H_{\bullet}(\Pop)\rtimes H_{\bullet}(G),   
\end{equation} The action of the special orthogonal group $\SO(d)$ on the unit disk $\Disc^{d}$, induces the action of $\SO(d)$ on the little $d$-disks contractad $\D_d$.

\begin{definition}
    The framed little $d$-disks contractad $\fD_d$ is the semidirect product contractad $\D_d\rtimes \SO(d)$.
\end{definition}
\noindent We will be mostly interested in the $2$-dimensonal case $\fD_2=\D_2\rtimes \SO(2)$. 
\begin{definition}
    The Batalin-Vilkovski contractad $\BV$ is the homology contractad $H_{\bullet}(\fD_2)$.
\end{definition}
\noindent Thanks to isomorphism~\eqref{eq::semidir_homology}, we have the isomorphism of contractads $\BV=H_{\bullet}(\D_2\rtimes S^1)\cong H_{\bullet}(\D_2)\rtimes H_{\bullet}(S^1)$. The Hopf algebra $H_{\bullet}(S^1)$ is exactly Grassman algebra $\mathsf{k}[\Delta]$ with generator $\Delta$ of degree $1$ and quadratic relation $\Delta^2=0$, with coproduct map $\Delta \mapsto \Delta\otimes 1+1\otimes \Delta$. Note that the action of this algebra on a contractad is equivalent to a datum of a dg contractad with respect to the differential $\Delta$ of order $1$. The action of $\Delta$ on $H_{\bullet}(\D_2)\cong \Gerst$ is given on generators by the rule: $\Delta(m)=c_2, \Delta(c_2)=0$.
\begin{proposition} The contractad $\BV$ is obtained from the Gerstenhaber contractad $\Gerst$ and unary operation $\Delta$ of degree $1$, satisfying the relations:
\begin{gather}
\label{eq::delta:2}
    \Delta\circ\Delta=0
    \\
    \Delta \circ m - m \circ_{\{1\}}^{\Path_2} \Delta - m \circ_{\{2\}}^{\Path_2} \Delta = c_2
    \label{eq:deltam}
    \\    
    \Delta \circ c_2 + c_2 \circ_{\{1\}}^{\Path_2} \Delta + c_2 \circ_{\{2\}}^{\Path_2} \Delta=0
    \label{eq:deltab}
\end{gather}
\end{proposition}
\begin{proof}
As we have mentioned above, we have the  isomorphism of contractads
\begin{equation*}
    \BV\cong \Gerst\rtimes \mathsf{k}[\Delta].
\end{equation*} Hence, this contractad is generated by $m$, $b$, and $\Delta$. Next, we need to check the relations~\eqref{eq:deltam},\eqref{eq:deltab}.  By the definition of a semidirect product, we deduce the desired relations:
\begin{gather*}
    \Delta\circ m - m\circ_1\Delta -m\circ_2\Delta= \Delta(m)=c_2
    \\
    \Delta \circ c_2 + c_2 \circ_1 \Delta + c_2 \circ_2 \Delta= \Delta(c_2)=0.
\end{gather*} Therefore, there is a surjective map of contractads $\phi\colon\Pop\to \BV$, if when denote by $\Pop$ the contractad defined by the above mentioned presentation. Finally, we note that these relations \eqref{eq::delta:2}--\eqref{eq:deltab} define a rewriting rule according to~\cite[Sec.~1.3.3]{khoroshkin2024hilbert}. Therefore, there is an epimorphism $\BV=\Gerst\rtimes \mathsf{k}[\Delta]\cong \Gerst\circ\mathsf{k}[\Delta]\twoheadrightarrow \Pop$. So the map $\phi$ is an isomorphism, which concludes the proof.
\end{proof}

\section{Cohomology of graphical configurations and Lie contractad} 
\label{sec::Graphic::Lie}
There are several recent works~\cite{petersen2020cohomology, knudsen2023projection, knudsen2017betti, Gadish2021configuration} that describe the (co)homology of certain types of configuration spaces as the homology of appropriate Lie algebras. In this section, we formulate similar statements in the case of graphical configuration spaces, mimicking the strategy of Petersen~\cite{petersen2020cohomology}. In the graphical case, we need to replace Lie algebras with left modules over the Lie contractad $\Lie$.

\subsection{Twisted algebras over contractads}
\label{sec:sub:twisted}
Similarly to the operad case, we define twisted algebras over contractads as left modules  
\begin{definition}
A twisted $\Pop$-algebra is a graphical collection $\A$ endowed with a product map
\[
\gamma_{\A}\colon \Pop\circ\A\to \A
\] that is compatible with the contractad product map $\gamma_{\Pop}$. Equivalently, $\A$ is a left $\Pop$-module with respect to the contracted product $\circ$.
\end{definition} We omit the definitions and constructions of coalgebras over (co)contractads, free twisted (co)algebras and differential graded twisted (co)algebras.

\begin{example}[twisted $\Com$-algebras] For the commutative contractad $\Com$, a twisted algebra is equivalent to a datum of a graphical collection $\A$ with a collection of multiplications 
\[m\colon \A(\Gr|_{G_1})\otimes\A(\Gr|_{G_2})\to \A(\Gr)\] for each pair of graph $\Gr$ and its $2$-partition $\{G_1,G_2\}$. This multiplication is commutative $m(x_1,x_2)=m(x_2,x_1)$ and associative: for a $3$-partition $\{G_1,G_2,G_3\}$ of $\Gr$ with $G_1\cup G_2$ and $G_2\cup G_3$ are tubes, we have $m(x_1,m(x_2,x_3))=m(m(x_1,x_2),x_3)$. We consider the following examples of $\Com$-algebras
\begin{itemize}
    \item[(i)] For a space $X$, the inclusions $\Conf_{\Gr}(X)\to \Conf_{\Gr|_{G_1}}(X)\times \Conf_{\Gr|_{G_2}}(X)$ defines the structure of a twisted $\Com$-coalgebra on the configurations $\cup_{\Gr}\Conf_{\Gr}(X)$ and passing to cohomology, we obtain the twisted $\Com$-algebra structure on $\cup_{\Gr}H^{\bullet}(\Conf_{\Gr}(X))$.
    \item[(ii)] An usual commutative algebra $A$ defines a constant twisted $\Com$-algebra $\underline{A}$ with components $\underline{A}(\Gr):=A$ for every graph and the multiplication obtained from that of $A$.
    \item[(iii)] The free commutative algebra $\Com(\mathcal{V})$ generated by a graphical collection $\mathcal{V}$ has the components
    \[
    \Com(\mathcal{V})(\Gr)=\oplus_{I\vdash \Gr} \otimes_{G\in I} \mathcal{V}(\Gr|_G).
    \] with the obvious multiplication given by concatenation.
\end{itemize}
\end{example}
\begin{example}[twisted $\Lie$-algebras]\label{ex::twistedLie} Similarly, the twisted $\Lie$-algbera is a graphical collection $\mathfrak{g}$ with an antisymmetric product $b\colon \mathfrak{g}(\Gr|_{G_1})\otimes\mathfrak{g}(\Gr|_{G_2})\to \mathfrak{g}(\Gr)$, $b(x_2,x_1)=-b(x_1,x_2)$, satisfying the generalised Jacobi identity:
\[
\epsilon_{12}b(b(x_1,x_2),x_3)+\epsilon_{31}b(b(x_3,x_1),x_2)+\epsilon_{23}b(b(x_2,x_3),x_1)=0,
\] where $\epsilon_{ij}=1$ if $G_i\cup G_j$ is a tube, and $0$ otherwise. 

Similarly to the classical case, given a twisted $\Com$-algebra $\A$ and a twisted $\Lie$-algebra $\mathfrak{g}$, their Hadamard product $\A\underset{\mathrm{H}}{\otimes} \mathfrak{g}$ is again a twisted $\Lie$-algebra with the bracket 
$$[a_1\otimes x_1,a_2\otimes x_2] := a_1 a_2\otimes[x_1,x_2].$$
\end{example}

For a twisted $\Pop$-algebra $\A$, we define its suspension algebra $\mathsf{S}\A$, by the rule
\[
\mathsf{S}\A(\Gr)=\A(\Gr)[-n]\otimes\mathrm{sgn}_{\Gr}, \text{ where }n=|V_{\Gr}|,
\] with the algebra structure
\[
\gamma_{\mathsf{S}\A}\colon \Pop\circ(\mathsf{S}\A) \cong \mathsf{S}(\Pop\circ\A)\overset{\mathsf{S}(\gamma_{\A})}{\to}\mathsf{S}\A.
\]

Recall that a Koszul contractad is a quadratic contractad $\Pop$ whose Koszul complex $\Pop^{\cokoszul}\circ^{\kappa}\Pop \to \mathbb{1}$ is quasi-isomorphic to the unit $\mathbb{1}$, see~\cite[Sec.~3]{lyskov2023contractads} for details. For a Koszul contractad $\Pop$ and a twisted $\Pop$-algebra, we define a quasi-free twisted $\Pop^{\cokoszul}$-coalgebra $C^{\Pop}(\A)$ by the rule
\[
C^{\Pop}(\A):=(\Pop^{\cokoszul}\circ^{\kappa}\Pop)\circ_{\Pop}\A=\Pop^{\cokoszul}(\A).
\] Since $\Pop$ is Koszul, the homology of the free $\Pop$-algebra are
\[
H^{\Pop}(\Pop(\mathcal{V}))\cong \mathcal{V}
\]
\begin{example}[Chevalley-Eilenberg complex] Due to the applications, we replace homological gradings with cohomological ones. Recall that the Lie contractad $\Lie$ is Koszul and its Koszul dual is the commutative contractad $\Com$. For a twisted (cohomologically graded) dg $\Lie$-algebra $\mathfrak{g}^{\bullet}$ with a differential of degree $+1$, we define the Chevalley-Eilenberg complex as the shifted complex
\[
C^{\bullet}_{\mathrm{CE}}(\mathfrak{g}):=C^{\Lie}(\mathfrak{g})[1]=\Susp^{-1}\Com^{\vee}(\mathfrak{g})[1]\cong \Com^{\vee}(\mathfrak{g}[1])
\] By the construction, this complex forms a dg twisted $\Com$-coalgebra. The differential $d_{\mathrm{CE}}=\delta+d_{\mathfrak{g}}$ is the sum of two differentials: $d_{\mathfrak{g}}$ is induced from the differential of $\mathfrak{g}$ and $\delta$ is given by the formula
\[
\delta(s^{-1}x_{G_1}\otimes \cdots\otimes s^{-1}x_{G_k})=\sum \pm\epsilon_{ij} s^{-1}[x_{G_i},x_{G_j}]\otimes s^{-1}x_{G_1}\otimes\cdots \widehat{s^{-1}x_{G_i}}\cdots\widehat{s^{-1}x_{G_j}}\cdots\otimes s^{-1}x_{G_k},
\] where signs are determined by Koszul sign rules and $\epsilon_{ij}$ from Example~\ref{ex::twistedLie}. Since $\Lie$ is a Koszul contractad, we have the isomorphism
\[
H_{\mathrm{CE}}^{\bullet}(\Lie(\mathcal{V}))\cong \mathcal{V}[1].
\]
\end{example}

\subsection{Compactly supported sheaf cohomology of graphical configuration spaces} 
\label{sec:sub::Compact:support}
In this subsection, we generalize the results of~\cite[Sec.~8]{petersen2020cohomology} to the graphical case. Let $X$ be a paracompact and locally compact Hausdorff space, and $\F$ a bounded below complex of sheaves of $R$-modules on $X$, where $R$ is a ring of finite global dimension.

We consider the dg graphical collection $C^{\bullet}_c(\Conf(X); \F^{\boxtimes})$ with components $C^{\bullet}_c(\Conf(X); \F^{\boxtimes})(\Gr):=C^{\bullet}_c(\Conf_{\Gr}(X); \F^{\boxtimes V_{\Gr}})$ are chain complexes of compactly supported cochains on graphical configuration spaces. This graphical collection forms a twisted dg $\Com$-coalgebra with respect to the extension by zero
\[
C^{\bullet}_c(\Conf_{\Gr}(X); \F^{\boxtimes V_{\Gr}})\to C^{\bullet}_c(\Conf_{\Gr|_G}(X); \F^{\boxtimes V_G})\otimes C^{\bullet}_c(\Conf_{\Gr|_H}(X); \F^{\boxtimes V_H})
\] associated with the open inclusion $\Conf_{\Gr}(X)\subset \Conf_{\Gr|_G}(X)\times \Conf_{\Gr|_H}(X)$. We show how to compute the corresponding cohomology coalgebra $H^{\bullet}_c(\Conf(X); \F^{\boxtimes})$.

Recall that $\F$ is functorially quasi-isomorphic to a bounded below complex of flat and flabby sheaves~\cite[Lem.~3.4]{petersen2020cohomology}. We define the twisted $\Com$-algebra of derived compactly supported global section $\Real\Gr^{\otimes}(X;\F)$ by the rule
\[
\Real\Gr_c^{\otimes}(X;\F)(\Gr):=\Gr_c(X;\La^{\otimes V_{\Gr}}),
\] where $\La$ is a functorial flabby flat resolution of $\F$, with the algebra structure induced from 
$$\La^{\otimes V_G}\otimes\La^{\otimes V_H}\to \La^{\otimes V_{G\cup H}}.$$ This twisted algebra computes the cohomology of derived tensor products
\[
H_c^{\bullet}(\Real\Gr_c^{\otimes}(X;\F)(\Gr))\cong H_c^{\bullet}(X;\otimes_{v\in V_{\Gr}}^{\mathbb{L}}\F).
\] 
\begin{theorem}\label{thm::sheaves_lie}
Let $X$ be a paracompact and locally compact Hausdorff space, and $\F$ a bounded below complex of sheaves of $R$-modules on $X$, where $R$ is a ring of finite global dimension. Then we have an isomorphism
\[
H_{\mathrm{CE}}^{\bullet}(\Real\Gamma_c^{\otimes}(X,\F[-1])\underset{\mathrm{H}}{\otimes} \Lie)(\Gr)\cong H_c^{\bullet}(\Conf_{\Gr}(X),\F^{\boxtimes V_{\Gr}}),
\] where $\Lie$ is considered as a free twisted $\Lie$-algebra. Moreover, the isomorphism holds on the level of twisted $\Com$-coalgebras. 
\end{theorem}
\begin{proof} The proof repeats the ones of~\cite[Cor.~8.6]{petersen2020cohomology}. Consider the lattice of graph partitions $\parti(\Gr)$ with respect to refinement order. For a twisted dg $\Com$-algebra $\A^{\bullet}$ over $R$, we define the functor $\Phi_A\colon \parti(\Gr)\to \mathrm{Ch}_R$, by the rule $\Phi_{\A}(I):=\otimes_{G\in I}\A(\Gr|_G)$. Respectively, for a partition $I'$ obtained from $I$ by merging some blocks, i.e. $I'\geq I$, we consider a chain map $\Phi_A(I')\to \Phi_A(I)$ induced from algebra structure $\A(\Gr|_{G_i})\otimes \A(\Gr|_{G_i})\to \A(\Gr|_{G_i\cup G_j})$. We consider a double complex $\mathsf{CF}(\Gr,\A)$ by
\[
\mathsf{CF}(\Gr,\A)=\bigoplus_{\lambda\subseteq\parti(\Gr)-\text{chain} } \Phi_A(\max \lambda),
\] with the "vertical" differential induced from the differential in $\Phi_A(I)$, and the "horizontal" differential is an alternating sum over all possible ways of adding an extra element to the chain.

From the general sheaf theory there is a quasi-isomorphism $$C^{\bullet}_c(\Conf_{\Gr}(X);\F^{\boxtimes V_{\Gr}})\cong \Real\Gr_c(X^{V_{\Gr}}; j_!j^*\F^{\boxtimes V_{\Gr}}),$$ 
where $j\colon \Conf_{\Gr}(X)\to X^{V_{\Gr}}$. For a graph partition $I$ of $\Gr$, we consider a subset 
$$X(I):=\{(x_v)\in X^{V_{\Gr}}|\text{ for }(v,w)\in E_{\Gr}: x_v=x_w \Leftrightarrow x_v \text{ and } x_w\text{ in the same block of }I\}.$$ This subsets defines a locally open stratification of $X^{V_{\Gr}}$ with only open strata $X(\{\{v\}\}_{v\in V_{\Gr}})=\Conf_{\Gr}(X)$ and the poset of this stratification is a lattice of graph partitions with respect to refinement order. Thanks to~\cite{petersen2017spectral}, we have the decomposition
\[
j_!j^*\F^{\boxtimes V_{\Gr}}\cong \bigoplus_{\lambda\subseteq\parti(\Gr)-\text{chain} }(\iota_{\max(\lambda)})_*(\iota_{\max(\lambda)})^*\F^{\boxtimes V_{\Gr}},
\] where $\iota_I\colon X(I)\to X^{V_{\Gr}}$. If $\La$ is the flat and flabby resolution of $\F$, then, thanks to decomposition above, we have a quasi-isomorphism $j_!j^*\F^{\boxtimes V_{\Gr}}\simeq \bigoplus_{\lambda\subseteq\parti(\Gr)-\text{chain} }(\iota_{\max(\lambda)})_*(\iota_{\max(\lambda)})^*\La^{\boxtimes V_{\Gr}}$. The later sheaf is the complex of soft sheaves, so this resolution can be used to compute the cohomology of $j_!j^*\F$. Moreover, this complex of sheaves is isomorphic to the double complex $\mathsf{CF}(\Gr,\Real\Gr^{\otimes}(X,\F))$. All in all, we obtain the quasi-isomorphism of complexes
\[
C^{\bullet}_c(\Conf_{\Gr}(X);\F^{\boxtimes V_{\Gr}})\simeq \mathsf{CF}(\Gr,\Real\Gr_c^{\otimes}(X,\F))
\] To complete the proof, we need the following lemma
\begin{lemma}
For a twisted dg $\Com$-algebra $\A$, we have a quasi-isomorphism of complexes
\[
\mathsf{CF}(\Gr,\A)\simeq C_{\mathrm{CE}}^{\bullet}(\mathsf{S}\A\underset{\mathrm{H}}{\otimes} \Lie)(\Gr).
\]
\end{lemma}
\begin{proof} We give a sketch of the proof since it completely repeats the one of~\cite[Thm~8.4]{petersen2020cohomology}. For a graph $\Gr$, let $\tilde C(\parti(\Gr))$ be the be the reduced order complexes of  $\parti(\Gr)$ (the augmented complex of chains
with minimal and maximal elements). Thanks to~\cite[Thm~3.2.1]{lyskov2023contractads}, we have the quasi-isomorphism $\mathsf{B}(\Com) \overset{\simeq}\hookrightarrow \tilde C(\parti)$ ( with respect to certain homological regrading of $\tilde C(\parti(\Gr))$), where $\mathsf{B}(\Com)$ is the Bar construction of the commutative cocontractad, see~\cite[Sec.~3]{lyskov2023contractads} for details. Since $\Com$ is a Koszul contractad with Koszul dual $\Lie$, we have the quasi-isomorphism $\Com=\Susp^{-1}\Lie^{\vee}\overset{\simeq}{\hookrightarrow} \mathsf{B}(\Com)$.

By the construction, the dg graphical collection $\mathsf{CF}(-,\A)$ has the form $\mathsf{CF}(-,\A)=\Com\circ (\A\underset{\mathrm{H}}{\otimes} \tilde C(\parti)^{\vee})$, where $\tilde C(\parti)^{\vee}$ is the linear dual dg graphical collection ( with cohomological grading). Comparing differentials and combining quasi-isomorphisms, we obtain a quasi-isomorphism
\begin{equation*}
\centering
\Com\circ (\A\underset{\mathrm{H}}{\otimes} \tilde C(\parti)^{\vee})\overset{\simeq}{\twoheadrightarrow} \Com\circ (\A\underset{\mathrm{H}}{\otimes} \mathsf{B}(\Com)^{\vee})\overset{\simeq}{\twoheadrightarrow} \Com\circ (\A\underset{\mathrm{H}}{\otimes} \Susp^{-1}\Lie)= C_{\mathrm{CE}}^{\bullet}(\mathsf{S}\A\underset{\mathrm{H}}{\otimes} \Lie)(\Gr),
\end{equation*} where the all dg graphical collections are considered with cohomological gradings instead of homological ones.
\end{proof} All in all, we have
\[
C^{\bullet}_c(\Conf_{\Gr}(X);\F^{\boxtimes V_{\Gr}})\simeq C_{\mathrm{CE}}^{\bullet}(\mathsf{S}\Real\Gr^{\otimes}(X,\F)\otimes \Lie)(\Gr)=C_{\mathrm{CE}}^{\bullet}(\Real\Gr^{\otimes}(X,\F[-1])\otimes \Lie)(\Gr)
\]
\end{proof} 

\subsection{Applications}
\label{sec:sub:application::twisted}
So, we are ready to state the main results of this section 
\begin{theorem}\label{thm::cohomology_twisted_lie} Let $X$ be a paracompact and locally compact Hausdorff space, then we have an isomorphism of twisted $\Com$-coalgebras
\[
H_{\mathrm{CE}}^{\bullet}(\underline{\A^c_X}\underset{\mathrm{H}}{\otimes} \mathsf{S}\Lie)\cong H_c^{\bullet}(\Conf(X);\mathbb{Q}),
\] where $\A^c_X$ is a cdga model for the compactly supported cochains on $X$.
\end{theorem}
\begin{proof} For the constant sheaf $\mathbb{Q}$, the twisted $\Com$-algebra $\Real\Gr_c^{\otimes}(X;\mathbb{Q})$ is equivalent to a constant twisted $\Com$-algebra $\underline{\A_X^c}$ associated with a model $\A_X^c$ for compactly supported cochains $C_c^{\bullet}(X)$~\cite[Cor~3.13]{petersen2020cohomology}.
\end{proof} 

In the special case $X=\Real^d$, we have
\begin{corollary}
The compactly supported cohomology of the graphical configuration spaces $\Conf_{\Gr}(\Real^d)$ forms a free twisted $\Com$-coalgebra
\[
H_c^{\bullet}(\Conf(\Real^d);\mathbb{Q})\cong \Com^{\vee}(\mathsf{S}\Lie[1-d]).
\]
\end{corollary}
\begin{proof} For $\Real^d$, its cdga model is the one-dimensional algebra $\A^c_{\Real^d}=\mathbb{Q}\langle x^d\rangle$ with zero multiplication, so the differential in $C_{\mathrm{CE}}^{\bullet}(\underline{\A^c_{\Real^d}}\underset{\mathrm{H}}{\otimes} \mathsf{S}\Lie)=\Com^{\vee}(\mathsf{S}\Lie[1-d])$ is trivial.
\end{proof} For the first non-trivial example, the wedge of $g$ copies of circles $\bigvee^g S^1$, the complete description of cohomology $H_c^{\bullet}(\Conf(\bigvee^g S^1))$ is a complicated task even in the case of ordinary configuration spaces; see~\cite{gadish2022configuration} for details. But for particular families of graphs, we can explicitly compute the twisted Chevalley-Eilenberg homology from Theorem~\ref{thm::cohomology_twisted_lie}.
\begin{example}[Paths and wedge of spheres]  Let us restrict ourselves to the family of paths $\Path_n$ and the corresponding configuration spaces
\[
\Conf_{\Path_n}(X)=\{(x_1,\cdots,x_n)\in X^{\times n}| x_i\neq x_{i+1} \text{ for }i=1,\cdots,n-1\}
\] Recall that, for the path $\Path_n$, its partition poset coincides with the poset of partitions of $[n]$ into intervals: $\parti(\Path_n)=\Pi^{\mathrm{ord}}_n$. Also, for the paths, the components of $\Lie$ are one-dimensional. So, for a cdga algebra $\A^{\bullet}$, the complex $C_{\mathrm{CE}}^{\bullet}(\underline{\A}\underset{\mathrm{H}}{\otimes} \mathsf{S}\Lie)(\Path_n)$ has the form
\[
C_{\mathrm{CE}}^{\bullet}(\underline{\A}\underset{\mathrm{H}}{\otimes} \mathsf{S}\Lie)(\Path_n)=\bigoplus_{I_1\oplus I_2\oplus \cdots\oplus I_k=[n]} \A^{\otimes k}[k-n] 
\] with the vertical differential coming from $A$ and the horizontal differential given by
\[
d(x_{I_1}\otimes x_{I_2}\otimes\cdots\otimes x_{I_k})=\sum (-1)^{j-1} x_{I_1}\otimes\cdots \otimes (x_{I_j}\cdot x_{I_{j+1}})_{I_j\oplus I_{j+1} }\otimes\cdots \otimes x_{I_k}
\]
Consider the wedge of $d$-spheres $\bigvee^g S^d$. Since $\bigvee^g S^d$ is formal and compact, we can take $\A^c_{\bigvee^g S^d}=H^{\bullet}(\bigvee^g S^d)=\mathbb{Q}\oplus \mathbb{Q}\langle x_1,x_2,\cdots,x_g\rangle$, where $x_i$ are of degree $d$ with zero multiplication.

We denote by $C(g,n)$ the complex $C_{\mathrm{CE}}^{\bullet}(H^{\bullet}(\bigvee^g_{i=1} S^1)\underset{\mathrm{H}}{\otimes} \mathsf{S}\Lie)$. For a multi-index $\alpha\in [g]^{k}$, let $C^{\alpha}(g,n)$ be the subcomplex generated by monomials that contain factors $x_{\alpha_1},\cdots,x_{\alpha_k}$ from left to right. We have the decomposition $C(g,n)=\bigoplus^n_{k=0} \oplus_{|\alpha|=k} C^{\alpha}(g,n)$ and the complex $C^{\alpha}(g,n)$ is isomorphic to $C^{(1^k)}(1;n)$. Note that we have $C_{\mathrm{CE}}^{\bullet}(H^{\bullet}(S^d)\underset{\mathrm{H}}{\otimes} \mathsf{S}\Lie)=\oplus^n_{k=1}C^{(1^k)}(1;n)$.

The projection $\Conf_{\Path_n}(S^d)\to \Conf_{\Path_{n-1}}(S^d)$ associated with the inclusion $\Path_{n-1}\subset \Path_n$ is a fibration with the fiber $S^d\setminus \mathrm{pt}\cong \Real^d$. Hence, we have a homotopy equivalence $\Conf_{\Path_n}(S^d)\simeq S^d$, so, by Poincare duality for c.s. cohomology, the non-zero cohomology $H_c^{\bullet}(\Conf_{\Path_n}(S^d))$ are concentrated in degrees $dn$ and $d(n-1)$ and are one-dimensional. In particular, the homology of $C^{(1^k)}(1;n)$ is non-zero only in the case $k=d(n-1),dn$ and in this case the homology is one-dimensional and concentrated in the degree $dk$. All in all, we have
\[
H_c^{i}(\Conf_{\Path_n}(\bigvee^g_{i=1} S^d))=\begin{cases}
    \mathbb{Q}^{\oplus g^n},\text{ if } i=dn
    \\
    \mathbb{Q}^{\oplus g^{n-1}},\text{ if } i=d(n-1)
    \\
    0, \text{ otherwise }
\end{cases}
\]
\end{example}
Recall that the \emph{chromatic polynomial of a graph} computes the number of ``proper'' colorings of vertices of $\Gamma$:
\begin{equation}
\label{eq::chromatic::coloring}
\mathfrak{X}_{\Gr}(q):=\#\left\{
f: V \to \{1,\dots, q\} \text{ such that } (uv)\in E_{\Gr} \ \Rightarrow \ f(u)\neq f(v). \right\}
\end{equation}
Chromatic polynomial was introduced by Birkhoff in~\cite{birkhoff1912determinant} and is widely used in combinatorics.

The following immediate consequence of Theorem~\ref{thm::cohomology_twisted_lie} motivates the name ``\textit{chromatic configuration space}'' used in~\cite{kallel2024configuration} for what we call in this paper ``graphical configuration spaces''. 
\begin{corollary}[Euler characteristic and chromatic polynomials]\label{cor::euler_characteristic}\hfill\break 
\begin{itemize}[itemsep=0pt,topsep=0pt]
    \item For a paracompact and locally compact Hausdorff space $X$, we have
\[
\chi_c(\Conf_{\Gr}(X))=\mathfrak{X}_{\Gr}(\chi_c(X)),
\] where $\chi_c$ is the compactly supported Euler characteristic and $\mathfrak{X}_{\Gr}(q)$ is a chromatic polynomial of $\Gr$.

\item In particular, if $M$ is an orientable $d$-dimensional manifold, we recognize the results~\cite{eastwood2007euler, kalle2021}
\[
\chi(\Conf_{\Gr}(M))=(-1)^{d|V_{\Gr}|}\mathfrak{X}_{\Gr}((-1)^d\chi(M))
\]
\end{itemize}
\end{corollary}
\begin{proof}
Recall from~\cite[Sec~2]{khoroshkin2024hilbert} that a graphical function is a function on the class of connected graphs that is constant on isomorphism classes. For a dg graphical collection $\Orb$ we consider the Euler graphical function $\chi(\Orb)$ by the rule $\chi(\Orb)(\Gr):=\chi(\Orb(\Gr))$. We have the $*$-product of the graphical function given by the formula
\[
(\psi*\phi)(\Gr)=\sum_{I\vdash \Gr} \psi(\Gr/I)\prod_{G\in I}\phi(\Gr|_G),
\] that preserves the $\circ$-product: $\chi(\Orb\circ \Q)=\chi(\Orb)*\chi(\Q)$. We use the notations $\mu:=\chi(\Susp^{-1}\Lie)$ and $\mathbb{1}_q$ for the graphical function $\mathbb{1}_q(\Gr)=q^{|V_{\Gr}|-1}$ and $\mathbb{1}:=\mathbb{1}_1$. Thanks to Theorem~\ref{thm::cohomology_twisted_lie}, we have
\[
\chi_c(\Conf(X)):=\chi(\Com\circ (\underline{\A_X}\underset{\mathrm{H}}{\otimes} \Susp^{-1}\Lie))=\mathbb{1}*(\chi(\underline{\A^c_X})\cdot \mu)=\mathbb{1}*(\chi_c(X)\cdot \mu)
\] Thanks to~\cite[Lem~2.3.4]{khoroshkin2024hilbert}, we have $\mathbb{1}*(q\cdot\mu)=\mathfrak{X}(q)$, where $\mathfrak{X}(q)$ is the graphical function encoding chromatic polynomials. The second statement follows from the first and from Poincar\'e duality. 
\end{proof}In the case $X$ is a manifold, we can compute integer homology of its configuration spaces using the Chevalley-Eilenberg homology 
\begin{theorem}
Let $M^d$ be a manifold of dimension $d$ (possibly with boundary) and $\mathsf{Or}_M$ the orientation sheaf. Then
\[
H^{\mathrm{CE}}_{\bullet}(\Real\Gamma_c^{\otimes}(M,\mathsf{Or}_M[d-1])\otimes \Lie)\cong H_{\bullet}(\Conf(M);\mathbb{Z})
\]
\end{theorem}
\begin{proof}
For a space $X$, we have $H_{\bullet}(X;\mathbb{Z})\cong H_c^{\bullet}(X;\mathbb{DZ})$, where $\mathbb{D}\mathbb{Z}$ is the dualizing complex. For a manifold $M$, the dualizing complex $\mathbb{D}\mathbb{Z}$ is the shifted orientation sheaf $\mathsf{Or}_M[d]$.
\end{proof}

\section{Projective and Spherical compactifications of graphical configuration spaces}
\label{sec::compactifications}
Recall that the ordinary configuration spaces $\Conf_k(M^d)$ of manifolds admit Fulton–MacPherson–-Axelrod–Singer compactifications $\FM_M(k)$, introduced in~\cite{axelrod1992chern}, that are manifolds with corners. When $M=\Real^d$, the collection of these compactifications $\FM_d:=\{\FM_{\Real^d}(k)\}$ forms an operad that is weakly homotopy equivalent to the little disks operad $\D_d$~\cite{salvatore2001configuration}. For a parallelizable manifold $M^d$, the Fulton-MacPherson compactifications form a right module $\FM_M$ over $\FM_d$. A similar statement holds for an arbitrary manifold if we consider the framed versions of FM compactifications. We refer to~\cite{sinha2004manifold} for the geometric description and to~\cite{lambrechts2014formality, markl1996compactification, campos2023model} for the operad structure description. In this section, we extend these results to the graphical case using the contractad language.

\subsection{Wonderful contractad and Modular compactifications}
\label{sec::wonerful_contractad} 

Before proceeding with the construction of the real compactifications of $\Conf_{\Gr}(\mathbb{R}^d)$, we briefly revisit the algebraic-geometric counterpart,  the Wonderful contractad $\bM(\Gr)$. This contractad serves as a graphical analogue to the Deligne-Mumford operad of stable rational curves $\beM_{0,n+1}$.  We refer to~\cite[Sec.~3]{khoroshkin2024hilbert} and~\cite[Sec.~2.6]{ lyskov2023contractads} for detailed treatment.

\subsubsection{Graphical compactifications} For a set with at least $2$ elements, we consider the vector space $W_{S}:=\Cmplx\langle e_v|v\in S\rangle/\langle \sum_{e\in S} x_v\rangle$. We shall use notation $x=(x_v)_{v\in S}$ for coordinates of a vector in $W_S$. 

For a connected graph $\Gr$, consider the complex hyperplane arrangement $\B(\Gr) = \{H_{vw} := \{x_v = x_w\} : (v, w) \in E_{\Gr}\}$ of $W_{V_{\Gr}}$. For a non-trivial tube $G \subseteq V_{\Gr}$, we consider the subspace
$$H_G := \bigcap_{e \in E_{\Gr|_G}} H_{e}=\{x_v = x_w, \text{ for all }v,w\in G\}.$$ This collection of subspaces forms a so-called \emph{building set} in the sense of~\cite{de1995wonderful}. Note that we have a natural isomorphism of vector spaces $W_{V_{\Gr}}/H_G\cong W_G$.

We consider the inclusion of the projective complement of the graphical arrangement $\M(\Gr):=\Conf_{\Gr}(\mathbb{C})/\Aff_1(\mathbb{C})=\Pro(W_{\Gr})\setminus \cup_{e\in E_{\Gr}} \Pro(H_e)$ to the product of projective spaces
\[
\rho\colon \M(\Gr)\rightarrow \prod_{G\subset V_{\Gr}\text{-tube} }\Pro(W_G),
\] where $\rho=\prod_G \rho_G$ is the product of projections $\rho_G\colon \M(\Gr)\to \Pro(W_{\Gr}/H_G)\cong \Pro(W_G)$.
 
The graphical compactification, denoted $\bM(\Gr)$, is the closure of the image of the embedding $\rho$. For the graph $\Path_1$, we let $\bM(\Path_1)=\mathrm{pt}$. By the construction, $\bM(\Gr)$ is the wonderful compactification associated with the graphic building set $\G_{\Gr}:=\{H_G\}_{G\subset V_{\Gr}\text{-tube}}$ in the sense of~\cite{de1995wonderful}. This compactification is a smooth projective $(|V_{\Gr}|-1)$-dimensional complex variety. The contractad structure on $\bM$ is given by closed inclusions of smooth varieties:
\begin{equation}\label{eq::rainsmaps}
  \iota\colon \bM(\Gr/G)\times \bM(\Gr|_{G})\to \bM(\Gr),  
\end{equation} defined as follows by determining $\rho_K\circ \iota$ for each tube $K$
\begin{equation}\label{eq::contractad_struc_cmplx}
\rho_K\circ \iota=\begin{cases}
\rho_{K}\circ \mathrm{pr}_{\Gr|_G}\text{, for }K\subset G,
\\
\Pro(j^K_G)\circ\rho_{K/G}\circ \mathrm{pr}_{\Gr/G}\text{, for }K\not\subset G,
\end{cases}  
\end{equation} where $\mathrm{pr}_{\Gr|_G}$ and $\mathrm{pr}_{\Gr/G}$ are projections on the corresponding factors, and $j^{K}_G\colon W_{K/G}\hookrightarrow W_{K}$ is the linear map: $j_G^K(x_v)=x_v$ for $v\not \in G$, and $j_G^K(x_{\{G\}})=\sum_{v\in G\cap K}x_v$. The image of this map, denoted $D_G$, forms an irreducible divisor. The complement $\partial\bM(\Gr):=\bM(\Gr)\setminus\M(\Gr)$ is a divisor with normal crossings with irreducible components $D_G$ indexed by tubes $G\subset \Gr$.

\subsubsection{Modular compactifications}
\label{sec::modular}
For a complete multipartite graph, the graphical compactification parametrizes the genus-zero stable pointed curves with certain conditions on marked points. We refer to~\cite[Sec.~3.2]{khoroshkin2024hilbert} for details.

For a partition $\lambda=(\lambda_1\geq \lambda_2\geq\cdots\geq\lambda_k)$, we consider the complete multipartite graph $\K_{\lambda}$.  This graph consists of blocks of vertices of sizes $\lambda_1,\lambda_2,\cdots,\lambda_k$, such that two vertices are adjacent if and only if they belong to different blocks. This class of graphs is closed under contractions and restrictions. Moreover, for a non-trivial contraction $\K_{\lambda}/G$, the contracted vertex $\{G\}$ is adjacent to all remaining vertices.
     \begin{figure}[ht]
\[
\vcenter{\hbox{\begin{tikzpicture}[scale=0.6]
    \fill (0,0) circle (2pt);
    \fill (0,1.5) circle (2pt);
    \fill (1.5,0) circle (2pt);
    \fill (1.5,1.5) circle (2pt);
    \node at (0.75,-0.6) {$\K_{(4)}$};
    \end{tikzpicture}}}
    \quad\quad
\vcenter{\hbox{\begin{tikzpicture}[scale=0.6]
    \fill (0,0) circle (2pt);
    \fill (-1,1.5) circle (2pt);
    \fill (0,1.5) circle (2pt);
    \fill (1,1.5) circle (2pt);
    \draw (0,0)--(-1,1.5);
    \draw (0,0)--(0,1.5);
    \draw (0,0)--(1,1.5);
    \node at (0,-0.6) {$\K_{(3,1)}$};
    \end{tikzpicture}}}
    \quad
\vcenter{\hbox{\begin{tikzpicture}[scale=0.6]
    \fill (0,0) circle (2pt);
    \fill (0,1.5) circle (2pt);
    \fill (1.5,0) circle (2pt);
    \fill (1.5,1.5) circle (2pt);
    \draw (0,0)--(1.5,0)--(1.5,1.5)--(0,1.5)-- cycle;
    \node at (0.75,-0.6) {$\K_{(2^2)}\cong \Cyc_4$};
    \end{tikzpicture}}}
    \quad
\vcenter{\hbox{\begin{tikzpicture}[scale=0.6]
    \fill (0,0) circle (2pt);
    \fill (0,1.5) circle (2pt);
    \fill (1.5,0) circle (2pt);
    \fill (1.5,1.5) circle (2pt);
    \draw (0,0)--(1.5,0)--(1.5,1.5)--(0,1.5)-- cycle;
    \draw (0,0)--(1.5,1.5);
    \node at (0.75,-0.6) {$\K_{(2,1^2)}$};
    \end{tikzpicture}}}
    \quad
\vcenter{\hbox{\begin{tikzpicture}[scale=0.6]
    \fill (0,0) circle (2pt);
    \fill (0,1.5) circle (2pt);
    \fill (1.5,0) circle (2pt);
    \fill (1.5,1.5) circle (2pt);
    \draw (0,0)--(1.5,0)--(1.5,1.5)--(0,1.5)-- cycle;
    \draw (0,0)--(1.5,1.5);
    \draw (1.5,0)--(0,1.5);
    \node at (0.75,-0.6) {$\K_{(1^4)}\cong \K_4$};
    \end{tikzpicture}}}
\]
\caption{List of complete multipartite graphs on 4 vertices.}
\end{figure} 

For a complete multipartite graph $\K_{\lambda}$, a $\K_\lambda$-pointed stable curve is a pointed $\mathbb{C}$-curve $(\Sigma, x_{\infty},\{x_v\}_{v\in V_{\K_{\lambda}}})$ of genus zero with only singularities are double points, satisfying the following conditions: 
 \begin{itemize}
    \item The marked point $x_{\infty}$ is separated from all other marked points.
     \item For adjacent vertices $v,w\in V_{\K_{\lambda}}$, the marked points are separated $x_v\neq x_w$.
     \item Each irreducible component $C\subset \Sigma$ not containing $x_{\infty}$ with a single node contains a pair of marked points $x_v,x_w\in C$ labeled by adjacent vertices $(v,w)\in E_{\K_{\lambda}}$.
 \end{itemize}
The moduli space $\beM_{0,\K_\lambda}$ of $\K_{\lambda}$-pointed stable curves is isomorphic to the compactification $\bM(\K_{\lambda})$~\cite[Thm.~3.2.12]{khoroshkin2024hilbert}. The moduli subspace of smooth curves $\eM_{0,\K_{\lambda}} $ coincides with the open moduli space $\M(\K_{\lambda})$. Moreover, under this isomorphism, contractad maps are indeed gluing maps. Specifically, the gluing map
\[
\circ^{\K_\lambda}_G\colon \beM_{0,\K_\lambda/G}\times\beM_{0,\K_\lambda|_G}\to \beM_{0,\K_\lambda},
\]  is obtained by gluing the special point $x_{\infty}$ of the $\K_{\lambda}|_G$-pointed stable curve $\Sigma$ to the marked point $x'_{\{G\}}$ of the $\K_{\lambda}|_G$-pointed stable curve $\Sigma'$. This procedure is well defined since both $x_{\infty}$ and $x'_{\{G\}}$ are separated from the remaining marked points. 

We list particular examples of these moduli spaces
\begin{example}
\label{ex::Wonderful}
\begin{itemize}
    \item[(i)] In the case $\K_{\lambda}=\K_n$ is complete graph, we recover the standard $\beM_{0,\K_n}\cong \beM_{0,n+1}$ Deligne-Mumford moduli space of $(n+1)$-marked stable curves. The restriction of contractad maps to these moduli spaces recovers the standard operad structure.
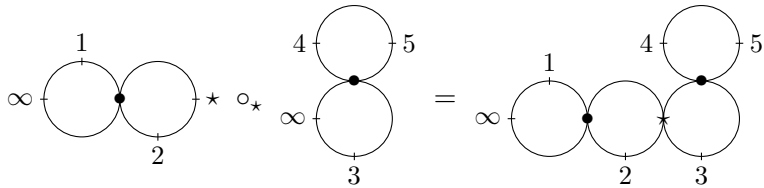
\[
\vcenter{\hbox{\begin{tikzpicture}[scale=0.5]
\draw (0,0) circle (1);
\draw (2,0) circle (1);
\node at (1,0) {\small$\bullet$};
\node [left] at (-1,0) {\small$\infty$};
\draw (-1.1,0) -- (-0.9,0);
\node [above] at (0,1) {\small$1$};
\draw (0,1.1) -- (0,0.9);
\node [below] at (2,-1) {\small$2$};
\draw (2,-0.9) -- (2,-1.1);
\node [right] at (3,0) {\small$\star$};
\draw (2.9,0) -- (3.1,0);
\end{tikzpicture}}} \circ_{\star} \vcenter{\hbox{\begin{tikzpicture}[scale=0.5]
\draw (0,0) circle (1);
\draw (0,2) circle (1);
\node at (0,1) {\small$\bullet$};
\node [left] at (-1,0) {\small$\infty$};
\draw (-1.1,0) -- (-0.9,0);
\node [below] at (0,-1) {\small$3$};
\draw (0,-0.9) -- (0,-1.1);
\node [left] at (-1,2) {\small$4$};
\draw (-1.1,2) -- (-0.9,2);
\node [right] at (1,2) {\small$5$};
\draw (0.9,2) -- (1.1,2);
\end{tikzpicture}}}= 
\vcenter{\hbox{\begin{tikzpicture}[scale=0.5]
\draw (0,0) circle (1);
\draw (2,0) circle (1);
\node at (1,0) {\small$\bullet$};
\node at (3,0) {\small$\star$};
\draw (4,0) circle (1);
\draw (4,2) circle (1);
\node at (4,1) {\small$\bullet$};
\node [left] at (-1,0) {\small$\infty$};
\draw (-1.1,0) -- (-0.9,0);
\node [above] at (0,1) {\small$1$};
\draw (0,1.1) -- (0,0.9);
\node [below] at (2,-1) {\small$2$};
\draw (2,-0.9) -- (2,-1.1);
\node [below] at (4,-1) {\small$3$};
\draw (4,-0.9) -- (4,-1.1);
\node [left] at (3,2) {\small$4$};
\draw (2.9,2) -- (3.1,2);
\node [right] at (5,2) {\small$5$};
\draw (4.9,2) -- (5.1,2);
\end{tikzpicture}}}
\]
    \item[(ii)] For stellar graph $\K_{(1,n)}=\St_n$, we recover the Losev-Manin moduli space $\overline{\EuScript{L}}_{0,n}$~ \cite{losev2000new}. This moduli space parametrizes chains of projective lines with two poles labeled by $\infty$ and $0$, along with $n$ marked points that may coincide.
    \[
    \vcenter
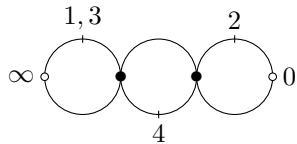
{\hbox{\begin{tikzpicture}[scale=0.5,]
\draw (0,0) circle (1);
\draw (2,0) circle (1);
\draw (4,0) circle (1);
\node at (1,0) {\small$\bullet$};
\node at (3,0) {\small$\bullet$};
\node [left] at (-1,0) {\small$\infty$};
\draw [fill=white] (-1,0) circle (0.1);
\node [right] at (5,0) {\small$0$};
\draw [fill=white] (5,0) circle (0.1);
\node [above] at (0,1) {\small$1,3$};
\draw (0,1.1) -- (0,0.9);
\node [below] at (2,-1) {\small$4$};
\draw (2,-0.9) -- (2,-1.1);
\node [above] at (4,1) {\small$2$};
\draw (4,0.9) -- (4,1.1);
\end{tikzpicture}}}
    \]
    \item[(iii)] More generally, for a complete multipartite graph $\K_{(1^n,m)}$, the moduli space $\beM_{0,\K_{(1^n,m)}}$ is the extended Losev-Manin moduli space $\overline{\EuScript{L}}_{0;n,m}$ of stable curves with $n$ white and $m$ black marked points~\cite{losev2004extended}.
\end{itemize}
\end{example}

\subsubsection{Rooted trees and Stratification}
Similarly, to the operad case, the Wonderful contractad admits a stratification by rooted trees. However, in contrast to the operad case, we need to restrict ourselves to admissible rooted trees~\cite[Sec.~1.3]{lyskov2023contractads}.

A \emph{rooted tree} is a connected directed tree $T$ in which each vertex has at least one input edge and exactly one output edge. This tree should have exactly one external outgoing edge, output. The endpoint of this edge is called the \emph{root}. The endpoints of incoming external edges that are not vertices are called \emph{leaves}. A tree with a single vertex is called a \emph{corolla}. For a rooted tree $T$ and edge $e$, let $T_e$ be the subtree with the root at $e$ and $T^e$ be the subtree obtained from $T$ by cutting $T_e$.

For a graph $\Gr$, a $\Gr$-\emph{admissible rooted tree} is a rooted tree $T$ with leaves labeled by the vertex set $V_{\Gr}$ of the given graph such that, for each edge $e$ of the tree, the leaves of the subtree $T_e$ form a tube of $\Gr$. Below, we list examples and a counterexample of $\Cyc_4$-admissible trees.
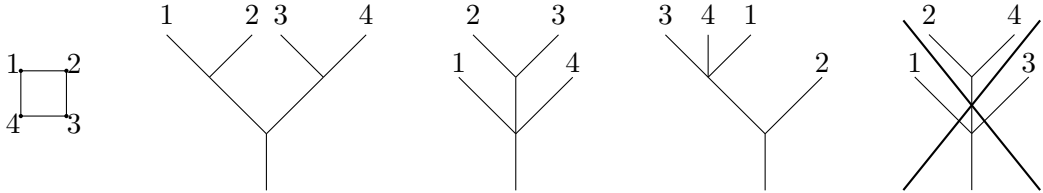
\[
\vcenter{\hbox{\begin{tikzpicture}[scale=0.4]
    \fill (0,0) circle (2pt);
    \fill (0,1.5) circle (2pt);
    \fill (1.5,0) circle (2pt);
    \fill (1.5,1.5) circle (2pt);
    \draw (0,0)--(1.5,0)--(1.5,1.5)--(0,1.5)-- cycle;
    \node at (-0.25,1.75) {$1$};
    \node at (1.75,1.75) {$2$};
    \node at (1.75,-0.25) {$3$};
    \node at (-0.25,-0.25) {$4$};
  \end{tikzpicture}}}
\qquad
\vcenter{\hbox{\begin{tikzpicture}[scale=0.75]
        \draw (0,0)--(0,1);
        \draw (0,1)--(1,2);
        \draw (0,1)--(-1,2);
        \draw (1,2)--(1.75,2.75);
        \draw (1,2)--(0.25,2.75);
        \draw (-1,2)--(-1.75,2.75);
        \draw (-1,2)--(-0.25,2.75);
        \node at (1.75,3.1) {$4$};
        \node at (0.25,3.1) {$3$};
        \node at (-1.75,3.1) {$1$};
        \node at (-0.25,3.1) {$2$};
  \end{tikzpicture}}}
\qquad 
\vcenter{\hbox{\begin{tikzpicture}[scale=0.75]
        \draw (0,0)--(0,1);
        \draw (0,1)--(1,2);
        \draw (0,1)--(-1,2);
        \draw (0,1)--(0,2);
        \draw (0,2)--(0.75,2.75);
        \draw (0,2)--(-0.75,2.75);
        \node at (-1,2.25) {$1$};
        \node at (-0.75,3.1) {$2$};
        \node at (0.75,3.1) {$3$};
        \node at (1,2.25) {$4$};
  \end{tikzpicture}}}
\qquad 
\vcenter{\hbox{\begin{tikzpicture}[scale=0.75]
        \draw (0,0)--(0,1);
        \draw (0,1)--(1,2);
        \draw (0,1)--(-1,2);
        \draw (-1,2)--(-1.75,2.75);
        \draw (-1,2)--(-1,2.75);
        \draw (-1,2)--(-0.25,2.75);
        \node at (-1.75,3.1) {$3$};
        \node at (-1,3.1) {$4$};
        \node at (-0.25,3.1) {$1$};
        \node at (1,2.25) {$2$};
  \end{tikzpicture}}}
\qquad 
\vcenter{\hbox{\begin{tikzpicture}[scale=0.75]
        \draw[thick] (-1.2,3)--(1.2,0);
        \draw[thick] (-1.2,0)--(1.2,3);
        \draw (0,0)--(0,1);
        \draw (0,1)--(1,2);
        \draw (0,1)--(-1,2);
        \draw (0,1)--(0,2);
        \draw (0,2)--(0.75,2.75);
        \draw (0,2)--(-0.75,2.75);
        \node at (-1,2.25) {$1$};
        \node at (-0.75,3.1) {$2$};
        \node at (0.75,3.1) {$4$};
        \node at (1,2.25) {$3$};
  \end{tikzpicture}}}
\] 
Note that the rightmost tree is not admissible since vertices $2$ and $4$ do not form a tube. We denote by $\Tree$ the corresponding graphical collection of \emph{admissible rooted trees}. Similarly to the operads, this collection assembles a contractad with respect to graftings of rooted trees. 

For a $\Gr$-admissible tree $T$ and a vertex $v\in \Ver(T)$, we define the \emph{input graph} $\In(v)$ as follows: its vertices are input edges $e_1,\cdots,e_k$ of $v$ and two vertices $e_i,e_j$ are adjacent if the corresponding union of leaf sets $L_{e_i}\cup L_{e_j}$ forms a tube in $\Gr$, where $L_{e_i}$ is the set of leaves of the subtree $T_{e_i}$. Similarly to operads, for a graphical collection $\E$, the free contractad $\T(\E)$ generated by $\E$ has an explicit combinatorial description: its elements are admissible rooted trees decorated by elements of $\E$
\[
\T(\E)(\Gr)=\bigoplus_{T \in \Tree(\Gr)} \E((T)), \text{ where } \E((T)):=\bigotimes_{v \in \Ver(T)} \E(\In(v)).
\] 
In particular, $\Tree$ is a free contractad generated by corollas. We also consider the free subcontractad of stable rooted trees $\Tree_{\mathrm{st}}\subset \Tree$ consisting of rooted trees whose all internal vertices have at least $2$ input edges.

The Wonderful contractad $\bM$ admits a contractad stratification by rooted trees. The component-wise inclusion $\M(\Gr)\hookrightarrow\bM(\Gr)$ produces a morphism of contractads $\T(\M)\to \bM$, where $\T(\M)$ is the free contractad in the category of sets on $\M$\footnote{we use convention $\M(\Path_1)=\varnothing$}. This morphism is an isomorphism of set-contractads. Moreover, for a stable $\Gr$-admissible tree $T$, the image of $\M((T))$ in $\bM(\Gr)$ is a locally open subset, whose closure is the intersection of canonical divisors $\bigcap_{v\in \Ver(T)} D_{\Leav(v)}$. All in all, there is a locally open stratification of $\bM(\Gr)$
\begin{equation}\label{eq::stratification}
 \bM(\Gr)=\bigsqcup_{T\in \Tree_{\mathrm{st}}(\Gr)}\M((T)), \quad \M((T))\cong \prod_{v\in \Ver(T)}\M(\In(v)).   
\end{equation}
\begin{example}
A stable rational $(n+1)$-pointed curve  $\Sigma$ determines a dual rooted tree $T_{\Sigma}$ whose vertices are components of the subvariety of smooth points of $\Sigma$, edges are double points, and leaves are marked points $[n]\cup\{\infty\}$, $\infty$ corresponds to root.
    \[
        \vcenter
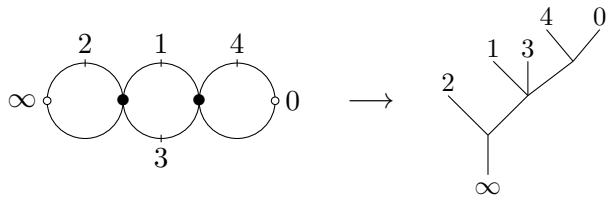
{\hbox{\begin{tikzpicture}[scale=0.5,]
        \draw (0,0) circle (1);
        \draw (2,0) circle (1);
        \draw (4,0) circle (1);
        \node at (1,0) {$\bullet$};
        \node at (3,0) {$\bullet$};
        \node [left] at (-1,0) {$\infty$};
        \draw [fill=white] (-1,0) circle (0.1);
        \node [right] at (5,0) {$0$};
        \draw [fill=white] (5,0) circle (0.1);
        \node [above] at (0,1) {$2$};
        \draw (0,1.1) -- (0,0.9);
        \node [above] at (2,1) {$1$};
        \draw (2,0.9) -- (2,1.1);
        \node [below] at (2,-1) {$3$};
        \draw (2,-0.9) -- (2,-1.1);
        \node [above] at (4,1) {$4$};
        \draw (4,0.9) -- (4,1.1);
        \end{tikzpicture}}}
    \quad
    \longrightarrow
    \quad
    \vcenter{\hbox{\begin{tikzpicture}[scale=0.7]
        \draw (0,0)--(0,0.75);
        \draw (0,0.75)--(-0.75,1.5);
        \draw (0,0.75)--(0.75,1.5);
        \draw (0.75,1.5)--(0.1,2.15);
        \draw (0.75,1.5)--(0.75,2.15);
        \draw (0.75,1.5)--(1.6,2.15);
        \draw (1.6,2.15)--(1.1,2.75);
        \draw (1.6,2.15)--(2.1,2.75);
        \node at (0,-0.25) {\small$\infty$};
        \node at (-0.75,1.75) {\small$2$};
        \node at (0.1,2.4) {\small$1$};
        \node at (0.75,2.4) {\small$3$};
        \node at (1.1,3) {\small$4$};
        \node at (2.1,3) {\small$0$};
  \end{tikzpicture}}}
\]
For a complete multipartite graph $\K_{\lambda}$, the strata element $\eM_{0,\K_{\lambda}}((T))$ consists of $\K_{\lambda}$-pointed stable curves with a dual tree $T$.
\end{example}

\subsection{Fulton-MacPherson contractad}
\label{sec:sub:FM}
Suppose we replace the complex vector spaces in the construction of the wonderful contractad with the real ones. In that case, we obtain the so-called Fulton-MacPherson compactification for the configuration spaces $\Conf_{\Gr}(\mathbb{R}^d)$ based on the spherical compactifications~\cite{gaiffi2003models}.

Let $\Dil(\mathbb{R}^d):=\mathbb{R}^d\rtimes \mathbb{R}_{>0}\subset \Aff_d(\mathbb{R})$ be the subgroup of dilations of $\mathbb{R}^d$ (compositions of translations and homotheties). The factor space $\Conf_{\Gr}(\Real^d)/\Dil(\mathbb{R}^d)$ is identified with the subspace of normalised configurations $\NConf_{\Gr}(\Real^d)\subset \Conf_{\Gr}(\Real^d)$ consisting of configurations with zero barycenter and unit radius
\[
b(x)=\frac{1}{|V_{\Gr}|}\cdot \sum_{v\in V_{\Gr}} x_v=0,\quad r(x)=\max_{V_{\Gr}} (\Vert  x_v\Vert)=1.
\] For a vector space $V$, let $\mathbb{S}(V):=V/\mathbb{R}_{>0}$ be its spherization. For a set $S$, similarly to the complex case, we define the vector space $W^d_S:=(\mathbb{R}^d)^{S}/\triangle(\mathbb{R}^d)$, where $\triangle(\mathbb{R}^d)$ is the diagonal subspace. For a connected graph $\Gr$, we consider the subspace arrangement $\B_d(\Gr)=\{H^d_e|e\in E_{\Gr}\}$ in $W^d_{\Gr}$ consisting of subspaces
\[
H^d_{vw}=\{(x_d^{(i)})^{i=1,\cdots,d}_{v\in V_{\Gr}}| x^{(i)}_v=x^{(i)}_w\text{ for all }i=1,\cdots, d\}\subset W^d_{V_{\Gr}},
\] and, for a tube $G$, the subspace $H_G=\cap_{e\subset G} H_e$. We consider the inclusion of $\NConf_{\Gr}(\Real^d)$ in the product of spheres
\[
\rho\colon \NConf_{\Gr}(\mathbb{R}^d)\cong \Conf_{\Gr}(\mathbb{R}^d)/\Dil(\Real^d)\rightarrow \prod_{G\subset V_{\Gr}\text{-tube}}\Sph(W^d_G),
\] where $\rho=\prod_G \rho_G$ is the product of projections $\rho_G\colon \Conf_{\Gr}(\mathbb{R}^d)/\Dil(\Real^d)\to \Sph(W^d_{\Gr}/H^d_G)\cong \Sph(W^d_G)$. We define the Fulton-MacPherson compactification as the closure of $\rho$ and denote it by $\FM_d(\Gr)$. Similarly to the projective case~\eqref{eq::contractad_struc_cmplx}, we define the contractad structure
\[
\circ^{\Gr}_G\colon \FM_d(\Gr/G)\times\FM_d(\Gr|_G)\to \FM_d(\Gr)
\] by examining the projections 
\begin{equation}\label{eq::contractad_struc_real}
 \rho_K \circ^{\Gr}_G=\begin{cases}
\rho_{K} \mathrm{pr}_{\Gr|_G}\text{, for }K\subset G, \\
\Sph(j^K_G)\rho_{K/G}\mathrm{pr}_{\Gr/G}\text{, for }K\not\subset G,
\end{cases}
\end{equation} where the maps $i^K_G$ and $j^K_G$ are defined in the same way as in~\eqref{eq::contractad_struc_cmplx}.

The FM compactification $\FM_d(\Gr)$ is the spherical wonderful compactification associated with a graphical building set $\G^d_{\Gr}:=\{H^d_G|G\text{- tube}\}$ introduced by Gaiffi in~\cite{gaiffi2003models}. From his results, we immediately get 
\begin{theorem}\cite[Theorem~4.2, Theorem~5.1, Theorem~5.2]{gaiffi2003models}
$\FM_d(\Gr)$ is a compact manifold with corners of dimension $d(|V_{\Gr}|-1)-1$. Moreover, the boundary $\partial\FM_d(\Gr)$ is the union of submanifolds $D_G:=\circ^{\Gr}_G(\FM_d(\Gr/G)\times \FM_d(\Gr|_G))$ of codimension $1$. The intersection of submanifolds $D_{\Susp}:=\cap_{G\in \Susp} D_G$ is non-empty and transversal if and only if $\Susp$ is a nested set.
\end{theorem} 

The component-wise inclusion $\NConf(\Real^d)\hookrightarrow\FM_d$ produces an isomorphism of set-contractads $\T(\NConf(\Real^d))\to \FM_d$. For a rooted tree $T\in\Tree_{\mathrm{st}}(\Gr)$, we denote by $\FM_d((T))$ the corresponding strata element. The closure of $\FM_d((T))$ is the intersection of canonical divisors $\overline{\FM_d((T))}=D_{\Susp(T)}=\cap_{v\in \Ver(T)} D_{\Leav(v)}$.
\begin{remark} Recall that initially, the ordinary FM compactifications were constructed as the closures of the images~\cite{axelrod1992chern}
\begin{gather*}
\Conf_n(\Real^d)/\Dil(\Real^d)\overset{s_{ij}, d_{ijk}}{\longrightarrow} \prod_{(i,j)} S^{d-1}\times\prod_{(i,j,k)} [0,+\infty], 
\\ 
\text{ here }\quad s_{ij}(x)=\frac{x_i-x_j}{|x_i-x_j|},\quad d_{ijk}(x)=\frac{|x_i-x_j|}{|x_i-x_k|}.
\end{gather*} 
These compactifications are equivalent to the spherical ones~\cite[Sec.~6.3]{gaiffi2003models}. However, for an arbitrary $\Gr$, it is not clear when $\FM_d(\Gr)$ admits such analog construction in terms of angles $s_{ij}$ and relative distances $d_{ijk}$, see~\cite[Ex.~3.1]{galashin2021p} for a counter-example in the case $d=1$ and $\Gr=\Path_4$.
\end{remark}
For a graph $\Gr$, the action of the special orthogonal $\SO(d)$ on the spheres $\Sph(W_{\Gr}^d/H^d_e)\cong S^{d-1}$ associated with edges, induces the action on the space $\prod_G \Sph(W_{\Gr}/H_G)$, that restricts to $\FM_d(\Gr)$. This action is compatible with the contractad structure, so $\FM_d$ forms a contractad in the category of $\SO(d)$-spaces.
\begin{definition}
The framed Fulton MacPherson contractad $\fFM_d$ is the semidirect product contractad $\FM_d\rtimes \SO(d)$.
\end{definition}

\subsection{Fulton-MacPherson contractad is $E_d$-contractad}
\label{sec:sub:FM_n=E_n}
 An equivalence of topological contractads $\Pop,\Q$ is a morphism of contractads $\phi\colon\Pop\to\Q$ that is a component-wise weak homotopy equivalence. Two contractads $\Pop$ and $\Q$ are said to be equivalent if they are related by a zigzag of equivalences.
\begin{definition}
A $E_d$-contractad is a contractad equivalent to the little $d$-disks contractad $\D_d$. Analogously, a framed $E_d$-contractad is a contractad equivalent to the framed little $d$-disks contractad $\fD_d$
\end{definition} 

Similarly to the operad case, we define the Boardman-Vogt resolution~\cite{boardman2006homotopy} for contractads as follows. For a topological contractad $\Pop$, its Boardman-Vogt resolution is the contractad $\mathsf{W}\Pop$ obtained from the free contractad $\T(\Pop)$ by in addition
\begin{itemize}
    \item labeling the inner edges of any admissible rooted tree by real numbers $l\in[0,1]$,
    \item identifying rooted trees one of whose edges has length $0$ with the rooted tree with that edge contracted and with the corresponding contracted-element labels composed,
    \item equipping the set of labeled trees with the corresponding quotient of the product topology, to make it into a topological space,
    \item contractad structure is given by grafting of trees and letting the new edges be of length $1$.
\end{itemize} There is an evident inclusion $\T(\Pop)\to \mathsf{W}\Pop$ obtained by regarding each edge of a tree as being of length $1$, and there is an evident factorization of $\T(\Pop)\to \Pop$ through $\mathsf{W}\Pop$.
\begin{lemma}\label{lemma::boardmanvogt} For a topological operad, the canonical projection $\pi_{\Pop}\colon\mathsf{W}\Pop\to \Pop$ is an equivalence of contractads.
\end{lemma}
\begin{proof} $\pi_{\Pop}$ is indeed the retraction with respect to the evident inclusion $\Pop\subset \mathsf{W}\Pop$ (it is not a contractad morphism), so it suffices to construct a homotopy $\phi_{t}\colon \Id\simeq \pi_{\Pop}$. The desired homotopy $\phi_t$ is given by scaling edge lengths by $t$.
\end{proof} 

\begin{theorem}\label{thm::FMn_is_En}
The (framed) Fulton-MacPherson contractad is a (framed) $E_d$-contractad
\[
\FM_d\simeq \D_d,\quad \fFM_d\simeq \fD_d.
\]
\end{theorem}
\begin{proof}
Thanks to~\cite{gaiffi2003models}, the compactification $\FM_d(\Gr)$ is obtained from $(\mathbb{R}^d)^{V_{\Gr}}$ by a sequence of real blow-ups along diagonal subspaces. 
The real blow up along the plane is homotopy equivelent to the complement of this plane.
Therefore, the map $\Conf_{\Gamma}(\mathbb{R}^d)\to \FM_d(\Gr)$ is a homotopy equivalence. Combining with the centering map $\D_d(\Gr)\overset{\simeq}{\to}\Conf_{\Gamma}(\mathbb{R}^d)$~\eqref{eq::centering_map}, we obtain a component-wise homotopy equivalence of graphical collections
\[
r\colon\D_d(\Gamma)\to \FM_d(\Gamma).
\] However, it is not compatible with the contractad compositions. Similarly to~\cite[Prop.~3.9]{salvatore2001configuration}, we can lift $r\colon \D_d\to \FM_d$ to the contractad morphism $\mathcal{R}\colon \mathsf{W}\D_d\to \FM_d$. Thanks to Lemma~\ref{lemma::boardmanvogt}, since $\mathsf{W}\D_d$ is $E_d$-contractad and $\R$ is an equivalence, it follows that $\FM_d$ is $E_d$-contractad.  The framed case repeats the unframed ones, see~\cite[Prop~4.11]{salvatore2001configuration} for instance.
\end{proof}
\subsection{Compactifications for manifolds and right modules}  In this section, we construct and study the graphical (framed) Fulton-MacPherson compactifications for a compact manifold using spherical blow-ups.

For a submanifold $Y\subset X$ (with corners), the spherical blow-up is the manifold $\Bl_Y(X)$ with a surjective map $p\colon \Bl^{\Real}_Y(X)\to X$ that is a diffeomorphism over the complement $X\setminus Y$ and is the spherical normal bundle $p\left|_Y\right.\colon \Sph(N_XY)\to Y$ over $Y$. Let $M$ be a compact Riemannian manifold. For a set $S$, let $\Bl^{\mathbb{R}}_{\triangle}M^{S}$ be the spherical blow-up of $M^S$ along a small diagonal $M\cong \triangle(M)\subset M^{S}$. Recall that, for a small diagonal $\triangle(M)\subset M^S$, its normal bundle is $TM^{\oplus S}/\triangle(TM)$, so the exceptional divisor $\E_{\triangle}=p^{-1}(\triangle(X))$ is $\Sph(TM^{\oplus S}/\triangle(TM))$.

For a graph $\Gr$, we consider the inclusion of graphical configuration space
\[
\Conf_{\Gr}(M)\to M^{V_{\Gr}}\times \prod_{G\subset V_{\Gr}-tube} \mathrm{Bl}^{\mathbb{R}}_{\triangle} M^{\times G},
\] and we let $\FM_{M}(\Gr)$ be the closure of the image.
\begin{proposition}
\label{thm::FM:M}
For a compact manifold $M$, the compactification $\FM_M(\Gr)$ is a compact manifold with corners of dimension $\dim M\cdot |V_{\Gr}|$.
\end{proposition}
We leave this result without proof, although technically it is quite simple to obtain it using Gaiffi's spherical approach~\cite{gaiffi2003models} and arguments from the projective case~\cite{li2009wonderful}. Since $\FM_M(\Gr)$ is obtained from $M^{V_{\Gr}}$ by a sequence of real blow-ups along diagonal subspaces, we have the following.
\begin{corollary}\label{sled::conf_to_FM} The inclusion $\Conf_{\Gr}(M)\hookrightarrow \FM_M(\Gr)$ is a homotopy equivalence.
\end{corollary} The boundary $\partial\FM_{M}(\Gr)=\FM_M(\Gr)\setminus \Conf_{\Gr}(M)$ is a union of codimension $1$ submanifolds $D_G$ indexed by tubes $G\subset V_{\Gr}$ (including $G=V_{\Gr}$) and each divisor is diffeomorphic to the product $D_G\cong \FM_{M}(\Gr/G)\times \FM_d(\Gr|_G)$, where $n=\dim M$.

Recall that, for a (Riemannian) parallelizable compact manifold, its ordinary (framed) FM compactifications form a right module over the (framed) FM operad ~\cite{markl1996compactification}. A similar statement holds in the graphical case. For a contractad $\Pop$, a right $\Pop$-module is a graphical collection $\M$ endowed with a product map
\[
\gamma_{\M}\colon \M\circ\Pop\to \M,
\] that is compatible with a contractad structure.
\begin{proposition}\label{prop::parallisable} If $M$ is an $n$-dimensional parallelizable manifold, then $\FM_M$ admits a natural structure of a right module over the contractad $\FM_d$.
\end{proposition}
\begin{proof}
For a parallelizable manifold $M$, we consider the trivialization of sphere bundle $\Sph(TM^{\oplus S}/\triangle(M))\cong M\times \Sph(W_S^d)$ with respect to the fixed trivialization $TM\cong M\times \mathbb{R}^d$. For a tube $G$, we define the morphism $\circ^{\Gr}_G\colon\FM_M(\Gr/G)\times \FM_d(\Gr|_G)\to \FM_M(\Gr)$ by specializing $\rho_{K}\circ^{\Gr}_G$, where $\rho_K\colon \FM_M(\Gr)\to \Bl_{\triangle}M^K$, for each tube $K$. If $K\subset G$, we let $\phi^K_G\colon M\times \Sph(W_K) \hookrightarrow \Bl_{\triangle} M^K$, where the last map is the inclusion of an exceptional divisor. Otherwise, we consider the inclusion $\psi^K_G\colon \Bl_{\triangle}M^{K/G}\hookrightarrow  \Bl_{\triangle}M^K$ induced from the map $\iota\colon M^{K/G}\to M^K$ given by $\iota((x_w)_{w\in K/G})_v=x_{v/G}$, and the map $\Sph(j^K_G)$ on the fibers over the small diagonal. So, we define $\circ_G^{\Gr}$ by
\[
\rho_K\circ^{\Gr}_G=\begin{cases}
\phi^K_G(\pi_{\{G\}}\times \rho_K)\text{, for }K\subset G,
\\
\psi^K_G\rho_{K/G}\text{, for }K\not\subset G,
\end{cases},
\] where $\pi_{\{G\}}\colon \FM_M(\Gr)\to M$ is the projection on $\{G\}$-factor.
\end{proof}For a Riemannian manifold, we consider the frame bundle $\mathrm{Fr}_M\to M$. For a graph $\Gr$, the framed configuration space $\Conf_{\Gr}^{\fr}(M)$ and its framed Fulton-MacPherson compactification are the pullbacks of the diagrams below.
\[
\begin{tikzcd}
  \Conf^{\fr}_{\Gr}(M) \arrow[dashed]{d} \arrow[dashed]{r}
    & \mathrm{Fr}^{V_{\Gr}}_M \arrow{d}{\xi} \\
  \Conf^{\fr}_{\Gr}(M) \arrow{r}& M^{V_{\Gr}} \end{tikzcd}, \quad \begin{tikzcd}
  \fFM_M(\Gr) \arrow[dashed]{d} \arrow[dashed]{r}
    & \mathrm{Fr}^{V_{\Gr}}_M \arrow{d}{\xi} \\
  \FM_M(\Gr) \arrow{r}& M^{V_{\Gr}} \end{tikzcd}
\] Explicitly, $\Conf^{\fr}_{\Gr}(M)$ consists of graphical configurations of points with additional choices of framings at marked points. Thanks to Corollary~\ref{sled::conf_to_FM}, the inclusion $\Conf^{\fr}_{\Gr}(M)\hookrightarrow\fFM_M(\Gr)$ is a homotopy equivalence. \begin{proposition}
For a Riemannian manifold $M$ of dimension $d$, $\fFM_M$ admits a natural structure of a right module over the contractad $\fFM_d$.
\end{proposition}
\begin{proof} The idea of the construction of $\fFM_d$-module structure repeats the ones of Proposition~\ref{prop::parallisable} with the only modification that we replace the trivialisation with the choice of framings $ \Fr_M\times\Sph(W_S)\to \Sph(TM^{\oplus S}/\triangle(TM))$.
\end{proof}

\subsection{One-dimensional case and poset associahedra}
\label{sec::polytopes} 
Recall that the ordinary Fulton-MacPherson compactifications $\FM_1(k)$ are homeomorphic to the disjoint union of the famous Stasheff polytopes~\cite{stasheff1963homotopy} whose faces are labeled by planar rooted trees. We present a straightforward generalisation to the graphical case.

Consider a graphical configuration spaces $\Conf_{\Gr}(\mathbb{R})$. This space is a disjoint union of convex areas diffeomorphic to $\mathbb{R}^{|V_{\Gr}|}$, in particular, it is homotopy equivalent to a finite number of points. An acyclic direction of a graph is an assignment of direction for each edge, such that there are no directed cycles $v_1\to v_2\to v_3\to\cdots\to v_n\to v_1$. Every acyclic direction $\alpha$ of $\Gr$ defines a connected component $\Conf^{\alpha}_{\Gr}(\mathbb{R})$ that consists of configurations $(x_v)_{v\in V_{\Gr}}$ with $x_v>x_w$ for a directed edge $v\to w$ in $\alpha$. All connected components are of such form
\[
\Conf_{\Gr}(\mathbb{R})=\coprod_{\alpha\in \Ass(\Gr)}\Conf^{\alpha}_{\Gr}(\mathbb{R}),
\] where $\Ass(\Gr)$ is a set of acyclic directions. In a similar way, for $\alpha\in\Ass(\Gr)$, let $\FM^{\alpha}_1(\Gr)$ be the corresponding connected component of $\FM_1(\Gr)$. 

The contractad structure on $\D_1$ induces the contractad structure on the graphical $\Ass:=\pi_0(\D_1)$. Explicitly, the contractad structure is given by the substitution of acyclic directions. As an example, for a graph $\Gr=\vcenter{\hbox{\begin{tikzpicture}[scale=0.3]
    \fill (0,0) circle (2pt);
    \fill (0,1.5) circle (2pt);
    \fill (1.5,0) circle (2pt);
    \fill (1.5,1.5) circle (2pt);
    \draw (0,0)--(1.5,0)--(1.5,1.5)--(0,1.5)-- cycle;
    \draw (0,0)->(1.5,1.5);
    \node at (-0.25,1.75) {\scriptsize$1$};
    \node at (1.75,1.75) {\scriptsize$2$};
    \node at (1.75,-0.25) {\scriptsize$3$};
    \node at (-0.25,-0.25) {\scriptsize$4$};
    \end{tikzpicture}}}$ and tube $G=\{2,4\}$, we have 
    \begin{gather*}
    \hbox{\begin{tikzpicture}[scale=0.6, edge/.style={->,> = latex, thick}]
    \fill (0,0) circle (2pt);
    \fill (1.5,0) circle (2pt);
    \fill (3,0) circle (2pt);
    \node at (0,0.4) {$1$};
    \node at (1.5,0.45) {\small$\{2,4\}$};
    \node at (3,0.4) {$3$};
    \draw[edge] (1.5,0)--(0,0);
    \draw[edge] (1.5,0)--(3,0);
    \end{tikzpicture}}
    \circ^{\Gr}_{G}
    \hbox{\begin{tikzpicture}[scale=0.6, edge/.style={->,> = latex, thick}]
    \fill (0,0) circle (2pt);
    \fill (1.5,0) circle (2pt);
    \node at (0,0.4) {$4$};
    \node at (1.5,0.4) {$2$};
    \draw[edge] (0,0)--(1.5,0);
    \end{tikzpicture}}=
    \vcenter{\hbox{\begin{tikzpicture}[scale=0.6, edge/.style={->,> = latex, thick}]
    \fill (0,0) circle (2pt);
    \fill (0,1.5) circle (2pt);
    \fill (1.5,0) circle (2pt);
    \fill (1.5,1.5) circle (2pt);
    \draw[edge] (1.5,1.5)--(0,1.5);
    \draw[edge] (1.5,1.5)--(1.5,0);
    \draw[edge] (0,0)--(1.5,0);
    \draw[edge] (0,0)--(0,1.5);
    \draw[edge] (0,0)--(1.5,1.5);
    \node at (-0.25,1.75) {$1$};
    \node at (1.75,1.75) {$2$};
    \node at (1.75,-0.25) {$3$};
    \node at (-0.25,-0.25) {$4$};
    \end{tikzpicture}}},\quad 
    \hbox{\begin{tikzpicture}[scale=0.6, edge/.style={->,> = latex, thick}]
    \fill (0,0) circle (2pt);
    \fill (1.5,0) circle (2pt);
    \fill (3,0) circle (2pt);
    \node at (0,0.4) {$1$};
    \node at (1.5,0.45) {\small$\{2,4\}$};
    \node at (3,0.4) {$3$};
    \draw[edge] (0,0)--(1.5,0);
    \draw[edge] (1.5,0)--(3,0);
    \end{tikzpicture}}
    \circ^{\Gr}_{G}
    \hbox{\begin{tikzpicture}[scale=0.6, edge/.style={->,> = latex, thick}]
    \fill (0,0) circle (2pt);
    \fill (1.5,0) circle (2pt);
    \node at (0,0.4) {$4$};
    \node at (1.5,0.4) {$2$};
    \draw[edge] (1.5,0)--(0,0);
    \end{tikzpicture}}=
    \vcenter{\hbox{\begin{tikzpicture}[scale=0.6, edge/.style={->,> = latex, thick}]
    \fill (0,0) circle (2pt);
    \fill (0,1.5) circle (2pt);
    \fill (1.5,0) circle (2pt);
    \fill (1.5,1.5) circle (2pt);
    \draw[edge] (0,1.5)--(1.5,1.5);
    \draw[edge] (1.5,1.5)--(1.5,0);
    \draw[edge] (0,0)--(1.5,0);
    \draw[edge] (0,1.5)--(0,0);
    \draw[edge] (1.5,1.5)--(0,0);
    \node at (-0.25,1.75) {$1$};
    \node at (1.75,1.75) {$2$};
    \node at (1.75,-0.25) {$3$};
    \node at (-0.25,-0.25) {$4$};
    \end{tikzpicture}}}
\end{gather*}

For an acyclic direction $\alpha\in \Ass(\Gr)$, we define the poset $P_{\alpha}$ on the vertex set with order given by the transitive closure of the relation $\{v>w| v\to w\text{ in } \alpha\}$. Note that, for every finite poset $Q$, we have $Q=P_{\mathrm{Hasse}(Q)}$, where $\mathrm{Hasse}(Q)$ is the Hasse diagram of $Q$ considered as an acyclic directed graph. In~\cite{galashin2021p}, Galashin constructed the compactification of configuration space associated with a poset. In the case $P_{\alpha}$, this configuration space is the quotient space $\NConf^{\alpha}_{\Gr}(\Real)\cong\Conf^{\alpha}_{\Gr}(\Real)/\Real\rtimes \Real_{>0}$. The Galashin's compactification of this space that we denote $\mathcal{K}_{\alpha}(\Gr)$ is a $(|V_{\Gr}|-2)$-dimensional convex hull polytope called a \emph{poset polytope} or \emph{poset associahedron}. 

Let us briefly discuss the combinatorics of these polytopes. For an acyclic direction $\alpha$ of $\Gr$, a tube $G$ is called a \emph{pipe} if, for every directed path with ends in $G$, the remaining vertices of the path are also in $G$. Given a pipe $G$, we have the well-defined acyclic direction $\alpha/G$ of $\Gr/G$ such that $\alpha=\alpha/G\circ^{\Gr}_G \alpha|_G$. Recall that the faces of the poset polytope $\mathcal{K}_{\alpha}(\Gr)$ are products of smaller poset polytopes~\cite{galashin2021p}. Explicitly, a cofaces $F_G$ of $\mathcal{K}_{\alpha}(\Gr)$ are indexed by proper nontrivial pipes $G$ and isomorphic to the products $F_G=\mathcal{K}_{\alpha/G}(\Gr/G)\times \mathcal{K}_{\alpha|_G}(\Gr|_G)$. The remaining faces are the intersections of cofaces indexed by nested sets $\Susp$ of pipes $F_{\Susp}:=\cap_{G\in \Susp} F_G$~\cite[Lemma~3.13]{galashin2021p}. We refer to~\cite{sack2023realization} for an explicit realization of $\mathcal{K}_{\alpha}(\Gr)$.

\begin{theorem}
\label{thm::D1=As1}
For a graph $\Gr$, the connected components $\FM^{\alpha}_1(\Gr)$ of FM compactification $\FM_1(\Gr)$ are poset polytopes $\mathcal{K}_{\alpha}(\Gr)$.
\end{theorem}
\begin{proof} From the contractad isomorphism $\pi_0(\FM_1)\cong \Ass$ we have 
$$\circ^{\Gr}_G(\FM^{\alpha}_1(\Gr/G)\times \FM^{\beta}_1(\Gr|_G))\subset \FM_1^{\alpha\circ^{\Gr}_G\beta}(\Gr).$$ 
Therefore, the  rooted trees stratification of $\FM_1$ induces the stratification of its connected components
\begin{equation}\label{eq::prod_FM1}\FM^{\alpha}_1(\Gr)=\coprod_{T\in \Tree^{\alpha}(\Gr)} \FM_1((T)),\quad  \FM_1((T)) \cong \prod_{v\in \Ver(T)} \NConf_{\In(v)}^{\alpha_v}(\Real),  
\end{equation} where $\Tree^{\alpha}(\Gr)$ is the pre-image of $\alpha$ under the evaluation map $\pi_{\Ass}\colon\T(\Ass)\to \Ass$. Also, the boundary of $\FM^{\alpha}_1(\Gr)$ is the union of the divisors $D_G^{\alpha}=\circ^{\Gr}_G(\FM^{\alpha/G}_1(\Gr/G)\circ\FM^{\alpha|_G}_1(\Gr|_G))$ indexed by pipes $G$. In particular, the closure of a strata element is the intersection of such divisors $\overline{\FM_1((T))}=\cap_{e\in\Edge(T)} D^{\alpha}_{\Leav(e)}$.

For $\alpha\in \Ass(\Gr)$, we have a bijection between elements of $\Tree^{\alpha}(\Gr)$ and nested sets of pipes with respect to $\alpha$. Indeed, for a monomial $T\in \Tree^{\alpha}(\Gr)$ the set $\{ \Leav(e)|e\in \Edge(T)\}$ defines a nested set of pipes. The inverse construction is clear. Under this identification, due to~\cite[Prop.~3.11]{galashin2021p} and decomposition~\eqref{eq::prod_FM1}, the interiors of faces of $\mathcal{K}_{\alpha}(\Gr)$ are diffeomorphic to the corresponding strata elements of $\FM^{\alpha}_1(\Gr)$. Since the faces of $\mathcal{K}_{\alpha}(\Gr)$ and the closures of strata of  $\FM^{\alpha}_1(\Gr)$ satisfy the same inclusion relations, the manifolds $\mathcal{K}_{\alpha}(\Gr)$ and $\FM^{\alpha}_1(\Gr)$ are diffeomorphic.
\end{proof}
 
\begin{example} Let us consider the examples of the corresponding polytopes for particular types of graphs
\begin{itemize}
    \item For a chain $1\to 2\to \cdots \to n$, the corresponding poset polytope is a $(n-2)$-dimensional Stasheff polytope $\mathcal{K}_{n-2}$. Hence, $\FM_1(\K_n)$ is a disjoint union of $n!$ copies of polytopes $\mathcal{K}_{n-2}$.
    \item More generally, if the Hasse diagram of a poset $P_{\alpha}$ is a tree, then the polytope $\mathcal{K}_{\alpha}$ is the graph associahedra associated with the line graph\footnote{A line graph $\La(\Gr)$ is the graph with vertex set $E_{\Gr}$ and $e$ and $e'$ are adjacent if they share a vertex in $\Gr$} $\La(P_{\alpha})$ of $P_{\alpha}$~\cite[Sec.~2.3.1]{sack2023realization}. So, for a tree $T$ on $n$ vertices, the component $\FM_1(T)$ is a disjoint union of $2^{n-1}$ copies of graph associahedra associated with $\La(T)$. In particular
    \begin{itemize}
        \item For $\Path_n$, $\FM_1(\Path_n)$ is disjoint union of $2^{n-1}$ copies of Stasheff polytopes $\mathcal{K}_{n-2}$.
        \item For $\St_n$, $\FM_1(\St_n)$ is disjoint union of $2^{n}$ copies of Permutohedrons $\mathcal{W}_{n-1}$.
    \end{itemize} 
\begin{figure}[ht]
    \centering
\begin{tikzpicture}[x=1.18mm,y=1mm,scale=0.6]
    \path[dashed, color=black!50!white]
    (-16,1) edge (-5.5,5)
    (16,1) edge (5.5,5)
    (0,-19) edge (0,-5)
    (0,20) edge (0,15)
    (0,-5) edge (-5.5,5)
    (-5.5,5) edge (0,15)
    (0,-5) edge (5.5,5)
    (5.5,5) edge (0,15)
    ;
    \path 
    (-7.3,-1.6) edge (7.3,-1.6)
    (7.3,-1.6) edge (11.5,12)
    (11.5,12) edge (0,20)
    (-7.3,-1.6) edge (-11.5,12)
    (-11.5,12) edge (0,20)
    (7.3,-1.6) edge (12,-11)
    (-12,-11) edge (-7.3,-1.6)
    (-12,-11) edge (0,-19)
    (0,-19) edge (12,-11)
    (-12,-11) edge (-16,1)
    (-16,1) edge (-11.5,12)
    (12,-11) edge (16,1)
    (11.5,12) edge (16,1)
    ;
\end{tikzpicture}\qquad
\begin{tikzpicture}[scale=0.5, z={(0,0,.5)}, line join=round]
\foreach [var=\x, var=\y, var=\z, count=\n] in {
2/1/0,2/0/1,2/-1/0,2/0/-1,
1/0/2,0/1/2,-1/0/2,0/-1/2,
1/2/0,0/2/1,-1/2/0,0/2/-1,
-2/1/0,-2/0/1,-2/-1/0,-2/0/-1,
1/0/-2,0/1/-2,-1/0/-2,0/-1/-2,
1/-2/0,0/-2/1,-1/-2/0,0/-2/-1
}{\coordinate (n\n) at (\x,\y,\z);}
\draw (n1)--(n2)--(n3)--(n4)--cycle;
\draw (n5)--(n6)--(n7)--(n8)--cycle;
\draw (n9)--(n10)--(n11)--(n12)--cycle;
\draw (n21)--(n22)--(n23);
\draw (n13)--(n14)--(n15);
\draw (n6)--(n10);
\draw (n2)--(n5);
\draw (n8)--(n22);
\draw (n15)--(n23);
\draw (n7)--(n14);
\draw (n11)--(n13);
\draw (n1)--(n9);
\draw (n3)--(n21);
\end{tikzpicture}
    \caption{Stasheff polytope $\mathcal{K}_3$ and Permutohedron $\mathcal{W}_3$}
    \label{fig:enter-label}
\end{figure}
    \item For an arbitrary graph $\Gr$, the corresponding polytopes are non-isomorphic in general. For example, for the cycle $\Cyc_4$, $\FM_1(\Cyc_4)$ consists of $8$ copies of $5$-gons, $4$ copies of $6$-gons, and $2$ copies of $8$-gons associated with acyclic directions of type  $\vcenter{\hbox{\begin{tikzpicture}[scale=0.35, edge/.style={->,> = latex, thick}]
    \fill (0,0) circle (2pt);
    \fill (0,1.5) circle (2pt);
    \fill (1.5,0) circle (2pt);
    \fill (1.5,1.5) circle (2pt);
    \draw[edge] (0,1.5)--(1.5,1.5);
    \draw[edge] (1.5,1.5)--(1.5,0);
    \draw[edge] (0,0)--(1.5,0);
    \draw[edge] (0,0)--(0,1.5);
    \end{tikzpicture}}}$,  \space $\vcenter{\hbox{\begin{tikzpicture}[scale=0.35, edge/.style={->,> = latex, thick}]
    \fill (0,0) circle (2pt);
    \fill (0,1.5) circle (2pt);
    \fill (1.5,0) circle (2pt);
    \fill (1.5,1.5) circle (2pt);
    \draw[edge] (0,1.5)--(1.5,1.5);
    \draw[edge] (1.5,1.5)--(1.5,0);
    \draw[edge] (0,0)--(1.5,0);
    \draw[edge] (0,1.5)--(0,0);
    \end{tikzpicture}}}$, and $\vcenter{\hbox{\begin{tikzpicture}[scale=0.35, edge/.style={->,> = latex, thick}]
    \fill (0,0) circle (2pt);
    \fill (0,1.5) circle (2pt);
    \fill (1.5,0) circle (2pt);
    \fill (1.5,1.5) circle (2pt);
    \draw[edge] (0,1.5)--(1.5,1.5);
    \draw[edge] (1.5,0)--(1.5,1.5);
    \draw[edge] (1.5,0)--(0,0);
    \draw[edge] (0,1.5)--(0,0);
    \end{tikzpicture}}}$ respectively.
\end{itemize}
\end{example}

\begin{remark} For a circle $S^1$, the ordinary FM compactification $\FM_{S^1}(n)$ is diffeomorphic to $(n-1)!$ copies of $S^1\times \C_{n-2}$ the product of a circle and $(n-2)$-dimensional cyclohedron $\C_{n-2}$~\cite{bott1994self, stasheff1996operads}. We suspect that a similar story holds in the graphical case, and the corresponding polytopes are the affine versions of poset polytopes introduced by Galashin~\cite{galashin2021p}.
\end{remark}




\section{Chordal graphs and cycles}
\label{sec::chordal_cycles}

For a graph $\Gr$, its complex configuration space $\Conf_{\Gr}(\Cmplx)$ is a complement to the graphic complex hyperplane arrangement $\{\{x_v-x_w\}\}_{(v,w)\in E_{\Gr}}$ and hence it is a formal topological space~\cite[Thm.~4.4]{yuzvinsky2001orlik}. So, the rational homotopy type of $\Conf_{\Gr}(\Cmplx)$ is fully-determined by its Orlik-Solomon algebra $\OS(\Gr)=H^{\bullet}(\Conf_{\Gr}(\Cmplx))$. But even in this simple case, the homotopical properties of $\Conf_{\Gr}(\Cmplx)$ vary significantly depending on the type of underlying graph. In particular, the ranks of rational homotopy groups of these spaces are unclear for an arbitrary graph.

So, instead of working with all graphs, we can consider special families of graphs closed under contractions and induced subgraphs, as in the case of complete multipartite graphs and modular compactifications from Section~\ref{sec::modular}. In this section, we show that the homotopical properties of graphical configuration spaces for manifolds vary significantly depending on the type of underlying graph. The main theses of this section are
\begin{itemize}
    \item[(i)] \textit{Most homotopical properties of the little disks operads and configuration spaces hold in the contractad setting if we restrict ourselves to the class of chordal graphs.}
    \item[(ii)] \textit{The obstructions appear already in cycles $\Cyc_n$, for $n\geq 4$.}
\end{itemize}

\subsection{Restriction to chordal graphs}\label{sec::restriction_to_chordal} We briefly recall some combinatorics of chordal graphs and explain the restriction of contractads to chordal graphs.
\begin{definition}
    A graph $\Gr$ is called \emph{chordal} if every cycle with four or more vertices contains an extra edge joining two non-consecutive vertices of that cycle\footnote{ Such an edge is called a chord, which explains the name chordal graphs.}.
\end{definition}
We will denote by $\mathrm{Chord}$ the corresponding class of chordal graphs. 
\begin{claim*}
The class of chordal graphs $\mathrm{Chord}$ is closed under contracted and induced graphs.   
\end{claim*}
\begin{proof}
If a tube $G$ contains a vertex set $V_{\Cyc}$ of a subcycle $\Cyc\subset \Gr$, then the induced subgraph $\Gr|_G$ contains all chords of $\Cyc$ appearing in $\Gr$. So, chordal graphs are closed under induced subgraphs. Also, for a chordal graph $\Gr$, a subcycle $\Cyc\subset \Gr$ with at least $4$ vertices, and an edge $e\in E_{\Cyc}$, the subcycle $\Cyc/e\subset \Gr/e$ is either a $3$-cycle or contains a chord. So, chordal graphs are closed under edge-contractions and hence under contractions by tubes.
\end{proof} 

We call a \emph{chordal collection} a contravariant functor $\mathsf{CGr}|_{\mathrm{Chord}}^{\mathrm{op}}\to \C$ from the full subgroupoid of connected chordal graphs. Thanks to the claim above, the contracted product~\eqref{eq::contract::product} is well-defined for chordal collections such that the forgetful functor from graphical collections to chordal ones is monoidal. These lead us to the following
\begin{itemize}
    \item We define a \emph{chordal contractad} as a monoid in the category of chordal collections. All related constructions, such as Bar/Cobar constructions, Koszul property, are defined in the same way as for contractads. Since the forgetful functor is exact and monoidal, it preserves all needed constructions and properties.
    \item For an usual contractad, by "... in the class of chordal graphs", we always mean the corresponding chordal contractad obtained by applying the forgetful functor.
\end{itemize}

A remarkable property that distinguishes chordal graphs from the class of all graphs is related to so-called simplicial vertices, that will be
used later.

\begin{definition}[\cite{rose1970triangulated}] \hfill\break
\begin{itemize}[topsep=0pt]
    \item[(i)] 
    A vertex $v$ of a graph is called \emph{simplicial} if its neighborhood 
    $\N(v)=\{w \mid (w,v)\in E_{\Gr}\}\cup \{v\}$ forms a \emph{clique}, i.e., it induces a complete subgraph.
    \item[(ii)] 
    A \emph{perfect elimination ordering} is a total ordering of the vertices such that 
    $\mathcal{N}(v_i)\cap \{v_1,v_2,\dots, v_i\}$ is a clique for all $i$. 
    \item[(iii)] A graph $\Gr$ admits a perfect elimination ordering if and only if it is chordal.
\end{itemize}
\end{definition}

\subsection{Chordal graphs and Configuration spaces}
\label{sec:sub:chordal} 
 Recall that the ordinary configuration space $\Conf_n(\mathbb{C})$ is a classifying space for the pure braid group $\Conf_n(\mathbb{C})\simeq K(\PB_n,1)$. This group is generated by "weave" braids $T_{ij}=\vcenter{\hbox{\begin{tikzpicture}[line width=0.8, scale=0.4]
\draw[line width=1]
  (0,-1) to (0,1);
\draw[line width=2pt, white]
  (-1, -1)
    --
  (-1, -.8)
    .. controls (-1,-.4) and (0.15, -.4) ..
  (0.15, 0)
    .. controls (0.15, +.4) and (-1, +.4) ..
  (-1,.8)
    --
  (-1,1);
\draw[line width=1] (-1, -1)--(-1, -.8) .. controls (-1,-.4) and (0.15, -.4) .. (0.15, 0) .. controls (0.15, +.4) and (-1, +.4) .. (-1,.8) -- (-1,1);
\draw[line width=2, white]
  (0,0) to (0,1);
\draw[line width=1]
  (0,0) to (0,1);
\node at (-1,-1.3) {\scriptsize$i$};
\node at (0,-1.3) {\scriptsize$j$};
\end{tikzpicture}}}$, for $1\leq i<j\leq n$, satisfying the relations
\begin{gather}
[T_{pq},T_{rs}]=[T_{ps},T_{qr}] = 1\text{ for }p<q<r<s
\\
T_{pr}T_{qr}T_{pq} =T_{qr}T_{pq}T_{pr}=T_{pq}T_{pr}T_{qr} \text{ for }p<q<r
\\
T_{rs}T_{pr}(T_{rs})^{-1}T_{qs} = 1\text{ for }p<q<r<s
\end{gather} For a graph $\Gr$, the graphical pure braid group is the fundamental group of the graphical configuration space
\[
\PB_{\Gr}:=\pi_1(\Conf_{\Gr}(\mathbb{C})).
\] We refer to~\cite{cohen2021graphic} for a description and basic properties of these groups. For the inclusion $\Gr\to \K_n$ of a graph $\Gr$ into the complete one $\K_n$ on the same vertex set, the corresponding inclusion $\Conf_n(\mathbb{C})\to \Conf_{\Gr}(\mathbb{C})$ produces a surjective morphism $\PB_n\twoheadrightarrow \PB_{\Gr}$ and braids $t_{vw}$ are trivial in the image for non-adjacent $v$ and $w$.  Therefore, $\PB_{\Gr}$ is obtained from $\PB_n$ by adding relations $T_{ij}=e$ for $(i,j)\not \in E_{\Gr}$. 

Recall that, for a manifold $M$, the forgetful map $\Conf_n(M)\to \Conf_{n-1}(M)$ is a fibration with the fiber $M\setminus \{(n-1)\text{ points}\}$~\cite{fadell1962configuration}. This statement has the following generalization in the graphical case:

\begin{proposition}\label{prop::fiber_arguement}
For a manifold $M$ and graph $\Gr$ with a simplicial vertex $v$, the forgetting map
\[
\Conf_{\Gr}(M)\to \Conf_{\Gr\setminus v}(M)
\] is a fibration with the fiber $M\setminus \{\deg_{\Gr}(v)\text{ points}\}$. Moreover, this fibration admits a cross-section.
\end{proposition}
\begin{proof}
Completely analogous to~\cite[Thm.~1]{fadell1962configuration}.
\end{proof}
\begin{remark} Note that the simplicial condition is required. For the cycle $\Cyc_4$ and vertex $3$, fibers of the forgetting map $\Conf_{\Cyc_4}(\mathbb{C})\to \Conf_{\Path_3}(\mathbb{C})$ are not even homotopy equivalent: over configurations with $x_2=x_4$ and $x_2\neq x_4$ the fibers are $\mathbb{C}\setminus \{\text{one point}\}$ and $\mathbb{C}\setminus \{\text{two points}\}$ respectively. 
\end{remark}

Consequently, thanks to Proposition~\ref{prop::fiber_arguement},
for a chordal graph $\Gr$ with elimination ordering $v_1,\cdots,v_n$ we have a sequence of fibrations with cross sections
\[
\Conf_{\Gr}(\mathbb{C})\overset{\mathbb{C}\setminus \{k_n \text{ points }\}}{\longrightarrow}\Conf_{\Gr\setminus v_n}(\mathbb{C})\overset{\mathbb{C}\setminus \{k_{n-1} \text{ points}\}}{\longrightarrow}\Conf_{\Gr\setminus \{v_n,v_{n-1}\}}(\mathbb{C})\rightarrow\cdots\longrightarrow\mathbb{C},
\] where $k_i=|\{j<i| (v_i,v_j)\in E_{\Gr}\}|$. Thanks to the equivalence $\mathbb{C}\setminus \{n\text{ points}\}\simeq \bigvee^n_{i=1} S^1=K(F_n,1)$, we have

\begin{proposition}
Let $\Gr$ be a chordal graph on $n\geq 2$ vertices, then
\begin{itemize}
    \item[(i)] The configuration space $\Conf_{\Gr}(\mathbb{C})$ is $K(\PB_{\Gr},1)$ space.
    \item[(ii)] The pure braid group is isomorphic to the iterated semiproduct of free groups
\[
    \PB_{\Gr}\cong F_{k_{n}}\rtimes F_{k_{n-1}}\rtimes \cdots F_{k_3}\rtimes F_{1}
\] where $k_i=|\{j<i| (v_i,v_j)\in E_{\Gr}\}|$ and $v_1,\cdots,v_n$ is a perfect elimination ordering.
    \item[(iii)] For $d\geq 3$, the graphical configuration spaces are simply-connected, and one has a splitting of abelian groups \[
    \pi_{\bullet}(\Conf_{\Gr}(\mathbb{R}^d))\cong \bigoplus^{n}_{i=2} \pi_{\bullet}((S^{d-1})^{\vee k_i})
    \]
\end{itemize}
\end{proposition}



\begin{proposition} 
For a chordal graph $\Gr$, the configuration space $\Conf_{\Gr}(\mathbb{C})$ is a rational $K(\pi,1)$-space. Its holonomy Lie algebra $\mathfrak{t}(\Gr)$ is generated by $\{t_{vw} : (v,w)\in E_{\Gr}\}$ subject to the relations:
\begin{gather*}
  [t_{vw},t_{wu}]=[t_{wu},t_{uv}]=[t_{uv},t_{vw}], \text{ for 3-cycles } \{v,w,u\}, \\
  [t_{vw},t_{vu}]=0, \text{ if } (w,u)\not\in E_{\Gr}, \\
  [t_{vw},t_{us}]=0, \text{ for pairwise distinct vertices } \{u,v,w,s\}.
\end{gather*}
\end{proposition}
\begin{proof} 
A formal space is a rational $K(\pi,1)$-space if its cohomology ring is generated in degree one and is Koszul~\cite[Thm~6.33]{yuzvinsky2001orlik}. As established in Section~\ref{sec::chordal_cycles}, $\Conf_{\Gr}(\Cmplx)$ is formal for any graph. For a chordal graph, the associated graphic arrangement is supersolvable~\cite{stanley1972supersolvable}, which ensures that the Orlik-Solomon algebra is Koszul~\cite[Cor.~6.21]{yuzvinsky2001orlik}. Thus, $\Conf_{\Gr}(\Cmplx)$ is a rational $K(\pi,1)$-space whose holonomy Lie algebra $\mathfrak{t}(\Gr)$ is the Koszul dual Lie algebra to $\OS(\Gr)$.
\end{proof}

The aforementioned holonomy Lie algebra admits a natural generalization to higher dimensions $d \geq 3$, where we assign the degree of the generators $t_{uv}$ to be $2 - d$.
This generalized version is called the \emph{Drinfeld–Kohno Lie algebra} $\mathfrak{t}_d(\Gamma)$ and describes the rational homotopy type of the graphical configuration space $\Conf_{\Gr}(\mathbb{R}^d)$.
One can easily verify the following generalization of the classical statement:

\begin{proposition}
For any chordal graph $\Gamma$ and any connected subset of vertices $G$, the following assignment
$$
t_{uv} \mapsto
\left[\begin{array}{l}
t_{uv}, \quad \text{if } u,v \in V(\Gamma)\setminus G, \\
\displaystyle\sum_{(uw)\in E(\Gamma)} t_{uw}, \quad \text{if } u \in V(\Gamma)\setminus G \text{ and } v = \{G\} \in V(\Gamma/G), \\
t_{uv}, \quad \text{if } u,v \in G = V(\Gamma|_G)
\end{array}\right.
$$
extends to a well-defined map of Lie algebras
$$
\mathfrak{t}_d(\Gamma/G) \oplus \mathfrak{t}_d(\Gamma|_G) \longrightarrow \mathfrak{t}_d(\Gamma).
$$
The collection $\{\mathfrak{t}_d(\Gamma) \mid \Gamma \text{ chordal}\}$ assembles into a contractad in the category of Lie algebras, which we call the Drinfeld–Kohno Lie contractad $\mathfrak{t}_d$.
This contractad provides a Lie–algebraic model for the little discs contractad $\mathcal{D}_d$ for $d\geq 2$ and for chordal graphs.
\end{proposition}

\subsection{Formality in the Chordal Case}
\label{sec:sub:chordal:formality}
In this section, all constructions---including semi-algebraic sets and dg (co)algebras---are defined over the field of real numbers $\Real$.

We follow the proof of the formality of the little disks operads suggested in~\cite{kontsevich1999operads}. The proof relies on the construction of the dg Hopf cooperad of graphs $\mathsf{Graphs}_d$ and a chain of quasi-isomorphisms of cooperads:
\begin{equation}
\label{eq::formality::Kontsevich}
\begin{tikzcd}
H^{\bullet}(\FM_d) \arrow[r,equal] & \OS_d &
\mathsf{Graphs}_d \arrow[l,"\mathbb{I}"'] \arrow[r,"\hat{\mathbb{I}}"] &
\Omega_{PA}^{\bullet}(\FM_d).
\end{tikzcd}
\end{equation}
Here, $\OS_d$ denotes the Hopf cooperad of Orlik-Solomon algebras (see~\eqref{eq::OS::composition}) and $\Omega_{PA}$ is the functor of piecewise algebraic forms which is applicable to semi-algebraic sets, specifically $\FM_d$. We refer to~\cite{hardt2011real} for the general theory of semi-algebraic sets and the $\Omega_{PA}$ functor, and to~\cite{lambrechts2014formality} for a comprehensive treatment of the formality of the little disks operad in all dimensions.

It is important to note that the functor of piecewise algebraic forms $\Omega_{PA}$ is not comonoidal in the strict sense. However, via the K\"unneth quasi-isomorphism $\Omega_{PA}(X) \otimes \Omega_{PA}(Y) \overset{\simeq}{\to} \Omega_{PA}(X \times Y)$, the cocomposition morphism $\circ^{\Gr}_G$ in the topological contractad/operad $\Orb$ can be represented by the following diagram of dg algebras, which defines the notion of a \emph{weak contractad/operad}:
\begin{equation}
\label{eq::weak::contractad}
\begin{tikzcd}
\Omega_{PA}(\Orb(\Gr)) \arrow[r, "{(\circ^{\Gr}_G)^*}"] & \Omega_{PA}(\Orb(\Gr/G)\times \Orb(\Gr|_G)) & \arrow[l,"\text{quis}"'] \Omega_{PA}(\Orb(\Gr/G))\otimes \Omega_{PA}(\Orb(\Gr|_G)).
\end{tikzcd}
\end{equation}
We refer to~\cite[Sec~3]{lambrechts2014formality} for the precise definition of weak operads and the corresponding formulation of formality~\eqref{eq::formality::Kontsevich}. Furthermore, several works~\cite{fresse2017homotopy,fresse2018extended} provide methods for replacing these weak structures with strict ones.

The main goal of this section is to show that the objects involved in the formality chain~\eqref{eq::formality::Kontsevich} can be extended to the context of \textbf{chordal graphs}. To achieve this, we adapt the following notions from the operadic setting to the case of contractads:
\begin{itemize}[itemsep=0pt, topsep=0pt]
\item The Kontsevich cooperad $\Graphs_d$ (see~\cite[Sec.~6]{lambrechts2014formality});
\item The projection $\mathbb{I}$ onto the cohomology;
\item The fibered integral map $\hat{\mathbb{I}}$ to the space of piecewise algebraic forms (\cite[Sec.~4]{lambrechts2014formality}).
\end{itemize}

\subsubsection{The Contractads $\Graphs_d$ and $\mathsf{ICG}_d$}
We begin by briefly recalling the construction of the Kontsevich cooperad $\Graphs_d$ (see~\cite[Sec.~6]{lambrechts2014formality}). For a set $V$, an \emph{admissible diagram} $D$ is a finite undirected graph with two types of vertices: external vertices (labeled by $V$) and internal vertices (unlabeled). These must satisfy the following: each internal vertex has valence at least $3$, and each connected component contains at least one external vertex. Below is an example for $V=[4]$:
\[
\begin{tikzpicture}
\node[ext] (v1) at (0,0) {$\scriptscriptstyle 1$};
\node[ext] (v2) at (0.5,0) {$\scriptscriptstyle 2$};
\node[ext] (v3) at (1,0) {$\scriptscriptstyle 3$};
\node[ext] (v4) at (1.5,0) {$\scriptscriptstyle 4$};
\node[int] (w1) at (0.5,.7) {};
\node[int] (w2) at (1.0,.7) {};
\draw (v1) edge (v2) edge (w1) (v2) edge (w1) edge (w2) (v3) edge (w1) edge (w2) (v4) edge (w2) (w1) edge (w2);
\end{tikzpicture}
\]
Internal vertices carry cohomological degree $-d$, while edges carry degree $d-1$. We further consider orientations and orderings of edges and internal vertices. Let $\Graphs_d(V)$ be the linear span of admissible diagrams modulo the following symmetry relations:
\[
\begin{array}{cl}
\D=(-1)^d \D' & \text{if } \D' \text{ is obtained from } D \text{ by inverting an edge or transposing internal vertices;} \\
\D=(-1)^{d+1} \D' & \text{if } \D' \text{ is obtained from } \D \text{ by transposing the edge ordering.}
\end{array}
\]
For even $d$, diagrams with multiple edges vanish; for odd $d$, diagrams with loops vanish.

The collection $\Graphs_d(m)_{m\geq 0}$ forms a dg Hopf cooperad. The differential is given by edge contraction, and the cocomposition is defined combinatorially by restriction and contraction:
\[
\begin{array}{c}
\circ^{*} : \Graphs_d(V\sqcup W) \longrightarrow \Graphs_d(V\sqcup\{*\}) \otimes \Graphs_d(W) \\
\circ^{*}(\D):= \sum_{\D'\subset \D} \D/\D' \otimes \D'|_{W},
\end{array}
\]
where the sum is over connected subdiagrams $\D'\subset \D$ whose external vertices belong to $W$. Multiplication is given by the union of graphs at external vertices. Each component $\Graphs_d(m)$ is a quasi-free cdga generated by the subspace of internally connected diagrams (graphs that remain connected if external vertices are removed). We have an isomorphism:
$$\Graphs_d(m)\cong C_{CE}^{\bullet}(\mathsf{ICG}_d(m))$$ 
between the dgca of graphs and the Chevalley--Eilenberg complex of the $L_{\infty}$-algebra of internally connected diagrams $\mathsf{ICG}_d$.\footnote{In the literature, such diagrams are referred to as graphs, which explains the letter “G” in the abbreviation ICG. We retain this abbreviation, but we do not refer to the elements as "graphs" so as not to confuse them with graphs used for the contractad structure.}

\begin{definition}
\label{def::Graphs::Gamma}
For a graph $\Gr$, we define $\Graphs_d(\Gr) \subset \Graphs_d(V_{\Gr})$ as the subspace generated by admissible diagrams $D$ satisfying the following property: 
\begin{equation}
\label{eq::graphs::contractad::condition}
\begin{varwidth}{0.85\linewidth}
If $D$ contains a path $v - u_1- \cdots - u_k - w$ between external vertices $v, w \in V_{\Gr}$ through internal vertices $u_i$, then $v$ and $w$ must be adjacent in $\Gr$.
\end{varwidth}
\end{equation}
In particular, two external vertices can be adjacent in the diagram only if they are adjacent in the underlying graph $\Gr$.
\end{definition}

\begin{lemma}
The subspace $\Graphs_d(\Gr)$ is closed under the differential and forms a quasi-free subalgebra generated by internally connected diagrams. The corresponding $L_{\infty}$-algebra, denoted $\mathsf{ICG}_d(\Gr)$, is a quotient of $\mathsf{ICG}_d(V_{\Gr})$:
\[
\imath_{\Gr}:\Graphs_d(\Gr) \hookrightarrow \Graphs_d(V_{\Gr}) \quad \text{and} \quad \pi_{\Gamma}: \mathsf{ICG}_d(V_{\Gr}) \twoheadrightarrow \mathsf{ICG}_d(\Gr).
\]
\end{lemma}
\begin{proof}
Direct verification on the level of graphs.
\end{proof}

\begin{lemma}
The collection $\Graphs_d(\Gr)$ forms a Hopf dg-cocontractad, with cocomposition obtained by restricting the cocomposition of the cooperad $\Graphs_d$. Correspondingly, the collection $\mathsf{ICG}_{d}(\Gr)$ forms a contractad of $L_{\infty}$-algebras.
\end{lemma}
\begin{proof}
Let $\D \in \Graphs_d(\Gr)$ satisfy Assumption~\eqref{eq::graphs::contractad::condition}. This property is inherited by any subdiagram $\D' \subset \D$. Thus, if $D' \subset D$ is a subdiagram whose external vertices belong to a tube $G$, then $\D'|_{G} \in \Graphs_d(\Gr|_{G})$. Similarly, the contracted diagram $\D/\D'$ satisfies the assumption for the contracted graph $\Gr/G$. It follows that the following diagram commutes:
\begin{equation}
\label{eq::diagram::graphs::composition}
\begin{tikzcd}
    \Graphs_d(\Gr) \ar[rr,"(\circ_{G}^{\Gamma})^{*}"] \arrow[d,hook,"\imath_{\Gamma}"] & & \Graphs_d(\Gr/G) \otimes \Graphs_d(\Gr|_G) \arrow[d,hook,"\imath_{\Gamma/G}\otimes\imath_{\Gamma|_G}"] \\
    \Graphs_d(V_{\Gr}) \ar[rr,"\circ^{*}"] & & \Graphs_d(V_{\Gr/G}) \otimes \Graphs_d(V_{\Gr|_G}).
\end{tikzcd}
\end{equation}
The coassociativity of the cocomposition in $\Graphs_d(\Gamma)$ follows from that of the cooperad $\Graphs_d$. The contractad structure on $\mathsf{ICG}_d(\Gamma)$ arises from the Koszul duality between $\Graphs_d$ and $\mathsf{ICG}_d$.
\end{proof}

\begin{lemma}
\label{lem::ICG::subgraph}
For an internally connected diagram $\D \in \mathsf{ICG}_d(\Gamma)$, the set of external vertices $G_{\D} \subseteq V_{\Gamma}$ that are the ends of the edges in $\D$ forms a clique, i.e., the induced subgraph $\Gamma|_{G_{\D}}$ is complete.
\end{lemma}
\begin{proof}
Let $v, w \in G_{\D}$ be any two external vertices connected to $\D$. Since $\D$ is internally connected, there exists a path in $\D$ between $v$ and $w$ such that all intermediate vertices are internal. 
By Definition~\ref{def::Graphs::Gamma}, this implies that $v$ and $w$ must be adjacent in $\Gamma$. As this condition holds for every pair of vertices in $G_{D}$, the induced subgraph $\Gamma|_{G_{D}}$ is complete.
\end{proof}

\subsubsection{Cohomology of $\Graphs_d$}
Consider the map of algebras:
\begin{equation}
\label{eq::Graphs=OS}
\mathbb{I}\colon \Graphs_d(\Gr) \longrightarrow \OS_d(\Gr) 
\end{equation}
defined on generators $\D \in \mathsf{ICG}_d$ by the following rule:
$$
\begin{cases}
\mathbb{I}(D)= 0, & \text{if $D$ contains an internal vertex,} \\
\mathbb{I}\left(
\begin{tikzpicture}[scale=0.5, baseline=-0ex]
	\node[ext] (v0) at (-0.5,0) {$\scriptstyle v$};
	\node[ext] (v1) at (0.5,0) {$\scriptstyle w$};
	\draw (v0) edge (v1);
\end{tikzpicture}
\right) := \omega_{vw}. & 
\end{cases}
$$
Kontsevich proved that for complete graphs $\K_n$, the map $\mathbb{I}$ is a quasi-isomorphism of cooperads. In Theorem~\ref{prop::Graphs_to_OS}, we extend this result to the case of chordal graphs. Note that for non-chordal graphs, $\mathbb{I}$ is not a quasi-isomorphism.

The proof of Theorem~\ref{prop::Graphs_to_OS} follows the strategy of~\cite[Thm.~8.1]{lambrechts2014formality}, using an induction argument based on a perfect elimination ordering of the vertices and maybe easily explained via the description of the normal forms (NBC-monomials) for Orlik--Solomon algebras $\OS_d(\Gamma)$. To avoid an explicit discussion of NBC-monomials, we use combinatorial data derived from the generating series:

\begin{lemma}
For a graph $\Gr$ with a simplicial vertex $v$, we have:
\begin{equation}\label{eq::mobius_chordal}
\dim \Lie(\Gr) = \deg(v) \cdot \dim \Lie(\Gr\setminus v).
\end{equation}
\end{lemma}
\begin{proof}
Recall that the graded generating series of the Orlik--Solomon algebra $\OS_d(\Gamma)$ is determined by the chromatic polynomial of $\Gamma$ (see~\cite[\S 2.3]{khoroshkin2024hilbert}). Due to the isomorphism~\eqref{eq::Ed::graphical} of graphical collections, we have:
$$\dim \Lie(\Gamma) = |\mathfrak{X}_{\Gamma}(0)|.$$
When a vertex $v$ is simplicial, all its neighbors are adjacent to each other. Consequently, in any proper coloring of $\Gamma$, the only restriction on the color of $v$ is that it must differ from the colors of its $\deg(v)$ neighbors (which are already distinct from one another). This leads to the following recurrence for chromatic polynomials:
\[
\mathfrak{X}_{\Gr}(q) = (q - \deg(v)) \mathfrak{X}_{\Gr\setminus v}(q).
\] 
Setting $q=0$ we obtain:
$$\dim \Lie(\Gamma) = |\mathfrak{X}_{\Gamma}(0)| = \deg(v) \cdot |\mathfrak{X}_{\Gr\setminus v}(0)| = \deg(v) \cdot \dim \Lie(\Gr\setminus v).$$
\end{proof}

We now state the main result of this subsection:

\begin{theorem}
\label{prop::Graphs_to_OS} 
For a chordal graph $\Gr$, the morphism $\mathbb{I}$ in~\eqref{eq::Graphs=OS} is a quasi-isomorphism compatible with the cocontractad structure.
\end{theorem}
\begin{proof}
For a chordal graph $\Gamma$ with $m$ vertices, the set of generators for $\OS_d(\Gamma)$ corresponds to the edges of $\Gamma$. Thanks to the definitions of $\mathbb{I}$ and $\Graphs_d$ we have the following commutative diagram:
$$
\begin{tikzcd}
    \Graphs_d(\Gamma) \arrow[r, two heads, "\mathbb{I}"] \arrow[d, hook, "\imath_{\Gamma}"] & \OS_{d}(\Gamma) \arrow[d, hook] \\
    \Graphs_{d}(m) \arrow[r, two heads, "\mathbb{I}"] & \OS_d(m).
\end{tikzcd}
$$
The vertical arrows are embeddings, and the horizontal arrows are projections. Since the bottom row is a morphism of complexes~\cite[Thm.~8.1]{lambrechts2014formality}, the top row is as well. Furthermore, the commutativity of Diagram~\eqref{eq::diagram::graphs::composition}, alongside the analogous properties for Orlik--Solomon algebras, implies that $\mathbb{I}$ is compatible with the cocomposition maps for chordal graphs.

It remains to show that $\mathbb{I}$ is a quasi-isomorphism. Let $\Graphs_d^{\mathrm{con}} \subset \Graphs_d$ be the subcomplex of connected admissible diagrams. Since the differential in $\Graphs$ preserves the set of connected components, we have a dg-isomorphism of graphical collections:
\[
\Graphs_d \cong \Com \circ \Graphs_d^{\mathrm{con}}.
\] 
By the K\"unneth isomorphism, $H^{\bullet}(\Graphs_d) \cong \Com \circ H^{\bullet}(\Graphs^{\mathrm{con}}_d)$. Similarly, for Orlik--Solomon algebras:
\[
\OS_d \cong \Com \circ \OS^{\mathrm{top}}_d \cong \Com \circ \Susp^{d-1}(\Lie)^*,
\]
where $\OS_d^{\mathrm{top}}$ is the sub-cocontractad concentrated in the top cohomological degree $(d-1)(|V_{\Gr}|-1)$. It suffices to show $H^{\bullet}(\Graphs^{\mathrm{con}}_d) \cong \OS^{\mathrm{top}}_d$, which we establish by induction on $|V_{\Gr}|$.

Fix a simplicial vertex $v \in V_{\Gr}$ and let $\widetilde{\Graphs}^{\mathrm{con}}_d(\Gr) \subset \Graphs^{\mathrm{con}}_d(\Gr)$ be the submodule generated by diagrams where $v$ is adjacent to exactly one external vertex. Following~\cite[Lemma~8.5]{lambrechts2014formality}, this inclusion is a quasi-isomorphism. This subcomplex decomposes as:
\[
\widetilde{\Graphs}^{\mathrm{con}}_d(\Gr) \cong \bigoplus_{(w,v) \in E_{\Gr}}  \Graphs^{\mathrm{con}}_d(\Gr\setminus v).
\]
Taking homology and using the induction hypothesis alongside Equations~\eqref{eq::mobius_chordal}, we obtain:
\[
\dim H^{\bullet}(\Graphs^{\mathrm{con}}_d(\Gr)) = \deg(v) \cdot \dim H^{\bullet}(\Graphs^{\mathrm{con}}_d(\Gr\setminus v)).
\]
This matches the dimension of $\OS^{\mathrm{top}}_d(\Gamma)$, confirming that $\mathbb{I}$ is a quasi-isomorphism.
\end{proof}

\begin{remark}
The proof of Theorem~\ref{prop::Graphs_to_OS} relies heavily on the chordal assumption. For the cycle $\Cyc_4$, for instance, $H^{\bullet}(\Graphs(\Cyc_4)) \not\cong \OS^{\bullet}_d(\Cyc_4)$.
\end{remark}

\subsubsection{Piecewise algebraic forms and Kontsevich integrals}
Let us outline several technical geometric statements about Fulton-Macferson compactifications.

For a graph $\Gr$, the configuration space $\Conf_{\Gr}(\Real^d)$ is a complement of a real affine space to a subspace arrangement, hence $\Conf_{\Gr}(\Real^d)$ is a semi-algebraic manifold in the sense of~\cite{bochnak2013real}. By similar reasons, the subspace $\NConf_{\Gr}(\Real^d)\subset \Conf_{\Gr}(\Real^d)$ is a semi-algebraic manifold. Recall that the Fulton-Macferson compactification $\FM_d(\Gr)$ is a closure of a semi-algebraic set $\NConf_{\Gr}(\Real^d)$ in the product of spheres that are also semi-algebraic sets, hence $\FM_d(\Gr)$ is a semi-algebraic set. Moreover, from the theory of wonderful compactifications we know that $\FM_d(\Gr)$ is obtained from the sphere $\Sph(W_{\Gr})\cong S^{d(|V_{\Gr}|-1)-1}$ by a sequence of spherical blow-ups along divisors indexed by tubes, see~\cite[p.~7]{de1995wonderful} for an analogous statement in the projective case, hence $\FM_d(\Gr)$ is indeed a semi-algebraic manifold.

For an inclusion of graphs $\Gr\subset \Omega$, the corresponding map of configuration spaces $\Conf_{\Omega}(\Real^d)\to \Conf_{\Gr}(\Real^d)$ lifts to the map of their Fulton-Macferson compactifications given by the commutative diagram
\[\begin{tikzcd}
	{\FM_d(\Omega)} & {\prod_{G\subset V_{\Omega}\text{-tube}} \mathbb{S}(W_G)} \\
	{\FM_d(\Gamma)} & {\prod_{H\subset V_{\Gamma}\text{-tube}} \mathbb{S}(W_H)}
	\arrow[from=1-1, to=1-2]
	\arrow[from=1-1, to=2-1]
	\arrow["{\mathrm{pr}}", from=1-2, to=2-2]
	\arrow[from=2-1, to=2-2]
\end{tikzcd},\] and this map is semi-algebraic since it is a composition of a closed inclusion and a projection. In the special case $V_{\Gr}=V_{\Omega}$, adopting the technique of iterated projective blow-ups~\cite[Prop~2.10]{rains2010homology} to the spherical case, we see that $p\colon \FM_d(\Omega)\to \FM_d(\Gr)$ is a sequence of spherical blow-ups indexed by tubes of $\Omega$ that are not tubes for $\Gr$.

We briefly recall the construction of Kontsevich's morphism $\hat{\mathbb{I}}\colon \Graphs_d(n)\to \Omega_{PA}(\FM_d(n))$, where $\Omega_{PA}$ is the de Rham complex of piecewise algebraic forms in the sense of~\cite[Sec.~5.4]{hardt2011real}. We fix a standard top-degree volume form 
$$\eta^{d-1}:=\kappa_d \sum_{i=1}^{d} x_i dx_1\wedge \ldots \wedge\widehat{dx_i}\wedge \ldots\wedge dx_n$$ 
on $\FM_d(2)\cong S^{d-1}\subset \mathbb{R}^{d}$. Here $x_i$ are standard coordinates in $\mathbb{R}^{d}$ and $\kappa_d$ is the constant such that $\int_{S^{d-1}}\eta^{d-1} = 1$. For an admissible diagram $\D\in \Graphs_d(n+m)$ with $m$ internal vertices, we define a semi-algebraic differential form $\eta_{\D}\in \Omega^{\bullet}_{PA}(\FM_d(n))$ by fiber integration:
\[
\eta_{\D}:=\int_{\FM_d(n+m)\to\FM_d(n)}\bigwedge_{e\in\Edge(\D)}\eta_e,
\]
where $\eta_e:=\pi^*_e\eta$ is the pullback of the fixed volume form along the projection $\pi_e\colon \FM_d(n+m)\to\FM_d(2)$.  
\begin{lemma}\label{lemma::kontsevich_map}
For a chordal graph $\Gr$ and a diagram $D\in\Graphs_d(\Gr)$ the differential form $\eta_D\in \Omega^{\bullet}_{PA}(\FM_d(n))$ is the pullback of the differential form $\nu_d\in \Omega^{\bullet}_{PA}(\FM_d(\Gr))$ along the projection $$p_{\Gr}\colon\FM_d(V_{\Gr})\twoheadrightarrow\FM_d(\Gr).$$
In other words, we have a unique factorization of morphisms of dg-algebras
\[\begin{tikzcd}
	{\Graphs_d(\Gr)} & & {\Omega_{PA}(\FM_d(\Gr))} \\
	{\Graphs_d(V_{\Gr})} &  & {\Omega_{PA}(\FM_d(V_{\Gr}))}
    \arrow["{\hat{\mathbb{I}}}_{V_{\Gr}}","D\to \eta_D"', from=2-1, to=2-3]
	\arrow["{\iota_{\Gr}}", hook, from=1-1, to=2-1]
	\arrow["{\hat{\mathbb{I}}}_{\Gr}","D\to \nu_D"', dashed, from=1-1, to=1-3]
	\arrow["{p^*_{\Gr}}", hook, from=1-3, to=2-3]
\end{tikzcd}\]
\end{lemma}
\begin{proof}
The morphism $p_{\Gr}\colon \FM_d(V_{\Gr})\to \FM_d(\Gr)$ corresponds to the inclusion $\Gr\subset \K_{V_{\Gr}}$ of graphs on the same vertex set, hence $p_{\Gr}$ is a sequence of spherical blow ups, so the pullback $p^*_{\Gr}$ is injective. Since both $\iota_{\Gr}, p^*_{\Gr}$ are injective, the existence and uniqueness of $\hat{\mathbb{I}}_{\Gr}$ is equivalent to an inclusion $$\hat{\mathbb{I}}_{V_{\Gr}}(\Graphs_d(\Gr))\subset p_{\Gr}^*\Omega_{PA}(\FM_d(\Gr)).$$ 

$\hat{\mathbb{I}}$ is a homomorphism of dgc algebras, so it suffice to check the desired inclusion for generators of $\Graphs_d(\Gr)$, i.e., for internally connected diagrams $\mathsf{ICG}_d(\Gr)$. For a diagram $\D\in \mathsf{ICG}_d(\Gr)$, let $G_{\D}\subset V_{\Gr}$ be a clique as in Lemma~\ref{lem::ICG::subgraph} and consider the corresponding commutative triangle
\[\begin{tikzcd}
	{\FM_d(V_{\Gr})} & {\FM_d(\Gr)} \\
	& {\FM_d(G_{\D})}
	\arrow["{p_{\Gr}}", from=1-1, to=1-2]
	\arrow["{\pi_{G_{\D}}}"', from=1-1, to=2-2]
	\arrow["{\pi_{G_{\D}}^{\Gr}}", from=1-2, to=2-2]
\end{tikzcd}\]
By the construction of $G_{\D}$, $\D$ has the form $\iota_{G_{\D}}(\D|_{G_{\D}})$, where $\iota_{G_{\D}}\colon \Graphs(G_{\D})\to \Graphs(V_{\Gr})$ is the inclusion. Recall that Kontsevich's morphism $\hat{\mathbb{I}}$ is a morphism of (weak) cooperads with nullary part~\cite[Pr.~9.21]{lambrechts2014formality}, hence, for an inclusion of sets $S\subset T$, we have $\hat{\mathbb{I}}_{T}(\iota_S\D)=\pi^*_S\hat{\mathbb{I}}_{S}(\D)$. By combining all arguments, we deduce the desired result
\[
\hat{\mathbb{I}}_{V_{\Gr}}(\D)=\hat{\mathbb{I}}_{V_{\Gr}}(\iota_{G_{\D}}(\D|_{G_{\D}}))=\pi^*_{G_{\D}}\hat{\mathbb{I}}_{G_{\D}}(\D|_{G_{\D}})=p^*_{\Gr}(\pi^{\Gr}_{G_{\D}})^*\hat{\mathbb{I}}_{G_{\D}}(\D|_{G_{\D}}) \in p^*_{\Gr}\Omega_{PA}(\FM_d(\Gr)).
\]
\end{proof} 
Thanks to Lemma~\ref{lemma::kontsevich_map}, the collection of morphisms $\hat{\mathbb{I}}_{\Gr}\colon \Graphs_d(\Gr)\to \Omega_{PA}(\FM_d(\Gr))$ defines a morphism of weak Hopf dg cocontractads $\hat{\mathbb{I}}\colon \Graphs_d\to \Omega_{PA}(\FM_d)$
\begin{proposition}\label{prop::Graphs_to_Forms} 
The morphism $\hat{\mathbb{I}}$ is a quasi-isomorphism of weak Hopf dg-cocontractads in the class of chordal graphs.
\end{proposition}
\begin{proof}
It remains to show that $\hat{\mathbb{I}}$ is a quasi-isomorphism for each particular chordal graph $\Gamma$.
For an edge $e=(v,w)\in E_{\Gr}$, the image of the diagram $\D_e=\begin{tikzpicture}[scale=0.5, ]
	\node[ext] (v0) at (-0.5,0) {$\scriptstyle v$};
	\node[ext] (v1) at (0.5,0) {$\scriptstyle w$};
	\draw (v0) edge (v1);
\end{tikzpicture}$ under $\hat{\mathbb{I}}$ is the closed form representing the class $\omega_e\in \OS^{d-1}_d(\Gr)\cong H^{\bullet}(\FM_d(\Gr))$. Hence, by Proposition~\ref{prop::Graphs_to_OS}, on cohomology we obtain an isomorphism of cocontractads:
\[
H^{\bullet}(\hat{\mathbb{I}})\colon H^{\bullet}(\Graphs_d)\overset{\mathbb{I}}{\cong} \OS_d\overset{\Id}{\to}\OS_d\cong H^{\bullet}(\FM_d).
\]
\end{proof} 

Combining Proposition~\ref{prop::Graphs_to_OS} and Proposition~\ref{prop::Graphs_to_Forms}, we deduce the following:

\begin{theorem}
\label{thm:chordal:formal}
For $d\geq 2$, the little $d$-disks contractad $\D_d$ is formal and coformal over $\mathbb{R}$ in the class of chordal graphs in the sense of Section~\ref{sec::restriction_to_chordal}.
\end{theorem}
\begin{proof}
For the formality statement, it suffices to combine
Proposition~\ref{prop::Graphs_to_OS} with Proposition~\ref{prop::Graphs_to_Forms}.
For the coformality statement, one should follow the approach outlined in the analogous result of~\cite{vsevera2011equivalence},
where a zigzag of quasi-isomorphisms of operads in the category of $L_\infty$-algebras is constructed,
connecting the Drinfeld–Kohno Lie operad $\mathfrak{t}_d$ with the operad $\mathsf{ICG}_d(\Gr)$ of internally connected diagrams.
\end{proof}

It is worth noting that the graphic matroids associated with chordal graphs are prominent examples of supersolvable matroids. The formality of the corresponding subspace arrangements was established in~\cite{coron2025matroid} using methods that generalize Kontsevich's configuration space integrals to a broader matroidal setting.

\subsection{Cyclic graphs}
\label{sec:sub:cycles}
In this section, we explain why the results of the previous section do not hold for arbitrary graphs. The main problem is that the forgetting map $\Conf_{\Gr}(M)\to \Conf_{\Gr\setminus v}(M)$ is not a fibration for non-simplical vertices. The simplest examples of non-chordal graphs are cycles $\Cyc_n$ with $n\geq 4$ vertices. The corresponding configuration spaces are of the form
\[
\Conf_{\Cyc_n}(X)=\{(x_1,\cdots,x_n)\in X^{\times n}| x_i\neq x_{i+1}, \text{ for all }i\in \mathbb{Z}/n\}.
\]In contrast with chordal graphs, the pure braid groups for cycles are abelian. More generally, we have
\begin{proposition} For a graph with no $3$-cycles, the pure braid group is a free abelian group generated by the set of edges
\[
\PB_{\Gr}=\pi_1(\Conf_{\Gr}(\mathbb{C}))=\mathbb{Z}^{\oplus E_{\Gr}}.
\] In particular, we have $\PB_{\Cyc_n}\cong \mathbb{Z}^{\oplus n}$, for $n\geq 4$.
\end{proposition}
\begin{proof}
For a triple of vertices $v,w,u$, consider the relation $T_{pr}T_{qr}T_{pq} = T_{pr}T_{qr}T_{pq}=T_{pr}T_{qr}T_{pq}$. Since the graph has no $3$-cycles, one of the elements in these products would be trivial, so we obtain a commuting relation.
\end{proof}

 We refer the reader to~\cite{felix2015rational} for the basics of rational homotopy theory for nilpotent spaces.  Recall from the beginning of Section~\ref{sec::chordal_cycles} that $\Conf_{\Gr}(\Cmplx)$ is formal for any graph, hence, the rational homotopy type of $\Conf_{\Cyc_n}(\mathbb{C})$ is determined by its minimal Sulivan model $\M\overset{\simeq}{\rightarrow} H^{\bullet}(\Conf_{\Cyc_n}(\mathbb{C}))=\OS^{\bullet}(\Cyc_n)$ of the Orlik-Solomon algebra
\[
\OS^{\bullet}(\Cyc_n)=\Lambda^{\bullet}\left(
\{\omega_{i,i+1}\}_{i\in \mathbb{Z}/n\mathbb{Z}}\left|\sum^{n}_{i=1} (-1)^{i-1}\omega_{12}\cdots \omega_{i,i+1}\cdots \omega_{n,1}\right.\right).
\]
Moreover, for a commutative graded algebra $A$, we have the isomorphism of $A_{\infty}$-algebras
\[
\mathcal{U}(\mathfrak{g}_A)\cong \mathrm{Ext}^{\bullet}_A(\mathsf{k},\mathsf{k}),
\] where $\mathfrak{g}_A$ is a $L_{\infty}$-algebra arising from the minimal Sullivan model for $A$, and $\mathrm{Ext}^{\bullet}_A(\mathsf{k},\mathsf{k})$ is the Yoneda algebra.
(We refer to~\cite{khoroshkin2023derived} and references therein for the description of the notion of a universal enveloping of an $L_\infty$ algebra and in which sense it is universal and why.)
The description of the rational homotopy groups of the $\Conf_{\Cyc_n}(\mathbb{C})$ and the Whitehead Lie bracket is described in Theorem~\ref{thm::cycl::homotopy} below.
The proof of this theorem is purely algebraic and is split into the following steps:
\begin{itemize}[itemsep=0pt,topsep=0pt]
    \item
First, we describe the noncommutative Gr\"obner basis for the commutative algebra $\OS^{\bullet}(\Cyc_n)$ (Lemma~\ref{lem::Grobner::Cycle}).
\item
Second, for $n\geq 5$ we describe the minimal noncommutative resolution of this algebra, which is shown to coincide with the Anick resolution (Lemma~\ref{lem::Anick::resol::cycle}).
\item
Third, we work out the quadratic component of the differential in the minimal resolution and show that the corresponding algebra $\mathrm{Ext}^{\bullet}_{\OS^{\bullet}(\Cyc_n)}(\mathbb{Q},\mathbb{Q})$ is quadratic and Koszul.
\item
Fourth, as a corollary, we end up with the description of the Whitehead Lie bracket and the first nontrivial Massey product on the rational homotopy groups of $\Conf_{\Cyc_n}(\mathbb{C})$ for $n\geq 5$ outlined in Theorem~\ref{thm::cycl::homotopy}.
Unfortunately, we do not know the description of all higher Lie brackets for these spaces.
\end{itemize}

In order to shorten the notation, we will write simply $\omega_i$ instead of the longer notation $\omega_{i,i+1}$.
\begin{lemma}
\label{lem::Grobner::Cycle}
The (skew) commutativity relations
$$\forall i<j \ \ \omega_i \omega_j + \underline{\omega_j \omega_i}=0, \quad \underline{\omega_i^2}=0 $$
together with the relation
\begin{equation}
\label{eq::cycle::rel}
\omega_1\cdot\ldots\cdot \omega_{n-1} -\omega_1\cdot\ldots \cdot \omega_{n-2} \cdot \omega_{n} +\ldots+(-1)^{n} \underline{\omega_2\cdot \ldots \omega_n}=0    
\end{equation}
form a noncommutative Gr\"obner basis for the Orlik–Solomon algebra $\OS(\Cyc_n)$ with respect to the standard degree-lexicographical ordering on monomials $\omega_1<\ldots<\omega_n$ (with leading monomials underlined).
\end{lemma}

\begin{proof}
This is a straightforward verification that all $s$-polynomials of all compositions reduce to zero. The only nontrivial compositions occur between Relation~\eqref{eq::cycle::rel} and the commutativity relations. For example:
\begin{multline}
(-1)^{n-1}\eqref{eq::cycle::rel}\cdot \omega_i - \omega_2\ldots \omega_{n-1}\cdot( \omega_n \omega_i - \omega_i \omega_n) = \sum_{j=1}^{n-1} (-1)^{n-j} \omega_1\ldots \widehat{\omega_j}\ldots \omega_n \omega_i - \omega_1\ldots \omega_{n-1} \omega_i \omega_n = \\
= \sum_{j=1,j\neq i}^{n} (-1)^{i} \omega_1 \ldots \omega_i^2 \ldots \widehat{\omega_j}\ldots \omega_n \pm \omega_1\ldots \omega_n + \omega_2 \ldots \omega_{n-1} \omega_i \omega_n \simeq \sum_{j\neq i} \pm  \omega_1 \ldots \omega_i^2 \ldots \widehat{\omega_j}\ldots \omega_n  = 0
\end{multline}
\end{proof}

Among the homological applications of noncommutative Gr\"obner bases, one of the most important is the \textbf{Anick resolution} of the trivial module. This construction is based on resolutions for monomial algebras introduced in~\cite{anick1986homology}. 
We briefly recall the recursive definition of Anick chains and tails—the central combinatorial objects in the Anick construction (following the clear exposition in~\cite[\S 3]{ufnarovski}). These are defined for a monomial algebra
$$\hat{A} := \Bbbk\langle \omega_1, \ldots, \omega_n \mid \hat{g}_1, \ldots, \hat{g}_m \rangle,$$
where the relations $\{\hat{g}_j\}$ are monomials.

\begin{definition}The sets of $m$-chains and their corresponding tails are defined recursively as follows:
\begin{itemize}[itemsep=0pt,topsep=0pt]
\item The unique $0$-chain is the unit element $1$, which has an empty tail.
\item The $1$-chains are the generators $\omega_1, \dots, \omega_n$ of the algebra $\hat{A}$. For a $1$-chain $f = \omega_i$, the tail is the word itself, $t = \omega_i$.
\item An $m$-chain ($m \ge 2$) is a word $f$ that can be written as a product $f = th$, where $h$ is an $(m-1)$-chain and $t$ is a non-empty word called the tail of $f$. Let $r$ be the tail of $h$; then the tail $t$ is uniquely determined by the following conditions:
\begin{enumerate}
\item The product $tr$ contains a unique monomial relation $\hat{g}$ as a subword.
\item This relation $\hat{g}$ is a prefix of $tr$ (i.e., $tr = \hat{g}s$ for some word $s$), but $\hat{g}$ is not a subword of $t$.
\item No proper suffix of $t$ satisfies the above conditions.
\end{enumerate}
\end{itemize}
\end{definition}
As shown in~\cite{anick1986homology}, the set of $m$-chains $C_m$ forms a basis for the $m$-th term in the minimal free resolution of $\Bbbk$ as a left $\hat{A}$-module:
\[
\begin{tikzcd}
\ldots \ar[r] & \hat{A} \otimes C_m \ar[rr, "a \otimes tf \mapsto at \otimes f"] && \hat{A} \otimes C_{m-1} \ar[r] & \ldots \ar[r] & \hat{A} \otimes C_1 \ar[r] & \hat{A} \ar[r, dashed] & \Bbbk.
\end{tikzcd}
\]

Now let
\[
A := \Bbbk\langle \omega_1, \ldots, \omega_n \mid g_1, \ldots, g_m \rangle
\]
be an algebra with a reduced Gr\"obner basis \( g_1, \ldots, g_m \), and let \( \hat{g}_i \) denote the leading monomial of \( g_i \). Then there exists a resolution of \( A \) whose basis is indexed by the chains constructed for the associated graded algebra \( \hat{A} \) with monomial relations, and whose differential is given by modifying the monomial differential with additional lower-order terms.

\begin{lemma}
\label{lem::Anick::Cycle}
The set of \( m \)-chains for the algebra \( \OS(\Cyc_n) \) is described as follows:
\begin{equation}
\label{eq::Anick::chains}
\left\{
\begin{array}{c}
(\omega_{i_{01}} \ldots \omega_{i_{0k_{0}}})(\omega_{2}\ldots \omega_n)
(\omega_{i_{11}} \ldots \omega_{i_{1 k_{1}}})(\omega_{2}\ldots \omega_n)
\ldots (\omega_{2}\ldots \omega_n)
(\omega_{i_{l1}} \ldots \omega_{i_{lk_{l}}}) \omega_1^{k} \\
\text{where } l, k_i, k \in \mathbb{Z}_{\geq 0} \text{ and } 
\forall j = 0, \ldots, l: \ i_{j1} \geq i_{j2} \geq \ldots \geq i_{jk_j} \geq 2, \\
\text{and } m = k + \sum_{i=0}^{l} k_i + 2l.
\end{array}
\right\}
\end{equation}
\end{lemma}

\begin{proof}
The possible tails of chains include the variables \( \omega_1, \ldots, \omega_n \) and the word \( \omega_2 \ldots \omega_{n-1} \). The structure described above provides an inductive characterization of the set of \( m \)-chains.
\end{proof}

\begin{lemma}
\label{lem::Anick::resol::cycle}
The Anick resolution for the algebra \( \OS(\Cyc_n) \) is minimal when \( n \geq 5 \).
\end{lemma}

\begin{proof}
The relations in the algebra \( \OS(\Cyc_n) \) are homogeneous, so the differential in the Anick resolution preserves the grading. Let \( C_m^{(s)} \) denote the set of \( m \)-chains of \( \omega \)-grading equal to \( s \). According to Lemma~\ref{lem::Anick::Cycle}, we have:
\[
C_m = C_m^{(m)} \oplus C_m^{(m - 2 + (n-1))} \oplus C_m^{(m - 4 + 2(n-1))} \oplus \ldots 
= \bigoplus_{l = 0}^{\left\lfloor \frac{m+1}{2} \right\rfloor} C_m^{(m + l(n-3))},
\]
where the index \( l \) corresponds to the number of factors of the form \( (\omega_2 \ldots \omega_n) \) in~\eqref{eq::Anick::chains}.

In particular, \( C_m^{(s)} \neq 0 \) if and only if \( s \equiv m \mod (n-3) \). It follows that for \( n \geq 5 \), no two consecutive homological degrees contain chains with the same \( \omega \)-grading. Hence, the Anick resolution is minimal, as the differential preserves the \( \omega \)-grading.
\end{proof}

The Anick resolution of an algebra \( A \) is a resolution in the category of left \( A \)-modules. It is particularly well suited for constructing a basis of the coalgebra \( \mathrm{Tor}^A_{\bullet}(\mathsf{k}, \mathsf{k}) \) and its graded linear dual \( \mathrm{Ext}_A^{\bullet}(\mathsf{k}, \mathsf{k}) \). 
Recently, P.~Tamaroff introduced in~\cite{tamaroff2018minimal} a generalization of the Anick resolution to free resolutions in the category of algebras, in which Anick chains also index the generators. As a consequence, whenever the Anick resolution is minimal, the corresponding free resolution in Tamaroff’s construction is also minimal. Furthermore, the algebra structure on the extension groups can be read off from the quadratic part of the differential in this minimal free resolution. This leads to the following:

\begin{lemma}
The algebra \( \mathrm{Ext}^{\bullet}_{\OS(\Cyc_n)}(\mathbb{Q}, \mathbb{Q}) \) is a quadratic Koszul algebra generated by elements \( y_1, \ldots, y_n \) of degree \( 0 \), and an element \( u \) of degree \( n-3 \) (corresponding to the cyclic relation~\eqref{eq::cycle::rel}), subject to the quadratic relations:
\begin{equation}
\label{eq::cycl::Ext::rel}
\forall i \neq j \colon [y_i, y_j] = 0, \qquad \sum_{i=1}^n [y_i, u] = 0.
\end{equation}
\end{lemma}

\begin{proof}
As a graded vector space, the algebra \( \mathrm{Ext}^{\bullet}_{\OS(\Cyc_n)}(\mathbb{Q}, \mathbb{Q}) \) is the graded linear dual of the coalgebra \( \mathrm{Tor}_{\bullet}^{\OS(\Cyc_n)}(\mathbb{Q}, \mathbb{Q}) \), with a basis indexed by Anick chains.
Since the Anick resolution is minimal and respects the \( \omega \)-grading, the multiplication in \( \mathrm{Ext} \)-algebra is determined by the quadratic component of the differential. In particular, the generators of the \( \mathrm{Ext} \)-algebra are exactly those Anick chains that cannot be decomposed as concatenations of two shorter chains (i.e., chains of lengths \( m_1 \) and \( m_2 \) with \( m_1 + m_2 = m \)).

From the description~\eqref{eq::Anick::chains}, we conclude that the elements \( y_i \) correspond to the dual basis of the set of \( 1 \)-chains, and the element \( u \) is dual to the unique \( 2 \)-chain \( \omega_2 \ldots \omega_n \) which is not a concatenation of two \( 1 \)-chains. These elements generate the \( \mathrm{Ext} \)-algebra.
The skew-commutativity of the elements \( \omega_i \) in \( \OS(\Cyc_n) \) implies commutativity of the dual generators \( y_i \). The relation \( \sum_{i=1}^n [y_i, u] = 0 \) follows directly from the invariance under the cyclic group \( \mathbb{Z}/n\mathbb{Z} \) (acting by automorphisms on the cyclic graph \( \Cyc_n \)), and from the homogeneity of the relation.

We order the generators as follows:
\[
y_1 > y_2 > \ldots > y_n > u
\]
and consider the algebra generated by \( y_i \)'s and \( u \) subject to the relations in~\eqref{eq::cycl::Ext::rel}. It is straightforward to verify that these relations form a quadratic Gr\"obner basis with respect to the degree-lexicographical order.
Moreover, there is a bijection between the set of normal words in this algebra and the set of Anick chains~\eqref{eq::Anick::chains}, where each occurrence of \( (\omega_2 \ldots \omega_n) \) is replaced by \( u \). Therefore, no additional relations arise in the multiplication, and the algebra \( \mathrm{Ext}^{\bullet}_{\OS(\Cyc_n)}(\mathbb{Q}, \mathbb{Q}) \) is indeed quadratic and Koszul.
\end{proof}

Finally, standard techniques from rational homotopy theory (see, e.g.,~\cite{felix2015rational}) allow us to conclude the following:

\begin{theorem}
\label{thm::cycl::homotopy}
For a cycle graph $\Cyc_n$ with $n \geq 5$, the Lie algebra of rational homotopy groups of the configuration space $\Conf_{\Cyc_n}(\mathbb{C})$, defined via Whitehead products, admits the following presentation:
\begin{equation}\label{eq::lie_structure_OS}
\pi_{\bullet+1}(\Conf_{\Cyc_n}(\mathbb{C})) \otimes \mathbb{Q} = 
\mathsf{Lie} \left\langle 
\begin{array}{c}
y^0_1, \ldots, y^0_n, \\
u^{n-3}
\end{array}
\left| 
\begin{array}{c}
[y_i, y_j] = 0, \\
\sum_{i=1}^{n} [y_i, u] = 0
\end{array}
\right.
\right\rangle
\end{equation}

Moreover, the minimal $L_{\infty}$-algebra contains a nontrivial higher Lie bracket:
\begin{equation}
\label{eq::L_inf_str_OS}
l_{n-1}(y_2, \ldots, y_n) = 
- l_{n-1}(y_1, y_3, \ldots, y_n) = 
\cdots = 
(-1)^n l_{n-1}(y_1, \ldots, y_{n-1}) = u.
\end{equation}
\end{theorem}

\begin{proof}
As shown above, for $n \geq 5$, the fundamental group of $\Conf_{\Cyc_n}(\Cmplx)$ is abelian, and the space is nilpotent.
Furthermore, according to~\cite{lyskov2023contractads}, graphical configuration space $\Conf_{\Gamma}(\Cmplx)$ is formal for all graphs $\Gamma$, meaning that the Orlik–Solomon algebra $\OS(\Gamma)$ is quasi-isomorphic to the cohomology of the De Rham complex.
As a result, the generators of the minimal free resolution of $\OS(\Cyc_n)$ in the category of (skew-)commutative algebras fully determine the rational homotopy $L_{\infty}$-algebra $\pi_{\bullet+1}(X) \otimes \mathbb{Q}$. Concretely, we have 
\(
\pi_{\bullet+1}(X) \cong (V^{\bullet+1})^{\vee},
\)
and the $L_{\infty}$-structure is given by $l_n := d_n^{\vee}$, where $d_n \colon V \to \Lambda^n V$ is the component of the differential of weight $n$.
In particular, the quadratic component $d_2^{\vee}$ defines the Whitehead product.
Moreover, the algebra $\mathrm{Ext}_{\OS(\Cyc_n)}^{\bullet}(\mathbb{Q}, \mathbb{Q})$ serves as the universal enveloping algebra of the corresponding Lie algebra of rational homotopy groups. Since its relations~\eqref{eq::cycl::Ext::rel} are given in terms of commutators, the corresponding Lie algebra has the same generators and relations.

To complete the proof, we analyze the differential for the generator $u$, which arises from the cyclic relation~\eqref{eq::cycle::rel}. This differential captures the first nontrivial higher Massey products, as reflected in the higher bracket $l_{n-1}$.
\end{proof}

\section{Logarithmic geometry and Formality}
\label{sec::logformality} 
In~\cite{vaintrob2021formality, dupont2024logarithmic}, the authors proved the formality of (framed) little $2$-disks operad using the methods of logarithmic geometry and mixed Hodge structures. The crucial part of their proofs is to endow the operad of moduli spaces of stable curves $\beM_{0,n+1}$ with the logarithmic structure such that its Kato-Nakayama realization is a (framed) $E_2$-operad. In this section, we explain how to extend their results to the contractad setting. We briefly recall the basics of logarithmic geometry and construct $\psi$-classes for graphical compactifications needed for the logarithmic structure.

\subsection{Psi-classes and Normal bundles}
\label{sec:sub:psi:classes}
For the Deligne-Mumford moduli space $\beM_{0,n+1}$, there is the collection of line bundles $\La_{\infty},\La_1,\cdots,\La_n$ whose fibers over a stable curve $[\Sigma]$ are cotangent lines at its marked points, and their first Chern classes, called $\psi$-classes, play a significant role in the intersection theory. We refer to~\cite{zvonkine2012introduction} for detailed treatments. In this section, we construct their graphical analogues.
\subsubsection{$h$-classes and divisors} 
The graphical compactification $\bM(\Gr)$ was defined in~\cite{khoroshkin2024hilbert} and recalled in Section~\ref{sec::wonerful_contractad}. The second cohomology group $H^2(\bM_{\mathbb{C}}(\Gr))$ is spanned by the Poincare dual classes of canonical divisors $[D_G]$ with only linear relations
\[
\sum_{G: e\subset G} [D_G]=\sum_{G': e'\subset G'} [D_{G'}], \text{ for each pair of edges }e,e'\in E_{\Gr}.\label{eq::linearelforbM}
\] 
For a tube $G$, let $\xi_G$ be the pullback of the line bundle $\mathcal{O}_{\Pro}(1)$ along the projection $\rho_G: \bM(\Gr) \rightarrow \Pro(W_{\Gr}/H_G)$, and we denote by $h_G=c_1(\xi_G)$ its first Chern class. In particular, for an edge $e$, we have $h_e=0$ for dimension reasons. In~\cite[Lemma~4.1.7]{khoroshkin2024hilbert}, it was shown the divisor expression of $h$-classes
\[
h_G = \sum_{K: e\subset K, G\not\subset K} [D_K],
\] where the sum ranges over all proper tubes $K$ that contain fixed edge $e$ and do not contain $G$. This formula is valid for any edge $e\subset G$.

\begin{proposition}\label{prop::divisor_h} We have
\[
[D_G]=-\sum_{H:G\subset H\subset \N(G)} (-1)^{|H|-|G|} h_H,
\] where $\N(G):=\{v| \exists w\in G\colon (v,w)\in E_{\Gr}\}$ is the neighborhood of a tube.
\end{proposition}
\begin{proof}
For a graph $\Gr$, let $2^{\Gr}$ be the poset of non-empty tubes of $\Gr$ ordered by inclusion. By verifying the properties of M\"obius functions for posets, we see that the function $\mu_{\Gr}\colon 2^{\Gr}\times 2^{\Gr}\to \mathbb{Z}$ is given by
\[
\mu_{\Gr}(G,H):=\begin{cases}
    (-1)^{|H|-|G|}\text{, if } G\subset H\subset \mathcal{N}(G)
    \\
    0\text{, otherwise.}
\end{cases} 
\] is the M\"obius function for poset $2^{\Gr}$.
Next, we define formal element $[D_{\Gr}]$ given by the linear relations $[D_{\Gr}]=-\sum_{e\subset H} [D_H]$ that is independent on the choise of edge $e$. Using this notation, we can rewrite class $h_G$ in the following way
\[
h_G=-\sum_{H: G\subset H} [D_H].
\] So, by M\"obius inversion formula, we obtain the desired formula
\[
[D_G]=-\sum_{G\subset H} \mu_{\Gr}(G,H)h_H=-\sum_{H:G\subset H\subset \mathcal{N}(G)} (-1)^{|H|-|G|} h_H.
\]
\end{proof} The advantage of $h$-classes is their compatibility with the contractad structure~\cite[Prop~4.1.12]{khoroshkin2024hilbert}
\[
(\circ^{\Gr}_G)^*h_H= \begin{cases}
1\otimes h_H\text{, if }H\subset G;
\\
h_{H/G}\otimes 1\text{, else,}
\end{cases}
\] and graph inclusions: for a graph inclusion $\iota\colon \Omega\to \Gr$, we have $\pi_{\iota}^*h_{G}=h_{\iota(G)}$, where $\pi_{\iota}\colon \bM(\Gr)\to \bM(\Omega)$ the induced morphism.

\subsubsection{Psi-classes for Modular compactifications} 
We recall the construction of the moduli space $\beM_{0,\K_{\lambda}}\cong \bM(\K_{\lambda})$ of $\K_{\lambda}$-stable curves from Section~\ref{sec::modular}. We consider the associated universal curve $\bC_{0,\K_{\lambda}}\to\beM_{0,\K_{\lambda}}$. Analogously to~\cite{zvonkine2012introduction}, we define the relative cotangent bundle $\La$ over the universal curve, whose restriction to the $\Gr$-stable curve $\La|_{\Sigma}$ is the line bundle of abelian differentials.  

\begin{definition}[Geometric Psi-classes]
 For $v\in V_{\K_{\lambda}}\cup \{\infty\} $, let $\La_v$ be the line bundle over $\beM_{0,\K_{\lambda}}$ defined as pullback of $\La$ through section $x_v\colon\beM_{0,\K_{\lambda}}\to \bC_{0,\K_{\lambda}}$. More precisely, the fiber at each point $[\Sigma]\in \beM_{0,\K_{\lambda}}$ is the cotangent line to $\Sigma$ at the marked point $x_v$. Denote by $\psi_v=c_1(\La_v)$ the first Chern class of this bundle.
\end{definition}

There is the divisor expression of $\psi$-classes on $\beM_{0,\K_{\lambda}}$ for complete multipartite graph $\K_{\lambda}$.
\begin{proposition}[Presentation of $\psi$-classes]\label{prop:prespsi} 
For any given edge $e\in E_{\K_{\lambda}}$ connecting vertices $v$ and $w$ we have
\begin{gather*}
 \psi_v=\sum_{G: v\in G, w\not\in G} [D_G], \\
 \psi_{\infty}=h_{\K_{\lambda}}=\sum_{G\supset e} [D_G].
\end{gather*}
\end{proposition}
\begin{proof} The proof repeats the ones of~\cite[Prop~2.13]{zvonkine2012introduction}. For a pair of adjacent vertices $w,u\in V_{\Gr}\cup \{\infty\}$ ( we consider $\infty$ as an additional vertex adjacent to all others), we define the meromorphic section $\alpha_{w,u}$ of $\La$ as follows. Recall that, for a stable curve $\Sigma$ of genus zero and two distinct smooth points $p,p'$, there is a unique meromorphic abelian form $\alpha$, that has two simple poles at points $p,p'$ with opposite residues $1$ and $-1$ respectively. Also, this form is non-zero only  at smooth points of the chain connecting components that contain given points. For a $\K_{\lambda}$-stable curve $\Sigma$, we let $\alpha_{v,w}(\Sigma)$ be the corresponding form for $p=x_w, p'=x_u$. The union  $\cup_{\Sigma} \alpha_{w,u}(\Sigma)$ of these forms over each stable curve is the section $\alpha_{w,u}$ of $\La$ over the whole universal curve $\bC_{0,\K_{\lambda}}$. 

Consider a vertex $v$ that is adjacent to both $w$ and $u$. In particular, for a stable curve, the marked point $x_v$ is not a pole of $\alpha_{wu}$ and $x_v$ is a zero if it is outside the chain connecting the components that contain marked points $x_w$ and $x_u$.  Hence, the composition $\beta_{wu}=\alpha_{wu}\circ x_v$ defines a holomorphic section of $\La_v$, so $\psi_v$ is presented by the zero divisor of $\beta_{wu}$. In the case $u=\infty$, the irreducible components of this divisor are $D_G$ for tubes $G$ with $v\in G, w\not\in G$. In the case $v=\infty$, the zero divisor is the union of $D_G$ for tubes containing the edge $e=(w,u)$.
\end{proof}  We consider the following example
\begin{example}[Psi-classes on Losev-Manin moduli spaces]\label{ex::psi_losevmanin} According to Example~\ref{ex::Wonderful}, for the graph $\St_n:=\K_{(n,1)}$, the compactification $\beM_{0,\St_n}$ coincides with the Losev-Manin moduli space of stable chains $\overline{\EuScript{L}}_{0,n}$. Thanks to Proposition~\ref{prop:prespsi}, the only non-trivial $\psi$-classes are $\psi_0$ and $\psi_{\infty}$, since all other vertices are leaves. Thanks to Proposition~\ref{prop::divisor_h}, we rewrite $\psi_0$ in terms of $h$-classes.
\[
\psi_0=\sum_{0\in G} (-1)^{|G|-1}h_G,
\] where the sum ranges over all tubes containing center $0$.
\end{example}

\subsubsection{Graphical Psi-Classes in general} 
Example~\ref{ex::psi_losevmanin} suggests how to define $\psi$-classes for a general graphical compactification. For a graph $\Gr$, let $\St(v)$ be the stellar graph on the vertex set $\N(v)$ with the center $v$. The graph-inclusion $\St(v)\hookrightarrow \Gr$ produces a projection
\[
\pi_v\colon \bM(\Gr)\twoheadrightarrow \bM(\St(v))\cong \overline{\EuScript{L}}_{0,\deg(v)}. 
\]
\begin{definition} For a vertex $v$, the line bundle $\mathcal{L}_v$ over $\bM(\Gr)$ is the pullback $\pi_v^*\La_0$ along the projection above. In the case $v=\infty$, we let $\La_{\infty}:=\xi_{\Gr}$. We let $\psi_{\infty}=h_{\Gr}$ and $\psi_v$ be the corresponding first Chern classes.
\end{definition} Thanks to Example~\ref{ex::psi_losevmanin} and compatibility of $h$ classes with graph inclusions, we have the following $h$-presentation of $\psi$-classes
\begin{equation}
  \psi_v=\sum_{H\colon v\in H\subset \N(v)} (-1)^{|H|-1}h_H.  
\end{equation} We refer to Proposition~\ref{prop:prespsi}, to verify that in the case of a complete multipartite graph, both definitions of $\psi$-classes coincide.

From the examination of pullback bundles along contractad maps, we see 
\begin{proposition}[Compatibility with Operadic structure]\label{prop::psi_comp} Line bundles $\La_{v}$ are well-behaved with respect to contractad structure
\[
(\circ_G^{\Gr})^*\mathcal{L}_v\cong \begin{cases}
\mathcal{L}_v\otimes 1\text{, if } v\not\in G\text{ or } v=\infty
\\
1\otimes \mathcal{L}_v\text{, if } v\in G
\end{cases}
\]
\end{proposition}
For $\beM_{0,n+1}$, the normal bundles of canonical divisors are expressed in terms of $\psi$-classes. A similar holds in the graphical case 
\begin{proposition}[Self-intersection formula]\label{prop::normalbundle}
For a divisor $D_G\cong \bM(\Gr/G)\times\bM(\Gr|_G)\subset \bM(\Gr)$, its normal bundle $\N_G$ is isomorphic to
\[
\mathcal{N}_G\cong \mathcal{T}_{\{G\}}\otimes \mathcal{T}_{\infty},
\] where $\mathcal{T}_{\{G\}}=\mathsf{pr}_{\Gr/G}^*\La^{\vee}_{\{G\}}$ and $\mathcal{T}_{\infty}=\mathsf{pr}_{\Gr|_G}^*\La^{\vee}_{\infty}$ are the pullbacks of dual line bundles 
\end{proposition}
\begin{proof} Since $H^1(\bM(\Gr))=0$, a holomorphic line bundle is uniquely determined by its first Chern class. For a divisor $\iota\colon Y\subset X$, we have $c_1(\mathcal{N}_{Y\subset X})=\iota^*[Y]$, so, it suffices to prove the identity $(\circ^{\Gr}_G)^*[D_G]=-\psi_{\{G\}}\otimes 1-1\otimes \psi_{\infty}$. The desired formula follows from explicit computations
\begin{multline*}
  (\circ^{\Gr}_G)^*([D_G])=-\sum_{H:G\subset H\subset \mathcal{N}(G)} (-1)^{|H|-|G|} (\circ^{\Gr}_G)^*(h_H)=-1\otimes h_G-\sum_{H:G\subsetneq H\subset \mathcal{N}(G)} (-1)^{|H|-|G|} h_{H/G}\otimes 1 =\\= -1\otimes \psi_{\infty}-\sum_{K:\{G\}\in K\subset \mathcal{N}_{\Gr/G}(\{G\})} (-1)^{|K|-1} h_{K}\otimes 1=-1\otimes\psi_{\infty}-\psi_{\{G\}}\otimes 1.
\end{multline*}
\end{proof}

\subsection{Logarithmic geometric models}
\label{sec:sub:log:geom}
\subsubsection{Log schemes}
We briefly recall necessary notions from (virtual) logarithmic geometry and refer to~\cite{dupont2024logarithmic} for details.

A \emph{log scheme} $\X=(X,\M_X)$ is a scheme $X$ endowed with a sheaf of monoids $\mathcal{M}_X$ on $X$ in etale topology and a morphism of sheaf monoids $\alpha_X\colon \mathcal{M}_X\to \Orb_X$ such that the induced morphism $\alpha_X^{-1}\Orb_X^{\times}\to \Orb_X^{\times}$ is an isomorphism. We consider the following examples
\begin{itemize}
    \item For a scheme $X$, we could consider the trivial log scheme structure $\X^{\mathrm{triv}}$ with $\M_X=\Orb_X^{\times}$.
    \item The log-point is the log scheme $*_{\mathrm{log}}=(\Spec(\mathsf{k}),
    \mathsf{k}^{\times}t^{\mathbb{N}})$, where $\mathsf{k}^{\times}t^{\mathbb{N}}$ is a free monoid with $\alpha$ is the evaluation at $t=0$.
    \item For a pair $(X,D)$ of a scheme $X$ and a divisor $D\subset X$ with a normal crossing, we consider the divisorial log scheme $\M_{X,D}$ whose sections are regular functions on $X$ that are invertible out of $D$.
    \item More generally, for a triple $(X,D,\E)$ of a scheme $X$, a divisor $D$, and a vector bundle $\E$ with a splitting into line bundles $\E=\oplus^n_{i=1} \La_i$, the \emph{Deligne-Faltings log scheme} is the scheme $X$ with a log-structure $\M_{X,D,\E}$ whose sections are monomial  functions on $\E$, i.e., regular functions on $\E$ given in local coordinates by $g(x)\prod^n_{i=1, a_i\in \mathbb{N}} t^{a_i}_i$, that are invertible on $\E^{\circ}|_{X\setminus D}$, there $\E^{\circ}=\La^{\times}_1\underset{X}{\times}\La^{\times}_2\underset{X}{\times}\cdots\underset{X}{\times}\La^{\times}_n$ is the associated $(\mathbb{G}_m)^n$-bundle, with $\alpha\colon \M_{X,D,\E}\to \Orb_X$ given by restriction to zero locus.
\end{itemize} An \emph{ordinary} morphism of log schemes $\phi\colon (X,\M_X)\to (Y,\M_{Y})$ is a morphism of schemes $\phi\colon X\to Y$ with a morphism of sheaves $\phi^*\colon \phi^{-1}\M_Y\to \M_X$ compatible with structure morphisms $\alpha_X, \alpha_Y$. A \emph{virtual} morphism is a morphism of schemes $\phi\colon X\to Y$ with an extension of $\phi^*\colon \phi^{-1}\Orb_Y^{\times}\to \Orb_X^{\times}$ to a morphism between group completions of log sheaves $\phi^*\colon \phi^{-1}\M_Y^{\mathrm{gp}}\to \M_X^{\mathrm{gp}}$.
\begin{example}\label{ex::logmorphisms} We consider particular examples of ordinary/virtual morphisms needed soon 
\begin{itemize}
    \item[(log point)] A log point $*_{\log}$ admits a multiplication $*_{\log}\times *_{\log}\to *_{\log}$ given by $(t^i,t^j)\mapsto t^{i+j}$, that is an ordinary morphism, while the inverse map $*_{\log}\to *_{\log}, t\mapsto t^{-1}$ and the unit $*\to *_{\log}$ are only virtual morphisms. So, $*_{\log}$ is a group object in log schemes with virtual morphisms.
    \item[(Monomial maps)]  For a pair of of splitting bundles $\E,\E'$, an ordinary morphism $(X,D,\E)\to (X,D,\E')$ over $X$ is a monomial map $\E\to \E'$ (maps monomial functions to monomial) which maps $X$ to $X'$ and whose restriction to the complement $X\setminus D$ maps $\E^{\circ}|_{X\setminus D}$ to $\E^{'\circ}|_{X\setminus D}$. A virtual morphism is just a monomial map $\E^{\circ}|_{X\setminus D}\to\E^{'\circ}|_{X\setminus D}$.
    \item[(Pairings)] A pairing of line bundles $\langle-,-\rangle\colon \La_1\otimes \La_2\to \N$ over $X$ that is a non-degenerate over $X\setminus D$, define an ordinary morphism $(X,D,\La_1\oplus\La_2)\to (X,D,\N)$. 
    \item[(Inverting)] For a line bundle $\La$ over $X$, inverting of non-zero sections induces a virtual isomorphism $(X,\La)\overset{\cong_{virt}}{\rightarrow} (X,\La^{\vee})$.
    \item[(Pullback)] For a divisor $\iota\colon D\subset X$, and splitting bundle $\E$, the pullback log structure on $D$ is $\iota^*(X,\E)=(D,N\oplus \iota^*\E)$, where $N$ is the normal bundle.
    \item[(Trivialising)] If a line bundle $\La$ is trivial over the complement $X\setminus D$, we have the virtual isomorphism $(X,D,\La)\cong_{virt} (X,D)\times *_{\log}$.
\end{itemize}
\end{example}
\subsubsection{Logarithmic little disks} For a graph $\Gr$, we consider the Deligne-Faltings log schemes
\[
\frakFM^{\fr}(\Gr):=(\bM(\Gr),\partial\bM(\Gr), \bigoplus_{v\in V_{\Gr}\cup \{\infty\}}\mathcal{T}_v), \quad \frakFM(\Gr):=(\bM(\Gr),\partial\bM(\Gr), \mathcal{T}_{\infty}),
\] where $\mathcal{T}_v:=\La^{\vee}_v$. For $\Gr=\Path_1$, we put $\frakFM^{\fr}(\Path_1)=*_{\log}$ and $\frakFM(\Path_1)=*$.
\begin{proposition} The contractad structure on $\bM$ lifts to the contractad structure on $\frakFM^{\fr}$ in the category of log schemes with ordinary morphisms.
\end{proposition}
\begin{proof} Consider the contractad morphism $\circ^{\Gr}_G\colon \bM(\Gr/G)\times \bM(\Gr|_G)\to \bM(\Gr)$. To lift this morphism to a morphism of log schemes, we need to define an ordinary morphism  $\frakFM^{\fr}(\Gr/G)\times \frakFM^{\fr}(\Gr|_G)$ to the pullback log-structure $(\circ^{\Gr}_G)^*\frakFM^{\fr}$. According to Example~\ref{ex::logmorphisms}, $(\circ^{\Gr}_G)^*\frakFM^{\fr}$ is also a Deligne-Faltings log scheme on the product $\bM(\Gr/G)\times \bM(\Gr|_G)$, with splitting bundle   $\E'=\N_G\oplus (\circ^{\Gr}_G)^*(\mathcal{T}_{\infty}\oplus\bigoplus_{v\in V_{\Gr}}\mathcal{T}_v)$. Thanks to Proposition~\ref{prop::psi_comp} and Proposition~\ref{prop::normalbundle}, we have $\E'\cong (\mathcal{T}_{\{G\}}^{\Gr/G}\otimes \mathcal{T}_{\infty}^{\Gr|_G})\oplus (\mathcal{T}_{\infty}^{\Gr/G}\oplus \bigoplus_{v\not\in G} \mathcal{T}_{v}^{\Gr/G}) \oplus (\bigoplus_{v\not\in G} \mathcal{T}_{v}^{\Gr|_G})$. Therefore, we are seeking an ordinary monomial map.  
\[
 (\mathcal{T}_{\{G\}}^{\Gr/G}\oplus \mathcal{T}_{\infty}^{\Gr|_G})\oplus (\mathcal{T}_{\infty}^{\Gr/G}\oplus \bigoplus_{v\not\in G} \mathcal{T}_{v}^{\Gr/G}) \oplus (\bigoplus_{v\not\in G} \mathcal{T}_{v}^{\Gr|_G})\to \E'= (\mathcal{T}_{\{G\}}^{\Gr/G}\otimes \mathcal{T}_{\infty}^{\Gr|_G})\oplus (\mathcal{T}_{\infty}^{\Gr/G}\oplus \bigoplus_{v\not\in G} \mathcal{T}_{v}^{\Gr/G}) \oplus (\bigoplus_{v\not\in G} \mathcal{T}_{v}^{\Gr|_G}).
\] We define it by the quadratic monomial map $\mathcal{T}_{\{G\}}^{\Gr/G}\oplus \mathcal{T}_{\infty}^{\Gr|_G}\to \mathcal{T}_{\{G\}}^{\Gr/G}\otimes \mathcal{T}_{\infty}^{\Gr|_G}\cong \N_G$ and identity maps on the remaining summands.
\end{proof}

The open part $\M(\Gr)$ is identified with the moduli space of $\{\infty\}\cup V_{\Gr}$-pointed smooth genus zero curves with edge-non-intersecting conditions
\[
\M(\Gr)\cong \Pro(\Conf_{\Gr}(\mathbb{C})/\mathbb{G}_{a})\cong \Conf_{\Gr}(\mathbb{C})/\Aff(\mathbb{C})=\Conf_{\Gr}(\Pro^1\setminus \{\infty\})/(\PGL_2)_{\infty}\cong \Conf_{\mathsf{C}_{\infty}\Gr}(\Pro^1)/\PGL_2
\] The restriction of the line bundles $\{\mathcal{T}_v\}_{v\in \Cyc_{\infty}\Gr}$ to $\M(\Gr)$ coincide with tangent lines at the corresponding marked points. For a pair of distinct points $p,q\in \Pro^1$, authors~\cite[Sec.~8.1]{dupont2024logarithmic} define a canonical parallel transport $s_{p,q}\colon (T_p\Pro^1)^{\times}\to (T_{p'}\Pro^1)^{\times}$, in the case $p=\infty$, we have $s_{\infty,q}\colon \lambda \partial_{1/z}|_{\infty}\mapsto \lambda^{-1}\partial_{z}|_q$. This procedure produces an isomorphism of $\mathbb{G}_m$-fibrations $s_{\infty,v}\colon \mathcal{T}^{\circ}_{\infty}\to \mathcal{T}^{\circ}_{v}$ over $\M(\Gr)$. In particular, it defines a virtual isomorphism of log schemes
\begin{equation}
  s_{\infty,v}\colon(\bM(\Gr),\partial\bM(\Gr), \mathcal{T}_{\infty})\to (\bM(\Gr),\partial\bM(\Gr), \mathcal{T}_{v})  
\end{equation}
This procedure leads to the contractad structure on the unframed version $\frakFM$.
\begin{proposition} The contractad structure on $\bM$ and parallel transport induces the contractad structure on $\frakFM$ in the category of log schemes with virtual morphisms.
\end{proposition} 
\begin{proof} Similarly to the framed case, to define a contractad structure, it remains to construct a virtual monomial map
\[
\mathcal{T}^{\Gr/G}_{\infty}\oplus \mathcal{T}^{\Gr|_G}_{\infty}\to \N_G\oplus \mathcal{T}^{\Gr/G}_{\infty}\cong \mathcal{T}^{\Gr/G}_{\infty}\oplus (\mathcal{T}^{\Gr/G}_{\infty}\otimes \mathcal{T}^{\Gr|_G}_{\infty}).
\] We define it by combining identity on $\mathcal{T}^{\Gr/G}_{\infty}$ and the composition $\mathcal{T}^{\Gr/G}_{\infty}\oplus \mathcal{T}^{\Gr|_G}_{\infty}\overset{s_{\infty,\{G\}}}{\longrightarrow}\mathcal{T}^{\Gr/G}_{\{G\}}\oplus \mathcal{T}^{\Gr|_G}_{\infty}\to \mathcal{T}^{\Gr/G}_{\{G\}}\otimes \mathcal{T}^{\Gr|_G}_{\infty}$.
\end{proof} The canonical line bundles $\{\Ta_v\}_{v\in V_{\Gr}\cup \{\infty\}}$ are trivial over $\M(\Gr)$, so we have a virtual decomposition $\frakFM(\Gr)\cong (\bM(\Gr),\partial\bM(\Gr))\times *_{\log}$ and $\frakFM^{\fr}(\Gr)\cong (\bM(\Gr),\partial\bM(\Gr))\times *_{\log}^{\times V_{\Gr}\cup \{\infty\}}$. Thanks to Example~\ref{ex::logmorphisms}, $*_{\log}$ is a group object in log schemes with virtual morphisms, hence $\frakFM$ forms a contractad in the category of $*_{\log}$-spaces with respect to the decomposition above. We consider the semidirect contractad $\frakFM\rtimes *_{\log}$. Note that parallel transports $s_{\infty,v}$, for each vertex $v\in V_{\Gr}$, induce an isomorphism of contractads
\begin{equation}\label{eq::log_semidirect}
     \frakFM\rtimes *_{\log}\to \frakFM^{\fr}
\end{equation}
  in log schemes with virtual morphisms. All in all, we recover and generalise the statement~\cite[Pr.~8.18]{dupont2024logarithmic}
\begin{proposition}
Parallel transports induce the isomorphism of contractads
$\frakFM^{\fr}\cong \frakFM\rtimes *_{\log}$ in the category of log schemes with virtual morphisms.
\end{proposition}

\subsection{Kato-Nakayama realisation} 
\label{sec:sub:KN}
With a fine saturated\footnote{We refer to~\cite[Sec.~1.7]{nakayama1997logarithmic} for a precise definition} log scheme $\X=(X,\M_X)$ over $\mathbb{C}$, we associate a topological space $\KN(\X)$, called Kato-Nakayama realisation~\cite{kato1999log}. A point in $\KN(X)$ is a pair $(x,\lambda)$, where $x\in X(\mathbb{C})$ is a complex point and monoid map $\lambda_x\colon \M_{X,x}\to S^1\subset \mathbb{C}^{\times}$, such that $\lambda_x(f)=\frac{f(x)}{|f(x)|}$, for $f\in \Orb_X^{\times}$. It was shown in~\cite[Prop.~4.2]{dupont2024logarithmic}, that Kato-Nakayama realisation defines a strict monoidal functor from the category of fine saturated schemes with virtual morphisms
\[
\KN\colon \mathsf{virtLogSchm}_{\mathbb{C}}^{\mathrm{fs}}\to \Top.
\]
\begin{itemize}
    \item For a trivial scheme $\X^{\mathrm{triv}}$, we have $\KN(\X^{\mathrm{triv}})=X(\mathbb{C})$
    \item For a log-point, we have $\KN(*_{\log})\cong S^1$ as topological groups.
    \item For a divisorial log scheme $(X,D)$, the Kato-Nakayama realisation $\KN(X,D)=\mathsf{Bl}_D^{\mathbb{R}}(X(\mathbb{C}))$ is the real oriented blow-up of $X(\mathbb{C})$.
    \item More generally, for a Deligne-Faltings log scheme $(X,D,\oplus^k_{i=1}\La_i)$, its KN realisation is the product of principal $S^1$-fibrations $\Sph(p^*\La_1)\underset{\mathsf{Bl}_D^{\mathbb{R}}(X(\mathbb{C}))}{\times}\Sph(p^*\La_2)\underset{\mathsf{Bl}_D^{\mathbb{R}}(X(\mathbb{C}))}{\times}\cdots\underset{\mathsf{Bl}_D^{\mathbb{R}}(X(\mathbb{C}))}{\times} \Sph(p^*\La_k)$.
\end{itemize} Since $\KN$ is strict monoidal, it sends log scheme contractads to topological contractads. We consider the corresponding topological contractads
\[
\ufrM:=\KN(\frakFM), \quad \frM:=\KN(\frakFM^{\fr})
\]

Consider the oriented blow-up $p\colon \bM^{\mathrm{or}}(\Gr)\to \bM(\Gr)$ along the normal divisor $\partial\bM(\Gr)$. So, the components of the contractads above are products of the principal $S^1$-fibrations associated with pullbacks $p^*\mathcal{T}_v$
\[
\ufrM(\Gr)=\mathbb{S}(p^*\mathcal{T}_{\infty}),\quad \frM(\Gr)=\mathbb{S}(p^*\mathcal{T}_{\infty})\underset{\bM^{\mathrm{or}}}{\times} \mathbb{S}(p^*\mathcal{T}_{v_1})\underset{\bM^{\mathrm{or}}}{\times}\cdots \underset{\bM^{\mathrm{or}}}{\times} \mathbb{S}(p^*\mathcal{T}_{v_n})
\]
\begin{example} To make some intuition, we consider the modular case. According to Section~\ref{sec::modular}, for a complete multipartite graph $\K_{\lambda}$, via the isomorphism $\bM(\K_{\lambda})\cong \beM_{0,\K_{\lambda}}$, the oriented blow-up $\Bl_{\partial\beM}(\beM_{0,\K_{\lambda}})$ consists of equivalence classes of stable $\K_{\lambda}$-pointed curves $\Sigma$ with real rays in the tensor product $T_p\Sigma'\otimes T_p\Sigma''$ at every double point $p\in \Sigma$. In the framed case $\widetilde{\EuScript{M}}^{\fr}_{0,\K_{\lambda}}$, we also add tangent real rays at each marked point $x_v$, while in the unframed case $\widetilde{\EuScript{M}}_{0,\K_{\lambda}}$ we only add a ray at the marked point $x_{\infty}$.
\begin{figure}[ht]
\[
    \vcenter{\hbox{\begin{tikzpicture}[scale=0.5]
    \draw (0,0) circle (1);
    \draw (2,0) circle (1);
    \node at (1,0) {\small$\bullet$};
    \node [left] at (-1,0) {\small$\infty$};
    \draw (-1.1,0) -- (-0.9,0);
    \node [above] at (0,1) {\small$1$};
    \draw (0,1.1) -- (0,0.9);
    \node [below] at (2,-1) {\small$2$};
    \draw (2,-0.9) -- (2,-1.1);
    \node [right] at (3,0) {\small$3$};
    \draw (2.9,0) -- (3.1,0);
    \end{tikzpicture}}} \quad 
    \vcenter{\hbox{\begin{tikzpicture}[scale=0.5]
    \draw (0,0) circle (1);
    \draw (2,0) circle (1);
    \node at (1,0) {\small$\bullet$};
    \node [left] at (-1,0) {\small$\infty$};
    \draw (-1.1,0) -- (-0.9,0);
    \node [above] at (0,1) {\small$1$};
    \draw (0,1.1) -- (0,0.9);
    \node [below] at (2,-1) {\small$2$};
    \draw (2,-0.9) -- (2,-1.1);
    \node [right] at (3,0) {\small$3$};
    \draw (2.9,0) -- (3.1,0);
    \draw[->, > = latex] (-1,0)--(-1,1);
    \draw[->, > = latex] (1,0)--(1,-1);
    \end{tikzpicture}}} \quad 
    \vcenter{\hbox{\begin{tikzpicture}[scale=0.5]
    \draw (0,0) circle (1);
    \draw (2,0) circle (1);
    \node at (1,0) {\small$\bullet$};
    \node [left] at (-1,0) {\small$\infty$};
    \draw (-1.1,0) -- (-0.9,0);
    \node [above] at (0,1) {\small$1$};
    \draw (0,1.1) -- (0,0.9);
    \node [below] at (2,-1) {\small$2$};
    \draw (2,-0.9) -- (2,-1.1);
    \node [right] at (3,0) {\small$3$};
    \draw (2.9,0) -- (3.1,0);
    \draw[->, > = latex] (-1,0)--(-1,1);
    \draw[->, > = latex] (1,0)--(1,-1);
    \draw[->, > = latex] (0,1)--(1,1);
    \draw[->, > = latex] (2,-1)--(3,-1);
    \draw[->, > = latex] (3,0)--(3,1);
    \end{tikzpicture}}}
\]
    \caption{Stable curves in $\beM_{0,\K_3}, \widetilde{\EuScript{M}}_{0,\K_{3}}$ and $\widetilde{\EuScript{M}}^{\fr}_{0,\K_{3}}$}
\end{figure}

The contractad structure in the framed case $\widetilde{\EuScript{M}}^{\fr}_{0,\K_{\lambda}}$ is clear. Namely, to compose a framed stable curve $\Sigma_1\in \widetilde{\EuScript{M}}^{\fr}_{0,\K_{\lambda}/G}$ with a framed curve $\Sigma'\in \widetilde{\EuScript{M}}^{\fr}_{0,\K_{\lambda}|_G}$, one glues $\Sigma$ and $\Sigma'$ at the points $x_{\{G\}}$ and $x'_{\infty}$ and tensors the corresponding tangent rays to obtain a ray in the gluing space $T_{\{G\}}\Sigma\otimes T_{\infty}\Sigma'$. If we restrict ourselves to the case of complete graphs, we recover the so-called Kimura-Stasheff-Voronov operad~\cite{kimura1995operad}. 

In the unframed case $\widetilde{\EuScript{M}}_{0,\K_{\lambda}}$, all the same, except that we transport a tangent ray at $\infty$ on $\Sigma$ to a tangent ray at $\{G\}$ using parallel transports and gluing data given in double points, see~\cite[Sec.~8.1]{dupont2024logarithmic}.
\end{example}
Recall that the Kimura-Stasheff-Voronov operad mentioned above is isomorphic to the framed Fulton-MacPherson contractad~\cite{giansiracusa2012cyclic}. A similar result holds in the graphical case.

\begin{theorem}\label{thm::KN_is_FM}
The Kato-Nakayama realisation of (framed) log little disks contractad is isomorphic to the (framed) Fulton-MacPherson contractad
\[
\ufrM=\KN(\frakFM)\cong \FM_2, \quad \frM=\KN(\frakFM^{\fr})\cong \fFM_2.
\]
\end{theorem}
\begin{proof} The proof of the isomorphism $\KN(\frakFM)\cong \FM_2$ repeats the ones of~\cite[Thm.~8.20]{dupont2024logarithmic}, so we just sketch the main ideas. The stratification of $\bM(\Gr)$ transfers to the stratification of $\ufrM(\Gr)$ with respect to $\KN$-realisation: $\M((T))\mapsto \KN(\frakFM(\Gr)|_{\M((T))})$. The corresponding maximal strata of $\ufrM(\Gr)$  is isomorphic to the ones of $\FM_2(\Gr)$
\[\KN(\frakFM|_{\M(\Gr)})\cong \M(\Gr)\times S^1\cong \Conf_{\Gr}(\mathbb{C})/\Aff(\mathbb{C})\times S^1\cong \Conf_{\Gr}(\mathbb{R}^2)/\Aff_0(\Real^2)\cong \NConf_{\Gr}(\mathbb{R}^2)\]
In particular, on the level of set contractads, we have
\[
\ufrM\cong \T(\KN(\frakFM|_{\M}))\cong \T(\NConf(\Real^2))\cong \FM_2.
\]So, both $\ufrM$ and $\FM_2$ are manifolds with corners with diffeomorphic stratification and the same inclusion posets of the closures of strata elements, which implies the isomorphism of topological contractads. In the framed case, the isomorphism follows from the compatibility of KN realisation with semidirect products
\[
\KN(\frakFM^{\fr})\overset{\eqref{eq::log_semidirect}}{=} \KN(\frakFM\rtimes *_{\log})\cong \KN(\frakFM)\rtimes \KN(*_{\log})\cong \FM_2\rtimes S^1\cong \fFM_2.
\]
\end{proof}

\subsection{Log forms and Formality} 
\label{sec:sub:log:formality}
For a sufficiently good\footnote{ It suffices to consider ideally smooth, fine and saturated log schemes } scheme $\X=(X,\M_X)$, there is a complex of sheaves of log de Rham forms $\Omega_{\X}^{\bullet}$ obtained from the usual forms on the underlying scheme $\Omega^{\bullet}_X$ by adding logarithmic forms $d\log(f)$ for $f\in \M^{\mathrm{gr}}_X$. Its hypercomology is called de Rham cohomology and is denoted $H^{\bullet}_{dR}(\X):=\mathbb{H}^{\bullet}(\Omega_{\X}^{\bullet}, d)$.

The \textit{Betti to de Rham comparison}~\cite{kato1999log} asserts that there is a canonical isomorphism of symmetric monoidal functors $H_{dR}^{\bullet}(\X)\cong H^{\bullet}_{B}(\X;\mathbb{C})$, where $H^{\bullet}_{B}(\X;\mathbb{C})$ is the usual singular homology of the Kato-Nakayama realization $\KN(\X)$. In~\cite{vaintrob2021formality, dupont2024logarithmic}, authors refined this statement to the case of complexes. In~\cite[Sec.~6.2]{vaintrob2021formality}, Vaintrob showed that the derived global forms $\mathbb{R}\Gamma(\X, \Omega^{\bullet}_{\X})$ (we consider Thom-Whitney normalization of the Godement resolution for functorial reasons) and complex of singular cochains $C^{\bullet}_{Betti}(\KN(X);\mathbb{C})$ are connected by a chain of symmetric lax-monoidal quasi-isomorphisms as dg functors from the category of log schemes with ordinary morphisms. In~\cite[Sec.~6]{dupont2024logarithmic}, the authors explained that this statement remains true on the level of virtual morphisms. In particular, if $\mathfrak{O}$ is a contractad in the category of log schemes with virtual/ordinary morphisms, then we can take $\mathbb{R}\Gamma(\mathfrak{O};\Omega^{\bullet}_{\mathfrak{O}})$ as the complex dg Hopf model for its Kato-Nakayam realization $\Orb:=\KN(\mathfrak{O})$.

For an acyclic\footnote{log schemes $\X$ with $\Omega^k_{\X}$ are acyclic for all $k\geq 0$} log-sheme $\X$, we could replace derived global forms with usual ones. So, the restrictions of dg functors $\Gamma(-,\Omega^{\bullet}_{-})$ and $C^{\bullet}_{sing}(\KN(-),\mathbb{C})$ to the full subcategory of acyclic log schemes (with ordinary/virtual morphisms) are connected by a chain of quasi-isomorphisms~\cite[Thm.~6]{vaintrob2021formality}. If we also require the underlying scheme $X$ to be proper, then by the degeneracy of the Hodge-to-de Rham spectral sequence $H^p(X,\Omega^q_{\X})\Rightarrow H^{\bullet}_{dR}(\X)$ on the first page, the differential in global forms is zero, so $H^{\bullet}(\X)=\Gamma(X, \Omega_{\X})$. All in all, the restriction of the Betti cochain functor $C^{\bullet}_{sing}(\KN(-))$ to proper acyclic schemes is formal as a symmetric monoidal dg functor~\cite[Thm.~7]{vaintrob2021formality}. In particular, we obtain the following formality criteria

\begin{theorem}[Log formality criteria]\label{thm::formality_log_criteria}
Let $\mathfrak{O}$ be a contractad in the category of log schemes with virtual/ordinary morphisms, and $\Orb:=\KN(\mathfrak{O})$ be the corresponding topological contractad. If each component of the log-contractad $\mathfrak{O}$ is a proper acyclic log scheme, then the contractad $\Orb$ is formal over $\mathbb{C}$. 
\end{theorem}
Let us prove the main theorem of this section.
\begin{theorem}
\label{thm:D2:log:formal}
The contractads $\fD_2$ and $\D_2$ are formal over $\mathbb{C}$.
\end{theorem}
\begin{proof}
Thanks to Theorem~\ref{thm::FMn_is_En}, Fulton-MacPherson contractad $\FM_2$  is $E_2$-contractad. Thanks to Theorem~\ref{thm::KN_is_FM}, the Fulton-MacPherson contractad is a Kato-Nakayama realisation of $\frakFM$. So, by Theorem~\ref{thm::formality_log_criteria}, it suffices to prove that, for each graph $\Gr$, the log schemes $\frakFM(\Gr)$ are proper and acyclic.

Recall that $\bM(\Gr)$ is a smooth projective variety, so it remains to verify the acyclic condition. Recall that the line bundle $\mathcal{T}_{\infty}$ is trivial over $\M(\Gr)$, so we have a virtual isomorphism $\frakFM(\Gr)\cong (\bM(\Gr),\partial\bM(\Gr))\times *_{\log}$ according to Example~\ref{ex::logmorphisms}. Since the log point is proper acyclic, it suffices to show that the divisoral log scheme $(\bM(\Gr),\partial\bM(\Gr))$ is. Thanks to~\cite[Lem~4.3.3]{khoroshkin2024hilbert}, the mixed Hodge structure of $H^k(\M(\Gr))$ is pure of weight $2k$, in other words, $\M(\Gr)$ is $2$-pure. If the complement $X\setminus D$ is $2$-pure, then the divisorial log scheme $(X,D)$ is acyclic~\cite[Lemma~11]{vaintrob2021formality}. This completes the proof. The framed case directly follows from the unframed ones.
\end{proof}

\section{Combinatorial model} 
\label{sec::combinatorial}

In this section, we construct a conjectural combinatorial model $\AD_d$ for the little disks contractad in the category of posets. The model $\AD_d$ is a natural generalization of the operad of acyclic directions on complete graphs introduced by Berger~\cite{berger1997real}. Its components are given by edge-weighted acyclic directions on graphs. We prove that $\AD_d$ is an $E_d$-contractad for $d=1,2$ and for chordal graphs. Most of the constructions and proofs follow, \textit{mutatis mutandis}, those of~\cite{beuckelmann2022small}. The case $d=2$, however, requires additional arguments, and it remains unclear whether the statement extends to $d>2$.

\subsection{Posets and contractads}
\label{sec:sub:Poset}

Let us briefly discuss contractads in the category $\mathsf{Pos}$ of partially ordered sets (posets). Most of the constructions and homotopy properties of poset contractads repeat those in the operad case~\cite[Sec.~3]{beuckelmann2022small}.

For a poset $\mathsf{P}$, we let $|\mathsf{P}|$ be the geometric realization of its nerve $\mathsf{N}(\mathsf{P})$. We say that two posets $\mathsf{P}$ and $\mathsf{Q}$ are \emph{homotopy equivalent} if their geometric realizations are weakly homotopy equivalent $|\mathsf{P}|\simeq |\mathsf{Q}|$.

For a map $\phi\colon X\to \mathsf{P}$ from a topological space to a poset, we consider a poset $X_{\mathsf{P}}=\{(x,p)|\phi(x)\leq p\}$ with an obvious order. We have the projections of spaces $X\leftarrow |X_{\mathsf{P}}|\to |\mathsf{P}|$. The left projection is always a homotopy equivalence. If for each $p\in \mathsf{P}$, the subspace $\phi/p=\{x|\phi(x)\leq p\}$ is a weakly contractible retract of a CW-complex, then projection $X_{\mathsf{P}}\to \mathsf{P}$ induces an equivalence of posets and hence $\phi$ induces the equivalence $X\simeq |\mathsf{P}|$. 

A poset contractad is a contractad in the category of partially ordered sets (posets) $\mathsf{Pos}$.  By monoidal reasons, the geometric realisation of a poset contractad $\Pop$ defines a topological contractad $|\Pop|$. We say that poset contractads are \emph{homotopy equivalent} if their geometric realizations are weakly homotopy equivalent contractads. Similarly, for a morphism $\psi\colon \Pop\to \Q$ from topological to poset contractad, we can construct the poset contractad $\R$ with the projections $\Pop\overset{\simeq}{\leftarrow} |\R|\rightarrow |\Q|$ and the left projection is an equivalence. We say that $\psi$ induces an equivalence $\Pop\simeq |\Q|$ if the right map $|\R|\to |\Q|$ is also an equivalence and hence $\Pop$ and $|\Q|$ are equivalent as topological contractads in the sense of Section~\ref{sec:sub:FM_n=E_n}.

\subsection{Acyclic directions contractad}
\label{sec:sub:Acyclic:contractad}

A \emph{direction} of a graph is an assignment of an orientation to each edge. Such a direction is called \emph{acyclic} if it contains no directed cycles of the form $v_1 \to v_2 \to v_3 \to \cdots \to v_k \to v_1$. For a graph $\Gr$, we denote by $\AD_d(\Gr)$ the set of acyclic directions of the underlying graph, where the edges are labeled by the set $[d]=\{1<2<\cdots <d\}$.

The sets $\mathsf{AD}_d(\Gr)$ form a contractad in the category of sets with respect to the composition $\alpha \circ^{\Gr}_G \beta$, defined for weighted directions $\alpha \in \mathsf{AD}_d(\Gr/G)$ and $\beta \in \mathsf{AD}_d(\Gr|_G)$ by the following rule:
\[
\begin{cases}
	a \overset{{\mathrm{f}}}{\to} a' & \text{if } a,a' \notin G \text{ and } a \overset{{\mathrm{f}}}{\to} a' \text{ in } \alpha, \\
	b \overset{{\mathrm{e}}}{\to} b' & \text{if } b,b' \in G \text{ and } b \overset{{\mathrm{e}}}{\to} b' \text{ in } \beta, \\
	a \overset{{\mathrm{f}}}{\to} b & \text{if } a \notin G, b \in G \text{ and } a \overset{{\mathrm{f}}}{\to} \{G\} \text{ in } \alpha, \\
	b \overset{{\mathrm{f}}}{\to} a & \text{if } a \notin G, b \in G \text{ and } \{G\}\overset{{\mathrm{f}}}{\to}a \text{ in } \alpha.
\end{cases}
\]

\begin{example}
	For the cycle $\Cyc_4$ and tube $G=\{1,2\}$, we have
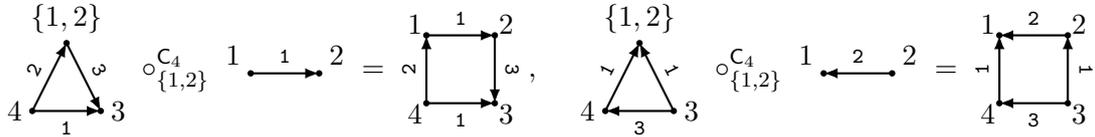
\[
\vcenter{\hbox{\begin{tikzpicture}[scale=0.6, edge/.style={->,> = latex, thick}]
			\fill (0,0) circle (2pt);
			\node [left] at (0,0) {$4$};
			\fill (1.5,0) circle (2pt);
			\node [right] at (1.5,0) {$3$};
			\fill (0.75,1.5) circle (2pt);
			\node [above] at (0.75,1.5) {$\{1,2\}$};
			\draw[edge] (0.75,1.5)--(1.5,0) node[midway,above, sloped] () {\scriptsize$\mathtt{3}$};
			\draw[edge] (0,0)--(1.5,0) node[midway,below, sloped] () {\scriptsize$\mathtt{1}$};
			\draw[edge] (0,0)--(0.75,1.5) node[midway,above, sloped] () {\scriptsize$\mathtt{2}$};
\end{tikzpicture}}} \circ^{\Cyc_4}_{\{1,2\}}
\hbox{\begin{tikzpicture}[scale=0.6, edge/.style={->,> = latex, thick}]
		\fill (0,0) circle (2pt);
		\fill (1.5,0) circle (2pt);
		\node [left] at (0,0.4) {$1$};
		\node [right] at (1.5,0.4) {$2$};
		\draw[edge] (0,0)--(1.5,0) node[midway,above, sloped] () {\scriptsize$\mathtt{1}$};
\end{tikzpicture}}=
\vcenter{\hbox{\begin{tikzpicture}[scale=0.6, edge/.style={->,> = latex, thick}]
			\fill (0,0) circle (2pt);
			\fill (0,1.5) circle (2pt);
			\fill (1.5,0) circle (2pt);
			\fill (1.5,1.5) circle (2pt);
			\draw[edge] (0,1.5)--(1.5,1.5) node[midway,above, sloped] () {\scriptsize$\mathtt{1}$};
			\draw[edge] (1.5,1.5)--(1.5,0) node[midway,above, sloped] () {\scriptsize$\mathtt{3}$};
			\draw[edge] (0,0)--(1.5,0) node[midway,below, sloped] () {\scriptsize$\mathtt{1}$};
			\draw[edge] (0,0)--(0,1.5) node[midway,above, sloped] () {\scriptsize$\mathtt{2}$};
			\node at (-0.25,1.75) {$1$};
			\node at (1.75,1.75) {$2$};
			\node at (1.75,-0.25) {$3$};
			\node at (-0.25,-0.25) {$4$};
\end{tikzpicture}}}, \quad  \vcenter{\hbox{\begin{tikzpicture}[scale=0.6, edge/.style={->,> = latex, thick}]
			\fill (0,0) circle (2pt);
			\node [left] at (0,0) {$4$};
			\fill (1.5,0) circle (2pt);
			\node [right] at (1.5,0) {$3$};
			\fill (0.75,1.5) circle (2pt);
			\node [above] at (0.75,1.5) {$\{1,2\}$};
			\draw[edge] (1.5,0)--(0.75,1.5) node[midway,above, sloped] () {\scriptsize$\mathtt{1}$};
			\draw[edge] (1.5,0)--(0,0) node[midway,below, sloped] () {\scriptsize$\mathtt{3}$};
			\draw[edge] (0,0)--(0.75,1.5) node[midway,above, sloped] () {\scriptsize$\mathtt{1}$};
\end{tikzpicture}}} \circ^{\Cyc_4}_{\{1,2\}}
\hbox{\begin{tikzpicture}[scale=0.6, edge/.style={->,> = latex, thick}]
		\fill (0,0) circle (2pt);
		\fill (1.5,0) circle (2pt);
		\node [left] at (0,0.4) {$1$};
		\node [right] at (1.5,0.4) {$2$};
		\draw[edge] (1.5,0)--(0,0) node[midway,above, sloped] () {\scriptsize$\mathtt{2}$};
\end{tikzpicture}}=
\vcenter{\hbox{\begin{tikzpicture}[scale=0.6, edge/.style={->,> = latex, thick}]
			\fill (0,0) circle (2pt);
			\fill (0,1.5) circle (2pt);
			\fill (1.5,0) circle (2pt);
			\fill (1.5,1.5) circle (2pt);
			\draw[edge] (1.5,1.5)--(0,1.5) node[midway,above, sloped] () {\scriptsize$\mathtt{2}$};
			\draw[edge] (1.5,0)--(1.5,1.5);
			\node [rotate=-90] at (1.9,0.75){\scriptsize$\mathtt{1}$};
			\draw[edge] (1.5,0)--(0,0) node[midway,below, sloped] () {\scriptsize$\mathtt{3}$};
			\draw[edge] (0,0)--(0,1.5) node[midway,above, sloped] () {\scriptsize$\mathtt{1}$};
			\node at (-0.25,1.75) {$1$};
			\node at (1.75,1.75) {$2$};
			\node at (1.75,-0.25) {$3$};
			\node at (-0.25,-0.25) {$4$};
\end{tikzpicture}}}
\]
\end{example}

\begin{lemma}
	\label{lem::AD::poset}	
	The sets $\AD_d(\Gr)$ admit a natural partial order:
	\[
	\alpha_1 \leq \alpha_2 \in \AD_d(\Gr) \quad \stackrel{\mathrm{def}}{\Leftrightarrow} \quad  
	\forall \text{ edge } v \overset{{\mathrm{e}}}{\to} w \in \alpha_1, \ 
	\begin{cases}
		v \overset{{\mathrm{f}}}{\to} w \in \alpha_2 \text{ with } {\mathrm{e}} \leq {\mathrm{f}}, \\
		v \overset{{\mathrm{f}}}{\leftarrow} w \in \alpha_2 \text{ with } {\mathrm{e}} < {\mathrm{f}}.
	\end{cases}
	\]
	This order is compatible with the contractad composition maps.
\end{lemma}	
\begin{proof}
	Direct verification.
\end{proof}

In other words, Lemma~\ref{lem::AD::poset} shows that $\AD_d(\Gr)$ is a contractad in the category of posets. 

For $d=1$, the set $\AD_1(\Gr)$ consists of acyclic orientations of $\Gr$, and any two distinct acyclic orientations are incomparable. In particular, we obtain the isomorphism $|\AD_1| = \Ass$ with the associative contractad. 	

Next, we describe the relation of these contractads to the little disks contractad. Consider the morphism of (set-theoretic) graphical collections
\[
\psi_d\colon \Conf(\mathbb{R}^d)\to \AD_d,
\]
which sends a configuration $x\in \Conf_{\Gr}(\mathbb{R}^d)$ to a weighted direction $\psi_d(x)$ defined as follows: 
$$\textit{ we set $v \overset{i}{\to} w$ if $i$ is the smallest index with $x^{(i)}_v \neq x^{(i)}_w$, and in this case $x^{(i)}_v < x^{(i)}_w$.} 
$$
\label{page::description_of_proper}Let us describe the image of this map. For a subcycle $\Cyc\subset \Gr$ on $n\geq 3$ vertices, we consider one of the two possible cyclic orientations of its edges ( clockwise and counterclockwise respectively). An acyclic weighted direction of $\alpha$ defines a decomposition of the edge set of $\Cyc$ into two non-empty subsets $E_{\Cyc}=\Cyc^+(\alpha)\coprod \Cyc^-(\alpha)$, where $\Cyc^{s}(\alpha)$ consists of edges whose orientations with respect to the cyclic orientation of $\Cyc$ and the acyclic direction $\alpha$ coincide for $s=+$ and do not coincide for $s=-$. We shall refer to a cyclic orientation of a cycle as internal to distinguish it from an acyclic orientation of a graph.

We call an acyclic weighted direction $\alpha$ of a graph $\Gr$ \emph{proper} if for every subcycle $\Cyc=\{v_i\}_{i\in \mathbb{Z}_n}\subset V_{\Gr}$, we have $\min_{e\in \Cyc^+(\alpha)} w(e)=\min_{e'\in \Cyc^-(\alpha)} w(e')$.  Below is an example of the proper condition for a weighted acyclic direction of $\Cyc_6$
\[
\vcenter{\hbox{\begin{tikzpicture}[scale=0.7, edge/.style={->,> = latex, thick}]
    \fill (-0.63,1.075)  circle (2pt);
    \fill (0.63,1.075)  circle (2pt);
    \fill (-1.22,0) circle (2pt);
    \fill (1.22,0) circle (2pt);
    \fill (-0.63,-1.075)  circle (2pt);
    \fill (0.63,-1.075)  circle (2pt);

    \draw[edge] (-0.63,1.075)--(0.63,1.075) node[midway,above, sloped] () {\scriptsize$\mathrm{w}_1$};
     \draw[edge] (1.22,0)--(0.63,1.075) node[midway,above, sloped] () {\scriptsize$\mathrm{w}_2$};
     \draw[edge] (1.22,0)--(0.63,-1.075) node[midway,below, sloped] () {\scriptsize$\mathrm{w}_3$};
     \draw[edge] (0.63,-1.075)--(-0.63,-1.075)node[midway,below, sloped] () {\scriptsize$\mathrm{w}_4$};
     \draw[edge] (-1.22,0)--(-0.63,-1.075) node[midway,below, sloped] () {\scriptsize$\mathrm{w}_5$};
     \draw[edge] (-0.63,1.075)--(-1.22,0) node[midway,above, sloped] () {\scriptsize$\mathrm{w}_6$};
    \draw[->] (240:0.4) arc[start angle=240, end angle=-60, radius=0.4];    
    \end{tikzpicture}}} \Rightarrow \begin{aligned}[c]
  \Cyc^+(\alpha)=\{e_1,e_3,e_4\}
  \\
  \Cyc^-(\alpha)=\{e_2,e_5,e_6\} 
\end{aligned} \Rightarrow \min\{\mathrm{w}_1,\mathrm{w}_3,\mathrm{w}_4\}= \min\{\mathrm{w}_2,\mathrm{w}_5,\mathrm{w}_6\}
\]
We denote by $\Pro_d(\Gr)\subset \AD_d(\Gr)$ the subcollection of proper acyclic directions.
\begin{lemma}\label{lemma::proper=realisable}
A weighted acyclic direction is in the image of $\psi_d$ if and only if it is proper, i.e., we have
\[
\psi_d(\Conf(\mathbb{R}^d))=\Pro_d
\]
\end{lemma}
\begin{proof} ($\Rightarrow$) For a point $x=(x^{(k)}_v)_{v\in V_{\Gr}}^{k=1,\cdots,d}\in \Conf_{\Gr}(\Real^d)$, let $\alpha=\psi_d(x)$ be the corresponding acyclic direction with weights $w_x=\{w_x(e)\}_{e\in E_{\Gr}}$. Let $\Cyc=\{v_i\}_{i\in\mathbb{Z}_n}\subset \Gr$ be a subcycle with the internal orientation $v_i\to v_{i+1}$ for all $i\in \mathbb{Z}_n$, and we consider the numbers $k_{\pm}=\min_{e\in \Cyc^{\pm}(\alpha)} w_x(e)$. We want to show that $k_-=k_+$. By the construction, we have
\begin{gather*}
\forall l<k_{\pm}, (v_i,v_{i+1})\in \Cyc^{\pm}(\alpha)\colon x^{(l)}_i=x^{(l)}_{i+1}
\\
\forall (v_i,v_{i+1})\in \Cyc^{\pm}(\alpha)\in \mathbb{Z}_n\colon \pm x^{(k_{\pm})}_i\geq \pm x^{(k_{\pm})}_{i+1}
\\
\exists (v_i,v_{i+1})\in \Cyc^{\pm}(\alpha) \colon \pm x^{(k_{\pm})}_i> \pm x^{(k_{\pm})}_{i+1}.
\end{gather*}
If $k_{-}<k_{+}$, then we get the following contradiction 
\[
0=\sum_{i\in \mathbb{Z}_n}x^{(k_-)}_i-x^{(k_-)}_{i+1}=\sum_{(v_i,v_{i+1})\in \Cyc^-(\alpha)} x^{(k_-)}_i-x^{(k_-)}_{i+1}<0.
\] So, we have $k_-\geq k_+$. For similar reasons, we deduce $k_-\leq k_+$, so $k_-=k_+$.

($\Leftarrow$) Let $\alpha$ be an acyclic direction of $\Gr$. From the theory of oriented matroids~\cite[Sec.~1.2(c)]{bjorner1999oriented}, for a subset $S\subset E_{\Gr}$ of edges, the region in $\Real^{V_{\Gr}}$ cut out by (in)equalities
\[
\bigcap_{\substack{ v\to w \text{ in }\alpha \\(v,w)\in S}} \{x_v>x_w \} \cap \bigcap_{\substack{ v\to w \text{ in }\alpha \\(v,w)\not\in S}} \{x_v=x_w \}
\] is non-empty if and only if, for each subcycle $\Cyc\subset \Gr$ with $E_{\Cyc}\cap S\neq \varnothing$, the both intersections $\Cyc^+(\alpha)\cap S$ and $\Cyc^-(\alpha)\cap S$ are non-empty.

Let $\alpha$ be a proper weighted acyclic direction of $\Gr$. To check that $g$ belongs to the image it is sufficient to check that the region $\R(\alpha)\subset (\Real^d)^{V_{\Gr}}$ cut out by the following (in)equalities
\[
\R(\alpha):=\bigcap_{v\overset{\mathrm{f}}{\to} w\in \alpha} \{x| x^{(i)}_v=x^{(i)}_v\text{, for }i<\mathrm{f}\text{, and }x^{(i)}_v>x^{(i)}_w\text{, for }i\geq \mathrm{f}\}
\] is non-empty. Note that $\R(\alpha)$ is non-empty iff the images $\R_i(\alpha):=\pi_i(\R(\alpha))\subset \Real^{V_{\Gr}}$ with respect to the coordinates $\pi_i\colon (\Real^d)^{V_{\Gr}}\to \Real^{V_{\Gr}}$ are non-empty for all $i$. The region $\R_i(\alpha)$ has the form
\[
\R_i(\alpha):=\bigcap_{\substack{v\overset{\mathrm{f}}{\to}w \\ \mathrm{f}\leq i}} \{x_v>x_w\}\cap\bigcap_{\substack{v\overset{\mathrm{f}}{\to}w \\ \mathrm{f}>i}} \{x_v=x_w\} ,
\] hence $\R_i(\alpha)$ is non-empty iff for every subcycle $\Cyc\subset \Gr$ with $E^{\leq i}_{\Gr}\cap \Cyc\neq \varnothing$, where $E^{\leq i}_{\Gr}$ consist of edges of weight at most $i$, we have $\Cyc^{\pm}(\alpha)\cap E^{\leq i}_{\Gr}\neq \varnothing$. Let us verify that a proper weighted acyclic direction $\alpha$ satisfy this condition. For a cycle $\Cyc$, let $w_0=\min \{w(e)|e\in \Cyc\}$. If $i<w_0$, then $E^{\leq i}_{\Gr}\cap \Cyc=\varnothing$. If $i\geq w_0$, since $\alpha$ is proper, there are at least two edges $e^{\pm}\in \Cyc^{\pm}(\alpha)$ of weight $w_0$.
\end{proof}
\begin{remark} The proof of Lemma~\ref{lemma::proper=realisable} relies heavily on ideas from oriented matroids theory, specifically oriented graphic matroids. In fact, the poset $\Pro_2(\Gr)$ is isomorphic to the face lattice of Salvetti complex of the graphic matroid~\cite{salvetti1987topology}.
\end{remark}
\begin{proposition}\label{prop::conf_to_P} 
	The map $\psi_d\colon \Conf(\mathbb{R}^d)\to \Pro_d$ of (set-theoretic) graphical collections induces an equivalence
	\[
	\Conf(\mathbb{R}^d) \simeq |\Pro_d|
	\]
	of topological graphical collections.
\end{proposition}
\begin{proof}
	It suffices to show that for each $\alpha \in \Pro_d(\Gr)$, the subspace 
	\[
	\psi/\alpha := \{x \mid \phi(x)\leq \alpha\}
	\]
	is contractible. The proof of contractibility follows the same lines as in~\cite[Thm.~5.8]{beuckelmann2022small}.
\end{proof}

\begin{conjecture}
	\label{conj::ADn_is_En}
	The map $\psi_d\colon \Conf(\mathbb{R}^d)\to \AD_d$ induces a homotopy equivalence
	\[
	\Conf_{\Gamma}(\mathbb{R}^d)\simeq |\AD_d(\Gamma)|
	\]
	of topological graphical collections.
\end{conjecture}

\begin{theorem}
	Conjecture~\ref{conj::ADn_is_En} holds in the following cases:
	\begin{itemize}[itemsep=0pt, topsep=0pt]
		\item for any $d$ in the class of chordal graphs in the sense of Section~\ref{sec::restriction_to_chordal}.,
		\item for $d=1,2$ and arbitrary graphs $\Gr$.
	\end{itemize}
\end{theorem}
\begin{proof}
	The case of complete graphs $\K_k$ was established in~\cite[Thm.~5.10]{beuckelmann2022small} using induction on $k$, with the induction step relying on the fibration $\Conf_k(\Real^d)\to \Conf_{k-1}(\Real^d)$. In the graphical setting, as discussed in Section~\ref{sec::chordal_cycles}, this fiber argument only works for simplicial vertices, which restricts the generalization to chordal graphs.
	
	The case $d=1$ was described in Section~\ref{sec::polytopes}, where $\Conf_{\Gamma}(\mathbb{R})$ was shown to be a disjoint union of contractible components indexed by acyclic orientations of $\Gamma$.
	
	The remaining case $d=2$ is addressed in detail in Section~\ref{sec:sub:acycl:2d}
	(Corollary~\ref{cor::p-ad}).
\end{proof}

\begin{theorem}\label{thm::implies_En}
	For $d \geq 1$, if $\psi_d\colon \Conf(\mathbb{R}^d)\to \AD_d$ induces a homotopy equivalence of graphical collections $\Conf(\mathbb{R}^d)$ and $|\AD_d|$, then $\AD_d$ is a $E_d$-contractad.    
\end{theorem}
\begin{proof}
	The argument follows~\cite[Thm.~9.1]{beuckelmann2022small}. Briefly, one defines the Boardman–Vogt resolution $\mathsf{W}\Pop \to \Pop$ for poset contractads by replacing the interval $[0,1]$ with the poset interval $[1]=\{0<1\}$. Explicitly, $\mathsf{W}\Pop$ is the free contractad $\T(\Pop)$, ordered by contracting edges in rooted trees and comparing $\Pop$-decorated vertices accordingly.
	
	We construct the morphism of contractads
	\[
	\FM_d \to \mathsf{W}(\AD_d), \qquad
	\FM_d((T)) \cong \prod_{v\in \Ver(T)} \NConf_{\In(v)}(\mathbb{R}^d)\overset{\psi}{\longrightarrow} \prod_{v\in \Ver(T)} \AD_d(\In(v)) = \mathsf{W}(\AD_d)((T)).
	\]
	By analogy with the operad case, the equivalence $\psi\colon \Conf(\mathbb{R}^d)\to \AD_d$ implies equivalence of contractads
	\(
	|\mathsf{W}(\AD_d)| \simeq \FM_d.
	\)
	Since $\mathsf{W}(\Pop)\simeq \Pop$ for poset operads, and by Theorem~\ref{thm::FMn_is_En}, we obtain the desired equivalence
	\[
	|\AD_d| \simeq |\mathsf{W}(\AD_d)| \simeq \FM_d \simeq \D_d.
	\]
\end{proof}

\begin{corollary}\label{cor::ADisEn}
	For $d=1,2$, and for all $d$ in the class of chordal graphs, $\AD_d$ is an $E_d$-contractad.
\end{corollary}

\subsection{Two-dimensional case}
\label{sec:sub:acycl:2d}

In this section, we prove Conjecture~\ref{conj::ADn_is_En} for $d=2$.
We work with the extended acyclic directions contractad $\AD_d^{\mathrm{ext}}$, defined in the same way as $\AD_d$, except that the acyclicity condition is relaxed to exclude only directed cycles of uniform weight. This construction is a graphical generalization of the operad introduced in~\cite{brun2007multiplicative}. As in the operadic case, we obtain:

\begin{proposition}\label{prop::ADext}
	The inclusion $j\colon \AD_d \to \AD_d^{\mathrm{ext}}$ is an equivalence of poset contractads.
\end{proposition}

\begin{proof}
	The proof follows the same pattern as in~\cite[Thm.~4.5]{beuckelmann2022small}. By Quillen's Theorem A~\cite{quillen2006higher}, it is enough to show that the geometric realization of the \emph{overposet}
	\[
	j/\alpha := \{x \in \AD_d(\Gr) \colon j(x) \leq \alpha\} = \downarrow \alpha \cap \AD_d(\Gr)
	\]
	is contractible for each $\alpha \in \AD_d^{\mathrm{ext}}(\Gr)$.  
	
	If $\alpha \in \AD_d(\Gr)$, the claim is immediate, and hence it holds for all minimal elements of $\AD_d^{\mathrm{ext}}(\Gr)$. We proceed by induction on the partial order of $\AD_d^{\mathrm{ext}}(\Gr)$.
	
	Suppose $\alpha$ lies in the complement $\AD_d^{\mathrm{ext}}(\Gr)\setminus \AD_d(\Gr)$. For a directed cycle
	\[
	\Cyc = v_1 \overset{\mathsf{w}_1}{\to} v_2 \to \cdots \to v_{k} \overset{\mathsf{w}_k}{\to} v_1
	\]
	in $\alpha$, define
	\[
	\mu(\Cyc) := \bigl(\mu_d(\Cyc), \ldots, \mu_1(\Cyc)\bigr),
	\]
	where $\mu_i(\Cyc)$ is the number of edges of weight $i$ in $\Cyc$.  
	Let $\Cyc_0$ be a minimal cycle in $g$ with respect to $\mu$-lexicographic order.
    Denote by $\Cyc^{>1}_0$ the set of non-unital edges of $\Cyc_0$. For any nonempty $\Susp \subset \Cyc^{>1}_0$, define $\alpha_{\Susp}$ to be the direction obtained from $\alpha$ by replacing each edge $v_i \overset{k}{\to} v_{i+1}$ with $v_i \overset{k-1}{\leftarrow} v_{i+1}$ whenever $i \in \Susp$.

    Note that $\alpha_{\Susp}$ is an uniformly acyclic orientation for any $\Susp\subset \Cyc^{>1}_0$. Assume the contrary, i.e., that $\alpha_{\Susp}$ contains a cycle $\Cyc'$ of uniform weight $k$. Suppose first that $\Cyc_0 \cap \Cyc'$ is a consecutive sequence of edges $\{(v_i,v_{i+1}), (v_{i+1},v_{i+2}), \ldots, (v_j,v_{j+1})\}$. Then $\alpha$ would contain a directed cycle $\Cyc'' = \Cyc_0 \cup \Cyc' \setminus (\Cyc_0 \cap \Cyc')$ with $\mu(\Cyc'') < \mu(\Cyc_0)$, contradicting the $\mu$-minimality of $\Cyc_0$. Alternatively, if $\Cyc_0 \cap \Cyc'$ contains two edges $(v_i,v_{i+1})$ and $(v_j,v_{j+1})$ with no edges of $\Cyc_0$ between them, then combining the corresponding segments of $\Cyc_0$ and $\Cyc'$ produces a directed cycle $\Cyc''$ in $\alpha$ with $\mu(\Cyc'') < \mu(\Cyc_0)$. This again contradicts the minimality of $\Cyc_0$. 
	
	By the induction hypothesis, each $j/\alpha_{\Susp}$ is contractible. Thus, we obtain a covering
	\[
	j/\alpha = \bigcup_{\varnothing \neq \Susp \subset \Cyc^{>1}_0} j/\alpha_{\Susp},
	\]
	where all pieces and their intersections are contractible (and hence nonempty). It follows that $j/\alpha$ itself is contractible.
\end{proof}
The following Proposition is the main part of the proof.
\begin{proposition}
\label{prp::p-ad-ext}	
	The inclusion $\Pro_2 \to \AD_2^{\mathrm{ext}}$ is an equivalence of graphical poset collections.
\end{proposition}
\begin{proof}
	As in the proof of Proposition~\ref{prop::ADext}, we show that the inclusion 
	\[
	\iota\colon \Pro_2(\Gr)\to \AD_2(\Gr)
	\] 
	satisfies the condition of Quillen's Theorem A by induction on the partial order of elements. Note that $\min \AD_2=\AD_1\subset \Pro_2$. If $\alpha \in \Pro_2(\Gr)$, $\iota/\alpha$ is immediately contractible, and hence it holds for all minimal elements of $\AD_2^{\mathrm{ext}}(\Gr)$. Thanks to the proof of Proposition~\ref{prop::ADext}, for an element $\alpha\in \AD^{\mathrm{ext}}_2$, we could find a collection of elements $\alpha_i$ from $\downarrow \alpha$ such that $\iota/\alpha$ admits a cover by $\iota/\alpha_i$ and by the induction assumption all their intersections are non-empty and contractible, hence $\iota/\alpha$ is contractible. So, it remains only to check that $\iota/\alpha$ is contractible for $\alpha$ in the complement $\AD_2(\Gr)\setminus \Pro_2(\Gr)$.
    
    Because $\alpha$ lies in the complement, it contains a \emph{broken circuit}, i.e., a sub-cycle $\Cyc\subset \Gr$ with $\min_{e\in\Cyc^+(\alpha)} w(e)\neq \min_{e\in\Cyc^-(\alpha)} w(e)$ in the notation of page~\pageref{page::description_of_proper}. Let $\Cyc_0$ be a minimal broken circuit with respect to $\mu$-lexicographical order and we choose the internal orientation of $\Cyc_0$ such that $\min_{e\in\Cyc^+(\alpha)} w(e)=2$ and $\min_{e\in\Cyc^-(\alpha)} w(e)=1$. In other words, all edges from $\Cyc^+(\alpha)$ are of weight $2$, and there exists an edge from $\Cyc^-(\alpha)$ of weight $1$.
    
    Let $\Cyc^{=2}_0\subset E_{\Cyc_0}$ be the subset of edges of weight $2$. For a signed set $\Susp\colon \Cyc^{=2}_0 \to \{+,-,0\}$, we define the weighted orientation $\alpha_{\Susp}$ obtained from $\alpha$ as follows: for an edge $e$ from the support $\underline{\Susp}=:\{e\in \Cyc^{=2}_0|\Susp(e)\neq 0\}$, we let it weight to be $1$ in $\alpha_{\Susp}$ and we reverse its orientation if $\Susp(e)=-$. We also consider a signed set $\Susp_{sing}$ which we shall call singular by the rule
    \[
    \Susp_{sing}(e)=\begin{cases}
    +,\text{ if }e\in \Cyc_0^{-}(\alpha)\cap\Cyc^{=2}_0
    \\
    -, \text{ if }e\in \Cyc_0^{+}(\alpha)
    \end{cases}
    \] Note that the set of signed sets with support in $\Cyc_0^{=2}$ except the singular signed set $\Susp_{sing}$ forms a simplicial complex $\mathcal{K}(\Cyc_0)$ on the vertex set $\Cyc_0^{=2}\times \{+,-\}$.
\begin{lemma} For $\Susp\neq \Susp_{sing}$, the direction $\alpha_{\Susp}$ is uniformal acyclic.
\end{lemma}
\begin{proof} Suppose the converse and we have an uniformal directed cycle $\Cyc_1$ in $\alpha_{\Susp}$. In particular, it implies that all edges from $\Cyc_1$ are of weight $1$ in $\alpha_{\Susp}$ and $\underline{\Susp}\cap\Cyc_1\neq \varnothing$. Without loss of generality we can assume $\underline{\Susp}\subset\Cyc_1$. Also, since $\Susp\neq \Susp_{sing}$ it implies that $\Cyc_0\neq  \Cyc_1$. Let us consider the internal orientation of $\Cyc_1$ inverse to that in $\alpha_{\Susp}$, i.e., such that $\Cyc_1^-(\alpha_{\Susp})=E_{\Cyc_1}$. In particular, we have $\Cyc^+_1(\alpha)=\Susp^-=\{e| \Susp(e)=-\}\subset \underline{\Susp}$ and hence $\Susp^-\neq \varnothing$. Since $\alpha$ has no cycles, it implies that some edges in $\alpha_{\Susp}$ are reversed and hence $\Susp^-\neq \varnothing$.

Suppose that the internal orientations of $\Cyc_0$ and $\Cyc_1$ coincide on $\underline{\Susp}$. In particular $\Cyc^+_0(g)\cap \Susp=\Cyc^+_1(\alpha)=\Susp^-$. If $\underline{\Susp}=\Cyc^{=2}_0$, then $\Susp^-=\Cyc^-_0(\alpha)\cap \Cyc^{=2}_0$ and hence $\Susp= \Susp_{singular}$, that contradicts the assumption. So, we have $\underline{\Susp}\neq \Cyc^{=2}_0$ and hence $\mu_2(\Cyc_1)<\mu_2(\Cyc_0)$ in $g$. Note that $\varnothing\neq \Cyc^+_1(\alpha)=\Susp^-\subset \Cyc^2_0$, therefore $\min_{e\in \Cyc^+_1(\alpha)} w(e)=2$. Since $\Cyc_0\neq \Cyc_1$, the complement $\Cyc^{-}_1(\alpha)\setminus \Susp$ is non-empty and all edges from it are of weight $1$ in $\alpha$, we have $\min_{e\in \Cyc^-_1(\alpha)} w(e)=1$. So, $\Cyc_1$ is a broken circuit and $\mu_2(\Cyc_1)<\mu_2(\Cyc_0)$ that contradicts $\mu$-minimality of $\Cyc_0$.

Therefore, there is an edge $e_0\in\underline{\Susp}$ that is oriented in different ways with respect to the internal orientations of $\Cyc_0$ and $\Cyc_1$. So, we could find a new cycle $\Cyc_2$ with $E_{\Cyc_2}\subset E_{\Cyc_1}\cup E_{\Cyc_0}\setminus \{e_0\}$ with the internal orientation such that $\Cyc_2^{\pm}(\alpha)\subset \Cyc^{\pm}_0(\alpha)\cup \Cyc_1^{\pm}(\alpha)$. On the one hand, since $\Cyc_0\neq \Cyc_1$, the complement $ E_{\Cyc_2}\setminus E_{\Cyc_0}\subset$ is non-empty and is contained in $\Cyc_1^-(\alpha)$ ( and hence in $\Cyc_2^-(\alpha)$) and all edges of this complement are of weight $1$, we have $\min_{e\in \Cyc_2^{-}(\alpha)} w(e)=1$. On the other hand, we have $\min_{e\in \Cyc_2^{-}(\alpha)} w(e)=2$. Indeed, if there is $e\in \Cyc^+_2(\alpha)$ with $w(e)=1$, then $e\not\in \Susp$ and therefore $e\not\in \Susp^-=\Cyc_1^+(\alpha)$, so $e\in \Cyc^+_0(\alpha)\subset E_{\Cyc_0}^{=2}$, which contradicts $w(e)=1$. So, $\Cyc_2$ is a broken circuit with respect to $\alpha$ with $\mu_2(\Cyc_2)\leq \mu_2(\Cyc_0)-1$, that contradicts $\mu$-minimality of $\Cyc_0$

\end{proof}
Note that we have a cover of $\iota/\alpha$
\[
\iota/\alpha=\cup_{s\in \Cyc^{=2}_0\times \{+,-\}} \iota/\alpha_{s},
\] where $\alpha_{(e,\pm)}$ is obtained from $g$ be decreasing the weight of $e$ and reversing its orientation depending on a sign. Moreover for $\Susp\subset \Cyc^{=2}_0\times \{+,-\}$, the intersection $\cap_{s\in\Susp} \iota/\alpha_s$ is non-empty iff $\Susp\in \mathcal{K}(\Cyc_0)$ and in this case $\cap_{s\in\Susp} \iota/\alpha_s=\iota/\alpha_{\Susp}$. Since $\alpha_{\Susp}<\alpha$ for all non-empty $\Susp\in \mathcal{K}(\Cyc_0)$, the induction hypothesis implies that each $\iota/\alpha_{\Susp}$ is contractible. Hence $\iota/\alpha$ is covered by contractible subspaces whose non-empty intersections are also contractible, so by the Alexandrov Nerve Theorem, $\iota/\alpha$ has the homotopy type of $\mathcal{K}(\Cyc_0)$.

Note that the union $\mathcal{K}(\Cyc_0)\cup_{\partial\triangle_{sing}} \triangle_{sing}$, where $\triangle_{sing}$ is the simplex corresponding to the singular signed set $\Susp_{sing}$, is homeomorphic to the boundary of the $(k-1)$-dimensional cross polytope, where $k=|\Cyc^{=2}_0|$. Since cross polytopes are convex, $\mathcal{K}(\Cyc_0)$ is homeomorphic to the complement of an open disk in the $(k-1)$-sphere, hence contractible. So, $\iota/\alpha$ is contractible.
\end{proof}

\begin{corollary}
\label{cor::p-ad}	
	The inclusion $\iota\colon \Pro_2 \to \AD_2$ is an equivalence of graphical poset collections. Hence, $|\AD_2|$ is $E_2$-contractad.
\end{corollary}
\begin{proof}
	By Proposition~\ref{prop::ADext}, the inclusion $j\colon \AD_2 \to \AD_2^{\mathrm{ext}}$ is a homotopy equivalence. Proposition~\ref{prp::p-ad-ext} further shows that the composition $j\circ \iota\colon \Pro_2 \to \AD_2^{\mathrm{ext}}$ is also a homotopy equivalence. The claim then follows from the $2$-out-of-$3$ property for homotopy equivalences. Thanks to Proposition~\ref{prop::conf_to_P}, the map $\psi_2\colon\Conf(\mathbb{R}^2)\to \AD_2$ induces a homotopy equivalence $\Conf(\mathbb{R}^2)\simeq |\AD_2|$ and hence, by Theorem~\ref{thm::implies_En}, we have an equivalence $|\AD_2|\simeq \D_2$.
\end{proof}

\normalem

\bibliographystyle{alpha}
\bibliography{biblio.bib}
\end{document}